\documentclass[11pt]{article}
\usepackage[margin=1in]{geometry}

\usepackage{amsmath,amsfonts,amsthm,amssymb,latexsym,amscd,amsopn,eucal,mathrsfs,graphics,color,enumerate,lscape}
\usepackage{afterpage}
\usepackage{chngpage}
\usepackage[small]{titlesec}
\usepackage[colorlinks]{hyperref}

\theoremstyle{plain}
  \newtheorem{thm}{Theorem}[section]
  \newtheorem{prop}[thm]{Proposition}
  \newtheorem{cor}[thm]{Corollary}
  \newtheorem{lemma}[thm]{Lemma}
  \newtheorem{claim}[thm]{Claim}
\theoremstyle{definition}
  \newtheorem{defn}[thm]{Definition}
  \newtheorem{example}[thm]{Example}
  
  \newtheorem{remark}[thm]{Remark}
  \newtheorem{construction}[thm]{Construction}
\theoremstyle{remark}

\input xy
\xyoption{all}

\DeclareMathOperator{\Jac}{Jac}
\DeclareMathOperator{\Pic}{Pic}
\DeclareMathOperator{\Hom}{Hom}

\DeclareMathOperator{\Aut}{Aut}

\DeclareMathOperator{\Sym}{Sym}
\DeclareMathOperator{\Herm}{Herm}
\DeclareMathOperator{\disc}{disc}
\renewcommand{\Im}{\mathrm{Im}\;}
\DeclareMathOperator{\Spec}{Spec\;}
\DeclareMathOperator{\Gr}{Gr}
\DeclareMathOperator{\Mat}{Mat}
\DeclareMathOperator{\Tr}{Tr}
\DeclareMathOperator{\tr}{tr}
\DeclareMathOperator{\Norm}{N}
\DeclareMathOperator{\norm}{n}
\DeclareMathOperator{\Spur}{S}
\DeclareMathOperator{\Spin}{Spin}

\DeclareMathOperator{\WC}{WC}
\DeclareMathOperator{\Br}{Br}
\DeclareMathOperator{\Ob}{Ob}
\DeclareMathOperator{\Seg}{Seg}
\DeclareMathOperator{\lin}{lin}

\newcommand{\PGL}{\mathrm{PGL}}
\newcommand{\SL}{\mathrm{SL}}
\newcommand{\GL}{\mathrm{GL}}

\newcommand{\Sp}{\mathrm{Sp}}

\newcommand{\Q}{\mathbb{Q}}
\newcommand{\Z}{\mathbb{Z}}
\newcommand{\G}{\mathbb{G}}
\newcommand{\C}{\mathbb{C}}

\newcommand{\Gm}{\mathbb{G}_m}

\newcommand{\Euniv}{E^{\mathrm{univ}}}
\newcommand{\univ}{\mathrm{univ}}
\newcommand{\Moneone}{\mathscr{M}_{1,1}}

\newcommand{\Hilb}{\mathcal{H}\mathrm{ilb}}

\def\id{{\rm id}}

\renewcommand{\AA}{\mathcal{A}}

\newcommand{\CC}{\mathbb{C}}
\newcommand{\mC}{\mathcal{C}}
\renewcommand{\H}{\mathcal{H}}
\newcommand{\HH}{\mathscr{H}}

\newcommand{\OO}{\mathcal{O}}
\newcommand{\PP}{\mathbb{P}}

\newcommand{\VV}{\mathcal{V}}

\newcommand{\ZZ}{\mathbb{Z}}
\newcommand{\Pone}{\mathbb{P}^1}
\newcommand{\Poneone}{\Pone \times \Pone}

\newcommand{\tns}{\otimes}
\newcommand{\ra}{\longrightarrow}
\newcommand{\subs}{\lrcorner \,}

\newcommand{\Kbar}{{\overline{K}}}
\newcommand{\JA}{{J(A)}}
\newcommand{\CJ}{{\mathscr{C}(J)}}

\def\ie{{\textrm i.e., }}
\def\eg{{\textrm e.g., }}

\def\cube#1#2#3#4#5#6#7#8{{ 
  \setlength{\unitlength}{.155in}
  \begin{picture}(8,8)
  \put(1,5.5){\makebox(0,0){$#1$}}
  \put(6,5.5){\makebox(0,0){$#2$}}
  \put(1,0.5){\makebox(0,0){$#3$}}
  \put(6,0.5){\makebox(0,0){$#4$}}
  \put(3,7.5){\makebox(0,0){$#5$}}
  \put(8,7.5){\makebox(0,0){$#6$}}
  \put(3,2.5){\makebox(0,0){$#7$}}
  \put(8,2.5){\makebox(0,0){$#8$}}
  \put(1,1.2){\line(0,1){3.6}}
  \put(6,1.2){\line(0,1){3.6}}
  \put(8,3.2){\line(0,1){3.6}}
  \put(3,3.2){\line(0,1){1.8}}
  \put(3,6){\line(0,1){0.8}}
  \put(1.7,0.5){\line(1,0){3.6}}
  \put(1.7,5.5){\line(1,0){3.6}}
  \put(3.7,7.5){\line(1,0){3.6}}
  \put(3.7,2.5){\line(1,0){1.8}}
  \put(6.5,2.5){\line(1,0){0.8}}
  \put(1.5,1){\line(1,1){1.0}}
  \put(1.5,6){\line(1,1){1.0}}
  \put(6.5,1){\line(1,1){1.0}}
  \put(6.5,6){\line(1,1){1.0}}
 \end{picture} }}

\def\cubea#1#2#3#4#5#6#7#8{{ 
  \setlength{\unitlength}{.155in}
  \begin{picture}(8,8)
  \put(1,6){\makebox(0,0){$#1$}}
  \put(6,6){\makebox(0,0){$#2$}}
  \put(1,1){\makebox(0,0){$#3$}}
  \put(6,1){\makebox(0,0){$#4$}}
  \put(3,8){\makebox(0,0){$#5$}}
  \put(8,8){\makebox(0,0){$#6$}}
  \put(3,3){\makebox(0,0){$#7$}}
  \put(8,3){\makebox(0,0){$#8$}}
  \put(1,1.7){\line(0,1){3.6}}
  \put(6,1.7){\line(0,1){3.6}}
  \put(8,3.7){\line(0,1){3.6}}
  \put(3,3.7){\line(0,1){1.8}}
  \put(3,6.5){\line(0,1){0.8}}
  \put(2,1){\line(1,0){2.8}}
  \put(2,6){\line(1,0){2.8}}
  \put(4,8){\line(1,0){2.8}}
  \put(4,3){\line(1,0){1.7}}
  \put(6.2,3){\line(1,0){0.7}}
  \put(1.5,1.5){\line(1,1){0.9}}
  \put(1.5,6.5){\line(1,1){0.9}}
  \put(6.5,1.5){\line(1,1){0.9}}
  \put(6.5,6.5){\line(1,1){0.9}}
 \end{picture} }}

\title{Coregular spaces and genus one curves}
\author{\!\!Manjul Bhargava\thanks{Department of Mathematics, Princeton University, Princeton, NJ 08544. Supported by a Packard Foundation Fellowship and NSF grant DMS-1001828.} \ \, and \ Wei Ho\thanks{Department of Mathematics, Columbia University, New York, NY 10027. Partially supported by NSF grants DMS-0902853 and DMS-0739400.}}
\date{}
\begin{document}

\maketitle

{\small \tableofcontents}

\section{Introduction} \label{sec:intro}

In 1963, Birch and Swinnerton-Dyer \cite{BSD} carried out a seminal study of the moduli space of genus one curves equipped with degree 2 line bundles, in terms of the orbits of the action of $\GL_2$ on the space $\Sym^4(2)$ of binary quartic forms.  This orbit space parametrization was a key ingredient in the explicit 2-descent computations that led them to the celebrated Birch and Swinnerton-Dyer conjecture.  The analogues of this parametrization for line bundles of degree 3, 4, and 5 (i.e., ``elliptic normal curves'' of degrees $3$, $4$, and $5$) were subsequently investigated in the important works of Cassels \cite{Cassels}, and more recently, Cremona \cite{cremona-binarycubicquartic}, Cremona--Fisher--Stoll \cite{cremonafisherstoll}, and Fisher \cite{fisher-pfaffianECs}.  These works have, in particular, enabled explicit 3-, 4-, and 5-descent computations on elliptic curves analogous to the original 2-descent computations of Birch and Swinnerton-Dyer. 
Recently, these parametrizations of elliptic normal curves have also been used to obtain bounds on the average rank and Selmer ranks of elliptic curves (see \cite{arulmanjul-bqcount,arulmanjul-tccount}).

The important consequences and elegance of these classical orbit parametrizations naturally raise the question as to whether further such correspondences exist that could shed light on other data attached to genus one curves.  The purpose of this article is to develop additional such correspondences.  In fact, we will show that the classical representations described above for elliptic normal curves are only four among at least $20$ such representations whose orbits parametrize nontrivial data on genus one curves---such as line bundles, vector bundles, points on the Jacobian, as well as more exotic structures.

The underlying philosophy is the use of orbit spaces to parametrize algebraic or geometric objects. In 1801, Gauss gave perhaps the first nontrivial example of such a parametrization in his celebrated Disquisitiones Arithmeticae \cite{Gauss}, where he studied integral binary quadratic forms under a certain action of the group $\GL_2(\Z)$. Although the space of binary quadratic forms is a {\it prehomogeneous vector space}, meaning that it has only one open orbit over $\C$, the rational and especially the integral orbits are in bijection with quite nontrivial arithmetic objects, namely, quadratic fields and  ideal classes in quadratic rings, respectively.  

In \cite{wrightyukie}, Wright and Yukie showed that orbits of many prehomogeneous vector spaces over a field $k$ correspond to field extensions of $k$.
The series of papers \cite{hcl1,hcl2,hcl3,hcl4,hcl5} describes how the integral orbits of most  prehomogeneous vector spaces parametrize arithmetic objects, such as rings of low rank together with ideals and modules.  These parametrizations were used in \cite{manjulcountquartic,manjulcountquintic}, for example, to determine the density of discriminants of quartic and quintic fields, thus completing the original program of Wright and Yukie (see~\cite[\S1]{wrightyukie}) of using prehomogeneous representations to determine densities of arithmetic objects. 

\nopagebreak

In this paper, we study a natural series of {\it coregular} representations, that is, representations for which the ring of invariants is a polynomial ring.  
More precisely, we consider here many representations of reductive groups $G$ for which the restricted representation on the semisimple part of $G$ is coregular.
It is interesting that ``most'' such representations that have more than one generating invariant---i.e., are not prehomogeneous---turn out to involve {\it genus one curves}.

Just as the space of $2\times 2\times 2$ cubes and $2\times 3\times 3$ boxes played a central role in the study of prehomogeneous vector spaces \cite{hcl1,hcl2}, here the spaces of $2\times 2\times 2\times 2$ hypercubes and $3\times 3\times 3$ cubes play a central role in the theory, from which we are then able to derive most other coregular spaces corresponding to genus one curves via suitable invariant-theoretic procedures.  Also, analogous to the prehomogeneous cases, the invariant theory of our spaces plays a crucial role in constructing and describing the corresponding geometric data.  Indeed, in  many cases, our bijections yield natural geometric interpretations for the generators of the invariant ring.


\afterpage{%
\begin{landscape}

\begin{table} \label{table:examples}
\begin{center}
	\begin{tabular*}{1.1\textwidth}{@{\extracolsep{\fill}}r|c|c|l|c|c|c|}
	\cline{2-7}
		& Group (s.s.) & Representation & Geometric Data & Invariants & Dynkin & \S \\
	\cline{2-7}
		1. &		$\SL_2$ & $\Sym^4 (2)$ & $(C,L_2)$ & $2, 3$ & $A_2^{(2)}$ &\ref{sec:binaryquartics} \\
		2. &		$\SL_2^2$ & $\Sym^2 (2) \tns \Sym^2 (2)$ & $(C, L_2, L_2') \sim (C, L_2, P)$ & $2, 3, 4$ & $D_3^{(2)}$&\ref{sec:bideg22forms}  \\
		3. &		$\SL_2^4$ & $2 \tns 2 \tns 2 \tns 2$ & $(C, L_2, L_2', L_2'')\sim(C,L_2,P,P')$ & $2, 4, 4, 6$ & $D_4^{(1)}$& \ref{sec:hypercube} \\
		4. &		$\SL_2^3$ & $2 \tns 2 \tns \Sym^2(2)$ & $(C, L_2, L_2') \sim (C, L_2, P)$ & $2, 4, 6$ & $B_3^{(1)}$ &\ref{sec:2symHC} \\ 
		5. &		$\SL_2^2$ & $\Sym^2 (2) \tns \Sym^2 (2)$ & $(C,L_2, L_2') \sim (C, L_2, P)$ & $2, 3, 4$ & $D_3^{(2)}$ & \ref{sec:22symHC} \\
		6. &		$\SL_2^2$ & $2 \tns \Sym^3(2)$ & $(C,L_2, P_3)$ & $2, 6$ & $G_2^{(1)}$ &\ref{sec:3symHC} \\
		7. &		$\SL_2$ & $\Sym^4 (2)$ & $(C, L_2, P_3)$ & $2, 3$ & $A_2^{(2)}$& \ref{sec:4symHC} \\
		8. &		$\SL_2^2 \times \GL_4$ & $ 2 \tns 2 \tns \wedge^2(4)$ & $(C, L_2, M_{2,6})$ & $2, 4, 6, 8$ & $D_5^{(1)}$&\ref{sec:2skewHC} \\
		9. &		$\SL_2 \times \SL_6$ & $ 2 \tns \wedge^3(6)$ & $(C, L_2, M_{3,6})$ with $L^{\tns 3} \cong \det M$
			& $2, 6, 8, 12$ & $E_6^{(1)}$& \ref{sec:3skewHC}\\
		10. &		$\SL_2 \times \Sp_6$ & $2 \tns \wedge^3_0(6)$ & $(C,L_2, (M_{3,6}, \varphi))$ with $L^{\tns 3} \cong \det M$ & $2, 6, 8, 12$ & $E_6^{(2)}$ & \ref{sec:excHC}  \\
		11. &		$\SL_2 \times \Spin_{12}$ & $2 \tns S^+(32)$ & $(C \to \PP^1(\HH_3(\mathbb{H})), L_2)$
			 & $2, 6, 8, 12$ & $E_7^{(1)}$&  \ref{sec:excHC} \\
		12. &		$\SL_2 \times E_7$ & $2 \tns 56$ & $(C \to \PP^1(\HH_3(\mathbb{O})), L_2)$ 
			& $2, 6, 8, 12$ & $E_8^{(1)}$ & \ref{sec:excHC} \\
	\cline{2-7}
		13. &		$\SL_3$ & $\Sym^3 (3)$ & $(C,L_3)$ & $4, 6$ & $D_4^{(3)}$ & \ref{sec:ternarycubics} \\
		14. &		$\SL_3^3$ & $3 \tns 3 \tns 3$ & $(C,L_3,L_3') \sim (C,L_3,P)$ & $6, 9, 12$ & $E_6^{(1)}$ &  \ref{sec:333} \\
		15. &		$\SL_3^2$ & $3 \tns \Sym^2 (3)$ & $(C,L_3,P_2)$ & $6, 12$ & $F_4^{(1)}$& \ref{sec:2symRC} \\
		16. &		$\SL_3$ & $\Sym^3 (3)$ & $(C,L_3,P_2)$ & $4, 6$ & $D_4^{(3)}$&\ref{sec:3symRC} \\
		17. &		$\SL_3 \times \SL_6$ & $3 \tns \wedge^2 (6)$ & $(C,L_3,M_{2,6})$ with $L^{\tns 2} \cong \det M$
			 & $6, 12, 18$ & $E_7^{(1)}$ & \ref{sec:deg3special} \\
		18. &		$\SL_3 \times E_6$ & $3 \tns 27$ & $(C \hookrightarrow \PP^2(\mathbb{O}),L_3)$	& $6, 12, 18$ & $E_8^{(1)}$& \ref{sec:deg3moduli}  \\
	\cline{2-7}
		19. &		$\SL_2 \times \SL_4$ & $2 \tns \Sym^2 (4)$ & $(C,L_4)$ & $8, 12$ & $E_6^{(2)}$&\ref{sec:deg4} \\
	\cline{2-7}
		20. &		$\SL_5 \times \SL_5 $ & $\wedge^2(5) \tns 5$ & $(C,L_5)$ & $20, 30$ & $E_8^{(1)}$ & \ref{sec:deg5} \\
	\cline{2-7}
	\end{tabular*}
\end{center}
\caption{Table of coregular representations and their moduli interpretations}
\end{table}

\end{landscape}
}

\vspace{.1in}

A summary of the parametrizations we obtain is provided in Table \ref{table:examples}.
In this table, the first and second columns list the representations in question, although we only list the semisimple
parts of the groups here, since some of the actions of the non-semisimple parts of the relevant groups are not completely standard.
The third column lists the geometric data (up to isomorphism) arising from a general orbit of this representation: the data in every case includes a genus one curve $C$.  The curve may also be equipped with line bundles, denoted by $L_d$, $L_d'$, $L_d''$, etc., where $d$ is the degree of the line bundle, or with a vector bundle, denoted by $M_{r,d}$, where $r$ is the rank and $d$ is the degree of the vector bundle.  The notation $P$ or $P'$ indicates a rational point on the Jacobian of $C$ (often in a certain arithmetic subgroup of $\Jac(C)$), and $P_e$ indicates that the point is a nontrivial rational torsion point of order $e$.  The notation $\sim$ indicates that the data on the two sides are equivalent and are both suitable interpretations for the moduli problem.  There may be some additional mild (open) conditions on the geometric data not indicated in column three.
The fourth column gives the degrees of the invariants of the representation of the semisimple group, and the fifth contains
the extended or affine Dynkin diagram associated with the representation (see Section \ref{sec:ExcLieAlgs}).
Finally, the sixth column indicates the subsection in which that case is proved
and/or discussed most carefully, although most of the theorems are previewed in Sections~\ref{sec:HCpreview} and \ref{sec:RCpreview}.  Note that in many cases, changing the form of the group over $K$ leads to a twisted version of the geometric data in column three.

For example, line 3 of Table~\ref{table:examples} corresponds to the case of $2\times 2\times 2\times 2$ hypercubical matrices over a field $K$ ($\mathrm{char}(K) \neq 2,3$). We show that the nondegenerate $\GL_2(K)^4$-orbits of $K^2\otimes K^2 \otimes K^2\otimes K^2$ correspond to the data $(C,L,(P,P',P''))$, where: $C$ is a genus one curve over $K$;\, $L$ is a degree 2 line bundle on $C$;\, and $P$, $P'$, and $P''$ are non-identity $K$-points summing to zero in a certain arithmetic subgroup of the Jacobian of $C$.  Meanwhile, the ring of polynomial invariants for the action of $\SL_2(K)^4$ on $K^2\otimes K^2 \otimes K^2\otimes K^2$ is freely generated by four invariants $a_2$, $a_4$, $a_4'$, and $a_6$, having degrees 2, 4, 4, and 6, respectively.  In terms of the geometric data, if we write the Jacobian of $C$ as a Weierstrass elliptic curve $y^2=x^3+Ax+B$ on which the points $P$, $P'$, $P''$ lie, then: $a_2$ can be interpreted as the slope of the line connecting $P$, $P'$ and $P''$;\, $a_4$ and $a_4'$ are the $x$-coordinates of the points $P$ and $P'$;\, and $a_6$ is the $y$-coordinate of~$P$.  

Similarly, line fourteen of Table~\ref{table:examples} corresponds to the case of $\GL_3(K)^3$ acting on the space $K^3\otimes K^3\otimes K^3$ of $3\times 3\times3$ cubical matrices over $K$.  We prove that the nondegenerate orbits parametrize data consisting of a triple $(C,L,(P,P'))$, where $C$ is a genus one curve over $K$, $L$ is a degree 3 line bundle on $C$ defined over $K$, and $P$ and $P'$ are non-identity points summing to zero in an arithmetic subgroup of the Jacobian of $C$.  The three generators $b_6$, $b_9$, and $b_{12}$ of the $\SL_3(K)^3$-invariant ring have degrees 6, 9, and 12, respectively.  If we again write the Jacobian of $C$ as a Weierstrass elliptic curve $y^2=x^3+Ax+B$, then $P=(b_6,b_9)$, $P'=(b_6,-b_9)$, and~$A=b_{12}$.

\vspace{.1in}

We briefly describe some forthcoming applications of these parametrizations.  In \cite{arulmanjul-bqcount,arulmanjul-tccount}, an implementation of certain geometric counting techniques (building on those in~\cite{manjulcountquintic})
for integral orbits of the representations in lines 1 and 13
of Table~\ref{table:examples} 
has led to  results on the average sizes of the 2- and 3-Selmer groups of the family of all elliptic curves over $\Q$ (when ordered by height), and corresponding (finite) upper bounds on the average rank of all elliptic curves. 
By developing these counting techniques further, so that they may be applied to other cases in Table~\ref{table:examples},  we determine in~\cite{cofreecounting} the average sizes of the 2- and 3-Selmer groups for various families of elliptic curves, e.g., those with marked points.  These results lead to corresponding average rank bounds for the curves in these families.  For example, 
the space of $3\times 3\times 3$ cubes allows us to show that the average size of the 3-Selmer group in the family of elliptic curves 
	\begin{equation*}
	y^2 +a_3 y = x^3 + a_2 x^2 + a_4 x
	\end{equation*}
having a marked point at $(0,0)$ is
12.  As a result, we show that the (limsup of the) average rank of this family of elliptic curves is finite (indeed, at most $2\frac16$), and that a positive proportion of curves in this family have rank one. Analogous results for average sizes of Selmer groups in families of elliptic curves with one marked point of order 3 or 2 (using lines 6 and 15, respectively, of Table~\ref{table:examples}) and elliptic curves with two general marked points (using line 3, the space of hypercubes) are also obtained.

\vspace{.1in}

\noindent {\em Outline.}
The organization of this paper is as follows.  Sections \ref{sec:HCpreview} and \ref{sec:RCpreview} form an extended introduction in which we describe the basic constructions and parametrizations corresponding to $2\times 2\times 2\times 2$ hypercubes and $3\times 3\times 3$ cubes, respectively, and how many of the various other coregular space orbit parametrizations in Table \ref{table:examples} may be (at least heuristically) derived from them.

Section \ref{sec:classical} describes orbit parametrizations for the moduli spaces of genus one curves with degree~$n$ line bundles, for~$2 \leq n \leq 5$.  Many of the ideas in this section are classical or well known, at least over algebraically closed fields, but our constructions generalize to other fields and more general base schemes.  We also show that the stabilizers of elements in these representations are naturally isomorphic to the automorphism group of a genus one curve with a degree $n$ line bundle, which is related to the so-called Heisenberg theta group.  These results about stabilizers play a central role in the works~\cite{arulmanjul-bqcount,arulmanjul-tccount}.  The parametrizations of elliptic normal curves of degree~$n$, especially for $n = 2$ and $n = 3$, will also be used extensively in the later parts of the paper.

In Section \ref{sec:hermRC}, we concentrate on the coregular spaces whose orbits are related
to a genus one curve and a degree~$3$ line bundle, possibly with additional data. 
We first discuss some of the fundamental cases, such as the aforementioned space of 
$3 \times 3 \times 3$ cubical matrices.
In each case, we show that the invariants of the representation are closely related to the geometric data, and in particular,
to the Jacobian of the genus one curve that arises.

We then study spaces of the form $V \tns J$, where $V$ is a $3$-dimensional vector space
and $J$ is a certain type of cubic Jordan algebra (to be specified), up to
the natural action of a group which we denote  by $\GL(V)\times \SL(J)$.  In each case, the group 
$(\mathbb G_m\times)\SL(J)$ acting on $J$ is a prehomogeneous vector space, equipped with 
a relative invariant cubic norm form and an
adjoint map.  We prove a uniform theorem about the orbits of these type of representations, and
then specialize to specific $J$ to recover a number of the introductory
cases as well as other interesting moduli problems.

Section \ref{sec:hermHC} develops analogous ideas to study orbits that parametrize genus one
curves with degree $2$ line bundles and additional data.  We again begin by discussing the
most fundamental representations,
including the space of bidegree $(2,2)$ forms on $\Pone \times \Pone$
and the aforementioned space of $2 \times 2 \times 2 \times 2$ hypercubical matrices.
We show that the invariants of 
each representation are again closely related to the corresponding geometric data.  
Analogously to the case of degree~$3$ line bundles, we then study a more general situation.
In particular, we consider the tensor product of a two-dimensional vector space $V$
and a space $\mathscr{C}(J)$ of ``Hermitian cubes'' with respect
to a cubic Jordan algebra $J$, under the action of
a group which we denote by $\SL_2(V)\times \SL_2(J)$; the space $\mathscr{C}(J)$ has a quartic
norm form and a natural adjoint map.  Again, we study representations of this kind uniformly, and 
then specializing recovers most of the earlier cases as well as several new moduli problems.

In Section \ref{sec:ExcLieAlgs}, we describe how all of the
representations we study are related to exceptional Lie algebras.
In particular, these representations all arise from Vinberg's theory of $\theta$-groups \cite{vinberg};
his constructions give a wide class of coregular spaces. 
In his recent Harvard Ph.D.\ thesis, J.\ Thorne~\cite{jthorne-thesis} studies some canonical constructions in invariant
theory arising from Vinberg theory, and it is thus an interesting question as to how his more representation-theoretic constructions are related to our geometric ones.  Finally, we also describe how our spaces are closely related to the representations found in the Deligne-Gross Magic Triangle \cite{delignegross}.  

The ``certain arithmetic subgroup'' of the Jacobian of the genus one curve arising in many of our moduli problems is called the {\it period-index subgroup}.  It is equivalent to the entire Jacobian when working over an algebraically closed field.  When working over number fields, it is a finite-index subgroup of the Jacobian, consisting of points that are ``unobstructed'' in relation to the genus one curve.  This is described in more detail in Appendix~\ref{appendix:torsors}, which may be of use to those interested in the arithmetic applications.   The appendix may also be safely skipped for those readers more interested in the geometric constructions and bijections over an algebraically closed field.

\vspace{.1in}

\noindent {\em Acknowledgments.}  We would like to thank Bhargav Bhatt, John Cremona, Benedict Gross, Joe Harris, Abhinav Kumar, and Catherine O'Neil for useful conversations.
The main theorems in \S \ref{sec:g1fromRC} and \S \ref{sec:333} are joint work with Catherine O'Neil.  


\section{Main results I: Genus one curves and \texorpdfstring{$2\times 2\times 2\times 2$}{2222} hypercubes}  \label{sec:HCpreview}

In this section, we discuss the space of $2\times 2\times 2\times 2$ hypercubical matrices over $K$,
and we describe the various parametrizations of genus one curves with extra data that may be obtained
from this perspective.  No proofs or details are given in this section.  
Here, we simply describe in an elementary way the constructions
of the genus one curves and extra data from the 
orbits of our representations, and state the main theorems related to these cases.
Further details, more functorial descriptions of the constructions, and proofs may be found in Section~\ref{sec:hermHC}.

\subsection{Preliminaries on \texorpdfstring{$2\times 2\times 2$}{222} cubes}\label{sec:cubes}

Before studying $2\times 2\times 2\times 2$ hypercubical matrices, we first review the case of
$2\times 2\times 2$ cubical matrices (see \cite{hcl1} for more details).

Let $K$ be a field with $\mathrm{char}(K)\neq 2$.  Let $\mC_2(K)$ denote the space $K^2\otimes K^2 \otimes K^2$.
Then each element of $\mC_2(K)$ can naturally be represented as a cubical matrix $A = (a_{ijk})$ with entries in $K$, where $i,j,k\in\{1,2\}$:
\vspace{.1in} 
\begin{equation}\label{eq:firstcube}
\raisebox{-2\baselineskip}{
\cube {a_{111}} {a_{112}} {a_{121}} {a_{122}} {a_{211}} {a_{212}} {a_{221}} {a_{222}}
}
\qquad .
\end{equation}
If we denote by $\{e_1,e_2\}$ the standard basis of $K^2$, then the element of $\mC_2(K)$ described by \eqref{eq:firstcube} is
\[\sum_{i,j,k} a_{ijk}\, e_i\otimes e_j\otimes e_k.\]
As the cubical matrix representation is both more intuitive and more convenient, we shall  
identify $\mC_2(K)$ with the space of $2\times 2\times 2$ matrices with entries in $K$, or the space of {\it cubes over $K$}.

Now a cube $A$ over $K$ may be naturally sliced into two $2\times 2$ matrices over $K$ in three different ways.
Namely, the cube $A=(a_{ijk})$ given by \eqref{eq:firstcube}
may be partitioned into the two $2\times 2$ matrices $M_\ell$ and $N_{\ell}$, for $\ell=1, 2, 3$, as follows:
\begin{itemize}
\item[{\rm 1)}]
$M_1=(a_{1jk})$ is the front face and $N_1=(a_{2jk})$ is the back face of $A$; 
\item[{\rm 2)}]
$M_2=(a_{i1k})$ is the top face and $N_2=(a_{i2k})$ is the bottom face of $A$; 
\item[{\rm 3)}]
$M_3=(a_{ij1})$ is the left face and $N_3=(a_{ij2})$ is the right face of $A$.
\end{itemize}
We may define a natural action of $\SL_2(K)^3$ on $\mC_2(K)$ so that, for any $\ell \in\{1,2,3\}$, the
element ${\left(\begin{smallmatrix}r&s\\t&u\end{smallmatrix}\right)}$
in the $\ell$th factor of $\SL_2(K)^3$ acts on the cube $A$ by replacing
$(M_\ell,N_\ell)$ by $(rM_\ell+sN_\ell,tM_\ell+uN_\ell)$.  The actions of these three
factors of $\SL_2(K)$ in $\SL_2(K)^3$ commute with each other, analogous to the fact that row and column operations on a rectangular matrix commute. Hence we obtain a natural and well-defined action of 
$\SL_2(K)^3$ on $\mC_2(K)$.

This action turns out to have a single polynomial invariant%
\footnote{We use throughout the standard abuse of terminology ``has a single
  polynomial invariant'' (or ``has $k$ polynomial invariants'') to
  mean that the corresponding polynomial invariant ring is generated
  freely by one element (respectively, $k$ elements).}%
, which we call the {\em discriminant}.
Given a $2\times 2\times 2$ cube $A$ over $K$, for each
$\ell \in\{1,2,3\}$, we obtain a binary quadratic
form 
\begin{equation}\label{bqfdet}
Q_\ell(x,y)=\det(M_\ell x + N_\ell y).
\end{equation}
The discriminants of these three
binary quadratic forms all coincide (see \cite[\S 2]{hcl1}), and this
gives the desired  invariant, called the
{discriminant} $\disc(A)$ of the cube~$A$.  (These triples of binary
quadratic forms with the same discriminant arising from cubes were
used to give a simple description of Gauss composition in
\cite{hcl1}.)  This fundamental invariant of degree four on the space \linebreak
$K^2\otimes K^2\otimes K^2$ of cubical matrices over $K$ will play a
key role in understanding the next space $K^2\otimes K^2\otimes
K^2\otimes K^2$ of hypercubical matrices over $K$.

\subsection{On \texorpdfstring{$2\times 2\times 2\times 2$}{2222} hypercubes}\label{sec:HCslicing}

Assuming now that $K$ has characteristic not $2$ or $3$, let $\H_2(K)$ denote the space $K^2\otimes K^2 \otimes K^2\otimes K^2$.  Then we may identify $\H_2(K)$ with the space of $2\times 2 \times 2\times 2$ hypercubical matrices $H=(h_{ijk\ell})$ over $K$, which we will call the space of {\it hypercubes over $K$}. Such hypercubes are somewhat harder to draw on paper; breaking symmetry, we write our hypercube $H=(h_{ijk\ell})$ thus:
\begin{equation}\label{eq:hyperdraw}
\raisebox{-2\baselineskip}{
\cubea {h_{1111}} {h_{1112}} {h_{1121}} {h_{1122}} {h_{1211}} {h_{1212}} {h_{1221}} {h_{1222}} 
\qquad \qquad \qquad \cubea
{h_{2111}} {h_{2112}} {h_{2121}} {h_{2122}} {h_{2211}} {h_{2212}} {h_{2221}} {h_{2222}} 
} \qquad .
\end{equation}

Now just as a cube $A$ over $K$ could be partitioned into two $2\times 2$ matrices in three different ways, a hypercube $H$ over $K$ may be partitioned into two $2\times 2\times 2$ matrices in four different ways.
More precisely, the hypercube $H=(h_{ijk\ell})$ can be partitioned into two cubes $A_m$ and $B_m$, for $m\in\{1,2,3,4\}$, as follows: 
\begin{itemize}
  \item[1)] $A_1=(h_{1jk\ell})$ and $B_1=(h_{2jk\ell})$; 
  \item[2)] $A_2=(h_{i1k\ell})$ and $B_2=(h_{i2k\ell})$; 
  \item[3)] $A_3=(h_{ij1\ell})$ and $B_3=(h_{ij2\ell})$; 
  \item[4)] $A_4=(h_{ijk1})$ and $B_4=(h_{ijk2})$,
\end{itemize}
where the first slicing 1) is depicted in \eqref{eq:hyperdraw}.

We define a natural action of $\SL_2(K)^4$ on the space of hypercubes so that, for any $i\in\{1,2,3,4\}$,  an element ${\left(\begin{smallmatrix}r&s\\t&u\end{smallmatrix}\right)}$ in the $i$th factor of $\SL_2(K)$ acts on the hypercube $H$ by replacing
$(A_i,B_i)$ by $(rA_i+sB_i,tA_i+uB_i)$.  The actions of these four
factors of $\SL_2(K)$ in $\SL_2(K)^4$ again commute with each other, so we obtain a well-defined action of 
$\SL_2(K)^4$ on $\H_2(K)$.

Now recall that a cube over $K$ naturally yields three binary quadratic forms over $K$, through its slicings into pairs of $2\times 2$ matrices over $K$.  Namely, for each slicing of a cube $A$ into a pair~$(M_i,N_i)$ of $2\times 2$ matrices, we may construct the form $Q_i(x,y)=\det(M_ix+N_iy)$.  As is well-known, the determinant function is the
unique polynomial invariant for the action of $\SL_2(K)^2$ on $2\times 2$ matrices over $K$, and since it a degree two invariant, we obtain binary quadratic forms.

Analogously, a hypercube over $K$ naturally yields four binary {\em quartic} forms via its slicings into pairs of cubes over $K$.  Indeed, we have seen that the action of $\SL_2(K)^3$ on $2\times 2\times 2$ cubes over $K$ has a single polynomial invariant, of degree {four}, given by its discriminant.
In analogy with the case of cubes, given 
a hypercube $H\in\H_2(K)$, for each $i\in\{1,2,3,4\}$, we may construct a binary
quartic form
\begin{equation}\label{bqfdef} 
f_i(x,y)=\disc(A_i x+B_i y), 
\end{equation}
where the $(A_i,B_i)$ denote the four slicings of the hypercube $H$ into pairs of cubes over $K$.

Note that the form $f_1$ is invariant under the action of the
subgroup $\{\mathrm{id}\}\times \SL_2(K)^3\subset \SL_2(K)^4$ on
$H\in\H_2(K)$, because the action of $\SL_2(K)^3$
on the cube $A_1x+B_1y$ of linear forms in $x$ and $y$ fixes
$\disc(A_1x+B_1y)$.  The remaining factor of $\SL_2(K)$
in $\SL_2(K)^4$ then acts in the usual way on the binary quartic form
$f_1$, and it is well known that this action has two independent
polynomial invariants, which are traditionally called $I(f_1)$ and
$J(f_1)$ (see \S \ref{sec:binaryquartics} for more details on
binary quartic forms).  These invariants have degrees $2$ and $3$,
respectively, in the coefficients of $f_1$.  Since they coincide with
the corresponding invariants $I$ and $J$ of $f_2$, $f_3$, and $f_4$
(as an easy calculation shows), this yields well-defined
$\SL_2(K)^4$-invariants $I(H)$ and $J(H)$ for all elements $H\in
\H_2(K)$.  The invariants $I(H)$ and $J(H)$ thus have degrees 8 and
12, respectively, in the entries of $H$.

The {\em discriminant} $\disc(f)$ of a binary quartic form $f$ is defined by
\begin{equation}
\disc(f):=4I(f)^3-J(f)^2,\end{equation}
which is nonzero precisely when $f$ has four distinct roots in $\Pone(\overline{K})$; such a binary quartic form is called {\em nondegenerate}.  For
a hypercube $H$, since the $I$ and $J$ invariants are the same for all the $f_i$, so are their discriminants.
We may define the {\em discriminant} of $H$ to be
\begin{equation}\disc(H) :=4I(H)^3-J(H)^2,\end{equation}
which is nonzero precisely when any of the binary quartic forms $f_i$ associated to $H$ has four distinct roots in $\PP^1(\overline{K})$.  
We say the hypercube $H$ is   {\it nondegenerate} if its discriminant is nonzero.

We give a conceptual explanation as to why $I(f_i)=I(f_j)$ and $J(f_i)=J(f_j)$ (and thus $\disc(f_i)=\disc(f_j)$) for all $i$ and $j$ in the next subsection.

\subsection{Genus one curves from hypercubes}\label{sec:hypergenusone}

We now explain how a nondegenerate hypercube $H$ gives rise to a number of genus one curves $C_i$ ($1\leq i\leq 4$), $C_{ij}$ ($1\leq i<j\leq 4)$, and $C_{ijk}$ ($1\leq i<j<k\leq 4$).  We also discuss how these genus one curves are related to each other, and the resulting description of the nondegenerate orbits of $\SL_2(K)^4$ on the space $K^2\otimes K^2\otimes K^2\otimes K^2$ of hypercubes over $K$.

\subsubsection{Genus one curves mapping to \texorpdfstring{$\Pone$}{P1}}\label{sec:binquargenusone}

Given a nondegenerate binary quartic form $f$ over $K$, we may attach to $f$ a genus one curve $C(f)$ over~$K$, namely the normalization of the projectivization of the affine curve
$y^2 = f(x,1)$.
It is known (see, e.g., \cite{BSD, ankim}) that the Jacobian of the curve $C(f)$ may be written as a Weierstrass elliptic curve with coefficients involving the invariants $I(f)$ and $J(f)$ of $f$, namely as
\begin{equation}\label{eq:BQjac}
E(f): y^2 = x^3 - 27 I(f) - 27 J(f).
\end{equation}
We always take $E(f)$ as our model for the Jacobian of $C(f)$.
 
Now given a nondegenerate hypercube $H\in \H_2(K)$, we have seen that we naturally obtain four binary quartic forms $f_1$, $f_2$, $f_3$, $f_4$ over $K$ from $H$.  Thus each hypercube $H\in\H_2(K)$ yields four corresponding genus one curves $C_1$, $C_2$, $C_3$, $C_4$ over $K$, where $C_i=C(f_i)$.

\subsubsection{Genus one curves in \texorpdfstring{$\Pone\times\Pone$}{P1P1}}\label{p1xp1}

These genus one curves obtained from a nondegenerate hypercube
$H\in\H_2(K)$ may be seen more explicitly in $\Pone\times\Pone$.  
Let us first identify $\H_2(K)$ with the space of
quadrilinear forms on $W_1\times W_2\times W_3\times W_4$, where each
$W_i$ ($i\in\{1,2,3,4\}$) is a 2-dimensional $K$-vector space.  (In
this identification, when we write $\H_2(K)=K^2\otimes K^2\otimes
K^2\otimes K^2$, then the $i$th factor of $K^2$ here is the
$K$-vector space dual to $W_i$.)  Then for any $H\in\H_2(K)$, viewed
as such a quadrilinear form, consider the set
\begin{align*}
  C_{12}(K) \,:=\, \bigl\{(w_1, w_2)\in \PP(W_1) \times \PP(W_2)\,:\, \det(H(w_1,w_2,\,\cdot\,, \,\cdot\,)) = 0\bigr\}\, \subset \Pone\times \Pone,
\end{align*}
where we view $H(w_1,w_2,\,\cdot\,, \,\cdot\,)$ naturally as a
bilinear form on $W_3\times W_4$, whose determinant's vanishing or
nonvanishing is thus well-defined.

By definition, $C_{12}(K)$ consists of the set of $K$-points of a
bidegree (2,2) curve $C_{12}$ in $\Pone\times \Pone$, which is a
genus one curve if smooth (precisely when $H$ is
nondegenerate).  The projections of $C_{12}$ to $\PP(W_1)$ or to $\PP(W_2)$ then yield the double covers
of $\Pone$ corresponding to $C_1$ and $C_2$, respectively.  Indeed,
the points of ramification of the projection $C_{12}\to \PP(W_1)$ are
the points $(w_1, w_2)\in C_{12}\subset \PP(W_1)\times\PP(W_2)$ for
which $\det(H(w_1,w_2,\,\cdot,\,,\,\cdot\,))$ has vanishing
discriminant as a quadratic form in $w_1$; this discriminant is
precisely the binary quartic form $f_1$ on $W_1$!

It follows that $C_{12}$ is isomorphic to both $C_1$
and $C_2$, and hence all these genus one curves $C_i$ are isomorphic
to each other: for $1 \leq i < j \leq 4$, we have natural isomorphisms
$$ C_i\cong C_{ij} \cong C_j.$$
It also follows then that all four binary quartic forms $f_i$ must
have the same values for the invariants $I$ and $J$, as claimed at the
end of \S\ref{sec:HCslicing}.  (Indeed, all $I(f_i)$ and all $J(f_i)$
must be the same for all forms $f_i$---up to scaling by $c^2$ and
$c^3$, respectively, for some constant $c$---in order for the
Jacobians in \eqref{eq:BQjac} to be isomorphic.  But then symmetry
considerations show that $c$ must be~1.)

\subsubsection{Genus one curves in \texorpdfstring{$\Pone\times\Pone\times\Pone$}{P1P1P1}}

The curve $C_{12}$ can in fact be mapped into $\PP(W_1)\times\PP(W_2)\times\PP(W_3)\cong \Pone\times\Pone\times\Pone$, as follows.  For a point $(w_1,w_2) \in C_{12}$, since $H(w_1,w_2,\,\cdot\,,\,\cdot\,)$ is singular as a bilinear form on $W_3\times W_4$, the kernel of $H(w_1,w_2,\,\cdot\,,\,\cdot\,)$ in $W_3$ is nonempty; if $H$ is nondegenerate, it
can be shown that the kernel in $W_3$ is one-dimensional.  We thus obtain a well-defined element $w_3\in\PP(W_3)$ such that $H(w_1,w_2,w_3,\,\cdot\,)\equiv 0$.  
Therefore,
\[C_{123}(K) \,:=\, \bigl\{(w_1,w_2,w_3)\in \PP(W_1)\times \PP(W_2)\times \PP(W_3)\,:\, H(w_1,w_2,w_3, \,\cdot\,)\equiv 0\bigr\}\,\subset\,\Pone\times \Pone\times \Pone\]
gives the set of $K$-points of a genus one curve $C_{123}$ in $\PP(W_1)\times\PP(W_2)\times\PP(W_3)\cong \Pone\times \Pone\times \Pone$ defined over $K$; moreover, the projection of $C_{123}$ onto $\PP(W_1)\times\PP(W_2)$ gives an isomorphism onto $C_{12}$.  In particular, $C_{123}$ provides us with explicit isomorphisms among the three curves $C_{12}$, $C_{13}$, and $C_{23}$, through projection and un-projection, which all commute with each other.
  It is natural to package this information together by keeping track simply of the single curve $C_{123}$ in $\PP(W_1)\times\PP(W_2)\times\PP(W_3)$.

\subsubsection{The fundamental tetrahedron of isomorphisms}

We may thus construct curves $C_{123}$, $C_{124}$, $C_{134}$, and $C_{234}$ in $\Pone\times\Pone\times\Pone$ from a nondegenerate hypercube $H\in\H_2(K)$.  These four genus one curves are all isomorphic to each other, as we have already seen.  In fact, we may construct explicit and natural isomorphisms between them, as follows.  Given a point $(w_1,w_2,w_3)$ on $C_{123}\subset \PP(W_1)\times\PP(W_2)\times\PP(W_3)$, the bilinear form $H(\,\cdot\,,w_2,w_3,\,\cdot\,)$ on $W_2\times W_3$ is singular and of rank 1, so there exists a unique $w_4\in\PP(W_4)$ such that $H(\,\cdot\,,w_2,w_3,w_4)\equiv 0$.  This implies that $(w_2,w_3,w_4) \in C_{234}(K)$, giving the desired map $$\tau_{123}^{234}: C_{123}\to C_{234},$$ and similarly we obtain the maps $\tau_{123}^{134}$, $\tau_{124}^{123}$, etc.  Note that each of these maps is invertible, since clearly $\tau_{234}^{123}=(\tau_{123}^{234})^{-1}$, etc.

We thus obtain a tetrahedron of maps:
\begin{equation}\label{eq:tet}
\xymatrix{
& C_{123} \ar@{<->}[dl] \ar@{<->}[dr] \ar@{<->}[dd] & \\
C_{124} \ar@{<-->}[rr] \ar@{<->}[dr] & & C_{134} \\
& C_{234} \ar@{<->}[ur] &
}
\end{equation}

However, these isomorphisms do not all commute with each other!
For example, starting at $C_{123}$ and tracing around the triangle of isomorphisms 
$\tau_{134}^{123}\circ\tau_{124}^{134}\circ\tau_{123}^{124}$ yields a hyperelliptic involution $\iota_1$ on $C_{123}$.  The quotient of $C_{123}$ by this involution $\iota_1$ yields the double cover of $\Pone$ given by $C_1:y^2=f_1(x,1)$.  Similarly, 
the other two triangles of isomorphisms starting at $C_{123}$ yield the involutions $\iota_2$ and $\iota_3$ on $C_{123}$ corresponding to the double covers $C_2$ and $C_3$ of $\Pone$ whose branch points are the roots of $f_2$ and $f_3$, respectively.

If one instead starts at $C_{123}$ and goes around a quadrilateral of isomorphisms, then (viewing the traversal of the quadrilateral as a traversal of two triangles) we obtain the automorphism $\iota_i\circ\iota_j$ of $C_{123}$, which is a translation of $C_{123}$ by a point $P_{ij}$ on the Jacobian of $C_{123}$.  We thus obtain three~points $P_{12}$, $P_{23}$, and $P_{31}$ on the Jacobian $E := E(H)$ of $C:=C_{123}$.  Since $(\iota_1\circ\iota_2)\circ(\iota_2\circ\iota_3)\circ(\iota_{3}\circ\iota_1)=\id$, we have the relation 
\begin{equation}\label{Prel}
P_{12}+P_{23}+P_{31}=0.
\end{equation}
Moreover, these points $P_{ij}$ are nonzero, and they lie in a certain subgroup of $E(K)$ related to the curve $C$, 
called the {\em degree $2$ period-index subgroup} and denoted by $\Jac_{C}^2(K)$ (see Appendix~\ref{appendix:torsors} for more details on the period-index subgroup).  The difference is that the points of $\Jac_{C}^2(K)$ correspond to the divisor classes of $K$-rational divisors of degree $0$ on $C$, whereas the points of $E(K)$ correspond to the $K$-rational divisor classes of degree $0$ (that is, where the divisor class is $K$-rational, but not necessarily any of the divisors in it).  Here, these points $P_{ij}$ arise as differences of actual rational divisors and thus lie in the period-index subgroup.  Indeed, if $D_1$, $D_2$, and $D_3$ denote the degree two divisors on $C_{123}$ corresponding to the three projections to $\PP(W_1)$, $\PP(W_2)$, and $\PP(W_3)$, respectively, then for a point $x$ on $C_{123}$ we have $\iota_i(x)=D_i-x$.  
Thus $\iota_i\circ \iota_j(x)=x+(D_i-D_j),$ so that $P_{ij}=D_i-D_j$.  This also implies the relation (\ref{Prel}).

In summary, we have seen that there is a canonical
elliptic curve $E(H)$ attached to any nondegenerate hypercube $H\in\H_2(K)$, namely,
\begin{equation}\label{jac2}
E(H): y^2 = x^3 - 27I(H) - 27 J(H)
\end{equation}
with
\begin{equation*}
 I(H):=I(f_i)  \qquad\textrm{and}\qquad  J(H):=J(f_i)
\end{equation*}
for $1 \leq i \leq 4$, where $f_1$, $f_2$, $f_3$, $f_4$ are the binary quartic forms naturally arising from $H$.  
Furthermore, $E:=E(H)$ is canonically the Jacobian of each of the genus one curves $C_i$ ($1\leq i\leq 4$) as well
as the genus one curves $C_{ij}$ ($1\leq i\leq j\leq 4$) and $C_{ijk}$ ($1\leq i\leq j\leq k\leq 4)$ arising from $H$.  Finally, there is a natural tetrahedron of isomorphisms \eqref{eq:tet} among the $C_{ijk}$ which does not commute. We thus obtain, in addition to a degree 2 divisor $D_1$ on the curve $C_{123}$, three nonzero rational~points $P_{12}$, $P_{23}$, and $P_{31}$ in the period-index subgroup $\Jac_{C_{123}}^2(K)$ for the curve~$C_{123}$ that sum to zero.

\subsection{Orbit classification of \texorpdfstring{$2\times 2\times 2\times 2$}{2222} hypercubes} \label{sec:HCorbitpreview}

We will show in \S \ref{sec:hypercube} that the data of a genus one curve $C=C_{123}$, the equivalence class of a degree two rational divisor $D=D_1$ on $C$, and three nonzero points $P=P_{12}$, $P'=P_{23}$, and $P''=P_{31}$ (with $P+P'+P''=0$) in the period-index subgroup $\Jac_{C}^2(K)$ of the Jacobian of $C$ is in fact sufficient to recover the orbit of a hypercube $H$.  We have:

\begin{thm}\label{hyperpar}
 For any field $K$ with $\mathrm{char}(K) \nmid 6$,  there is a canonical bijection between
nondegenerate $\GL_2(K)^4$-orbits on the space $K^2 \otimes K^2 \otimes K^2 \otimes K^2$ of hypercubes over $K$ and isomorphism classes of triples $(C,L,(P,P',P''))$, where $C$ is a smooth irreducible genus one curve over $K$, $L$ is a degree $2$ line bundle on~$C$, and $P$, $P'$, $P''$ are nonzero $K$-points that sum to zero in the degree $2$ period-index subgroup $\Jac_C^2(K)$ of the group of $K$-points of the Jacobian of~$C$.
\end{thm}

It is known (see \cite{vinberg,littelmann}) that the ring of polynomial invariants for the action of  $\SL_2(K)^4$ on $K^2 \otimes K^2 \otimes K^2 \otimes K^2$ is freely generated by four polynomials $a_2$, $a_4$, $a_4'$, and $a_6$, having degrees $2$, 4, 4, and $6$, respectively, in the entries of the hypercube.  In terms of the geometric data in Theorem~\ref{hyperpar}, we may write the Jacobian of $C$ as a Weierstrass elliptic curve $y^2=x^3+a_{8}x+a_{12}$, 
on which the points $P=(a_4,a_6)$, $P'=(a_4',a_6')$, $P''=(a_4'',a_6'')$ lie, such that $a_2$ can be interpreted 
as the slope $\frac{a_6'-a_6}{a_4'-a_4}$ of the line connecting $P$ and $P'$ (and $P''$).

From these four invariants, the invariant $a_6'$ (i.e., the $y$-coordinate of $P'$) may be determined:
$$a_6' = a_6 + a_2 (a_4'-a_4).$$  
The coefficients $a_{8}$ and $a_{12}$ of the Weierstrass elliptic curve may also be determined, since there is a unique such elliptic curve passing through the two points $P$ and $P'$. Indeed, we find
\begin{align} \label{eq:a8a12}
a_8 &= a_2(a_6+a_6')-(a_4^2+a_4a_4'+a_4'^2), \qquad \textrm{and} \\
a_{12}&= a_6^2 -a_2a_4(a_6+a_6')+a_4a_4'(a_4+a_4'). \nonumber
\end{align}
Finally, the coordinates of $P''=(a_4'',a_6'')$ may be recovered by finding the third point of intersection of the line $y-a_6=a_2(x-a_4)$ with the elliptic curve $y^2=x^3+a_{8}x+a_{12}$; this yields
\begin{align*}
a_4''&= 3 a_2^2 - a_4-a_4', \qquad \textrm{and} \\
a_6''&= a_2^3-3(a_2a_4-a_6)-a_6-a_6'.  
\end{align*}
In conclusion, $a_2$, $a_4$, $a_4'$, $a_4''$, $a_6$, $a_6'$, $a_6''$, $a_{8}=-27I$, and $a_{12}=-27J$ are all fundamental and important polynomial 
invariants for the action of $\SL_2(K)^4$ on $K^2\otimes K^2 \otimes K^2\otimes K^2$; they all have key 
geometric interpretations and can be expressed as simple polynomials in the four basic invariants 
$a_2$, $a_4$, $a_4'$, and~$a_6$.

\subsection{Symmetrization}\label{sec:symHCpreview}

Just as one may identify the binary quadratic form $ax^2+2bxy+cy^2$ with the symmetric $2\times 2$~matrix
\[\left[\begin{array}{cc}
  a\, &b \\
  b\, & c \end{array}\right],   \]
and the binary cubic form $ax^3+3bx^2y+3cxy^2+dy^3$
with the triply symmetric $2\times 2\times 2$ matrix
\begin{equation}\label{sym3cube}
\raisebox{-2\baselineskip}{
\cube a b b c b c c d
}
\end{equation}

\noindent
(see \cite{hcl1}), one may associate the binary quartic form
\begin{equation}\label{symquartic}
a x^4 + 4 b x^3 y + 6 c x^2 y^2 + 4 d x y^3 + e y^4
\end{equation}
with the  quadruply-symmetric $2\times2\times2\times2$ matrix 
\begin{equation}\label{sym4hyper}
\raisebox{-2\baselineskip}{
\cube a b b c b c c d \qquad \qquad \qquad  \cube b c c d c d d e
} \qquad .
\end{equation}

Using $\Sym_4 K^2$ to denote the space of binary quartic forms of this type, the above association of
the binary quartic form \eqref{symquartic} with the hypercube \eqref{sym4hyper} corresponds 
to the natural inclusion
\[
\Sym_4 K^2\hookrightarrow K^2\otimes K^2\otimes K^2\otimes K^2
\]
of the space of quadruply-symmetric hypercubes into the space of hypercubes.  

Such hypercubes lead to geometric data $(C,L,(P,P',P''))$ as in Theorem~\ref{hyperpar}, but (due to the symmetry) we also have $P=P'=P''$.  Since $P+P'+P''=0$, we see that $P$ is a 3-torsion point on the Jacobian 
of $C$.  Conversely, we will show in \S \ref{sec:4symHC} that this is the only constraint on $P$.  Thus we obtain the
following theorem classifying the orbits of $\GL_2(K)$ on $\Sym_4K^2$:

\begin{thm}\label{sympar}
For any field $K$ with $\mathrm{char}(K) \nmid 6$,  there is a canonical bijection between nondegenerate $\GL_2(K)$-orbits on the space $\Sym_4K^2$ of binary quartic forms over $K$ and isomorphism classes of triples $(C,L,P)$, where $C$ is a smooth genus one curve over $K$, \;$L$ is a degree $2$ line bundle on~$C$, and $P$ a nonzero $3$-torsion point on the Jacobian of $C$ defined over $K$.
\end{thm}

We have already noted in \S\ref{sec:binquargenusone} (see \S\ref{sec:binaryquartics} for further details) that certain $\GL_2(K)$-orbits on $\Sym^4K^2$ correspond to pairs $(X,L)$, where $X$ is a genus one curve and $L$ is a degree 2 line bundle on $X$.   When char$(K) \nmid 6$, these two spaces $\Sym_4K^2$ and $\Sym^4K^2$ are naturally identified, so we obtain two ``dual'' moduli interpretations of the space of binary quartic forms in terms of genus one curves.
The two genus one curves coming from a binary quartic are {\it not} the same, however; they are related by a Hessian-type construction (see \S \ref{sec:4symHC}).

\subsection{Triple symmetrization}\label{sec:triplesym}

The orbit description for binary quartic forms in \S\ref{sec:symHCpreview} was obtained
by imposing a symmetry condition on the orbit description for 
hypercubes.  Rather than imposing a fourfold
symmetry, we may instead impose only a threefold symmetry.  This leads 
to hypercubes of the form
\begin{equation}\label{sym3hyper}
\raisebox{-2\baselineskip}{
\cube a b b c b c c d
\qquad \qquad \qquad
\cube e f f g f g g h
} \qquad .
\end{equation}
That is, these hypercubes can be sliced (in a certain fixed
direction) into two triply symmetric cubes, and therefore 
can naturally be viewed as a pair of binary cubic forms 
\begin{equation}\label{pairofbcfs}
(ax^3+3bx^2y+3cxy^2+dy^3,ex^3+3fx^2y+3gxy^2+hy^3).
\end{equation}

The above association of the pair (\ref{pairofbcfs}) of binary cubic forms 
with the hypercube
\eqref{sym3hyper} corresponds to the natural inclusion map
\[\jmath:K^2\otimes\Sym_3 K^2\hookrightarrow K^2\otimes K^2\otimes K^2\otimes K^2.\]

To such nondegenerate triply symmetric hypercubes, we may associate the usual geometric data $(C,L,(P,P',P''))$ as in Theorem~\ref{hyperpar}, and as in the fully symmetric case,
the symmetry implies that $P$ is a 3-torsion point.
We will show in \S \ref{sec:3symHC} that this is again the only constraint on $P$, and so we obtain the
following theorem classifying the orbits of $\GL_2(K)^2$ on $K^2 \otimes \Sym_3K^2$:
\begin{thm}\label{triplesympar}
For any field $K$ with $\mathrm{char}(K) \nmid 6$,  there is a canonical bijection between
nondegenerate $\GL_2(K)^2$-orbits on the space $K^2\otimes\Sym_3K^2$ of pairs of binary cubic forms over $K$ and isomorphism classes of triples $(C,L,P)$, where $C$ is a smooth genus one curve over $K$, $L$ is a degree $2$ line bundle on~$C$, and $P$ a nonzero $3$-torsion point on the Jacobian of $C$ defined over $K$.
\end{thm}
It is interesting that the data parametrized by the orbits on both $\Sym_4K^2$ and $K^2\otimes\Sym_3K^2$ are the same, and our orbit description in fact allows us to determine an explicit linear transformation in 
$\GL_2(K)^2$ that takes any given nondegenerate element of $K^2 \otimes \Sym_3 K^2$ to an element of $\Sym_4 K^2$.

\subsection{Double symmetrization}

We may instead impose only a twofold symmetry, leading us to study the space
$K^2\otimes K^2\otimes \Sym_2 K^2$ of $2\times 2$ matrices of binary quadratic forms.  
In terms of the geometric data $(C,L,(P,P',P''))$ of Theorem~\ref{hyperpar}, we see that $P'$ and $P''$ coincide, which then determines $P$ by the relation $P+P'+P''=0$. Thus only the information of the point $P'$ needs to be retained.  Since $P\neq 0$, the point $P'$ cannot be 2-torsion, and so (writing now $P'$ as $P$) we obtain the following:

\begin{thm}\label{doublesympar}
For any field $K$ with $\mathrm{char}(K) \nmid 6$,  there is a canonical bijection between
nondegenerate $\GL_2(K)^3$-orbits on the space $K^2\otimes K^2\otimes\Sym_2K^2$ of $2\times 2$ matrices of binary quadratic forms over $K$ and isomorphism classes of triples $(C,L,P)$, where $C$ is a smooth genus one curve over~$K$, \;$L$ is a degree $2$ line bundle on $C$, and $P$ is a non-$2$-torsion point of the period-index subgroup $\Jac_C^2(K)$ in $\Jac(C)(K)$.
\end{thm}

\subsection{Double-double symmetrization}\label{sec:doubledoublesym}

We may, in fact, ask for the hypercubes to be symmetric under any subgroup of the symmetric group $S_4$.  One of the interesting cases arises from the hypercubes fixed under the action of $S_2\times S_2\subset S_4$, which we call double-double symmetrization:

\begin{thm}\label{doubledoublesympar}
For any field $K$ with $\mathrm{char}(K) \nmid 6$,  there is a canonical bijection between
nondegenerate $\GL_2(K)^2$-orbits on the space $\Sym_2K^2\otimes \Sym_2K^2$ of symmetric $2\times 2$ matrices of binary~quadratic forms over $K$ and isomorphism classes of triples $(C,L,P)$, where $C$ is a smooth genus~one curve over $K$, \;$L$ is a degree $2$ line bundle on $C$, and $P$ is a non-$2$-torsion point of the period-index subgroup $\Jac_C^2(K)$ in $\Jac(C)(K)$.
\end{thm}

Over a field $K$ not of characteristic dividing $6$, the space $\Sym^2 K^2 \otimes \Sym^2 K^2$ is isomorphic to the space of the doubly-doubly symmetric hypercubes.  Analogous to the case discussed in \S\ref{sec:triplesym}, there is a natural ``dual'' interpretation for the orbits of this space, also involving genus one curves~$X$ with degree $2$ line bundles and a point in $\Jac_X^2(K)$ (see \S \ref{sec:bideg22forms}); however, the two genus one curves~$C$ and $X$ obtained from an element of $\Sym^2 K^2 \otimes \Sym^2 K^2$ are not the same, but are again related by a certain Hessian-type construction (see \S\ref{sec:22symHC}).

\subsection{Double skew-symmetrization} \label{sec:2skewHCpreview}

Instead of imposing conditions of symmetry, one may impose 
conditions of {\it skew-symmetry} on hypercubes, analogous to those described in \cite[\S 2.6]{hcl1}.  
To define these skew-symmetrizations, let us view again our original hypercube
space $K^2\otimes K^2\otimes K^2\otimes K^2$ as the space of $K$-quadrilinear maps
$W_1\times W_2\times W_3\times W_4\rightarrow K$, where $W_1,W_2,W_3,W_4$ are
$K$-vector spaces of dimension 2 (namely, the $K$-duals of the four factors $K^2$
in $K^2\otimes K^2\otimes K^2\otimes K^2$).
Then given such a quadrilinear map 
\[\phi:W_1\times W_2\times W_3\times W_4\rightarrow K\] 
in $K^2\otimes K^2\otimes K^2\otimes K^2$, one may naturally construct another
$K$-quadrilinear map
\[\bar\phi:W_1 \times W_2\times (W_3\oplus W_4)\times (W_3\oplus W_4)\rightarrow K\] 
that is skew-symmetric in the third and fourth variables; this map
$\bar\phi=\id\otimes\id \otimes \wedge_{2,2}(\phi)$ is given by
\[ \bar\phi\left(r,s,(t,u),(v,w)\right)
=  \phi(r,s,t,w) - \phi(r,s,v,u). \]
Thus we have a natural $K$-linear mapping
\begin{equation}\label{doublefusion}
\id\otimes\wedge_{2,2}:
  K^2\otimes K^2\otimes K^2 \otimes K^2\rightarrow K^2\otimes K^2\otimes\wedge^2(K^2\oplus K^2)=
K^2\otimes K^2\otimes\wedge^2K^4 
\end{equation}
taking $2\times 2\times 2\times 2$ hypercubes to $2\times 2$ matrices of alternating 2-forms 
in four variables.  Explicitly, in terms of fixed bases for $W_1,W_2,W_3,W_4$,
the hypercube \eqref{eq:hyperdraw} maps to the $2\times 2$ matrix of skew-symmetric matrices 
as follows:

\vspace{-.1in}
\begin{equation}\label{explicit}
\left(\begin{array}{cc}
    \left[\begin{array}{cccc}
                 {} & {} & \,h_{1111} & \,\,\,h_{1112} \\
                 {} & {} & \,h_{1121} & \,\,\,h_{1122} \\
                 -h_{1111} & -h_{1121} &   &   \\
                 -h_{1112} & -h_{1122} &   &   \end{array}\right]
                 &
   \left[\begin{array}{cccc}
                 {} & {} & \,h_{1211} & \,\,\,h_{1212} \\
                 {} & {} & \,h_{1221} & \,\,\,h_{1222} \\
                 -h_{1211} & -h_{1221} &   &   \\
                 -h_{1212} & -h_{1222} &   &   \end{array}\right]
                 \\[.5in]                
   \left[\begin{array}{cccc}
                 {} & {} & \,h_{2111} & \,\,\,h_{2112} \\
                 {} & {} & \,h_{2121} & \,\,\,h_{2122} \\
                 -h_{2111} & -h_{2121} &   &   \\
                 -h_{2112} & -h_{2122} &   &   \end{array}\right]
                 &
   \left[\begin{array}{cccc}
                 {} & {} & \,h_{2211} & \,\,\,h_{2212} \\
                 {} & {} & \,h_{2221} & \,\,\,h_{2222} \\
                 -h_{2211} & -h_{2221} &   &   \\
                 -h_{2212} & -h_{2222} &   &   \end{array}\right]
\end{array}\right).
\end{equation}
\vspace{-.05in}

Analogous to the case of double skew-symmetrization of $2\times 2\times 2$ cubes,
where two ideal classes become replaced with a single rank 2 module, in the geometric data
corresponding to a double-skew-symmetric hypercube, two degree 2 line bundles $L_3$ and $L_4$
that come from the curves $C_3$ and $C_4$ are replaced by a single vector bundle $M$ of rank 2 and degree 4.
We thus have the following result (where the nondegeneracy condition on the geometric data is a
mild open condition and will be discussed in \S\ref{sec:deg2moduli}).

\begin{thm}\label{doubleskewpar}
For any field $K$ with $\mathrm{char}(K) \nmid 6$,  there is a canonical bijection between
nondegenerate $\GL_2(K)^2 \times \GL_4(K)$-orbits on the space $K^2\otimes K^2\otimes\wedge^2K^4$ of $2\times 2$ matrices of alternating quaternary $2$-forms over $K$ and isomorphism classes of nondegenerate quadruples $(C,L,P,M)$, where $C$ is a smooth genus one curve over $K$, \;$L$ is a degree~$2$ line bundle on $C$, $P$ is a nonzero point of the period-index subgroup $\Jac_C^2(K)$ in $\Jac(C)(K)$, and $M$ is a rank $2$ vector bundle of degree $4$ such that $\det M \cong P \tns L^{\tns 2}$.
\end{thm}

\subsection{Triple skew-symmetrization}

We may also impose a triple skew-symmetry on hypercubes.  With the same notation as in the previous section, given a quadrilinear map
$$\phi: W_1 \times W_2 \times W_3 \times W_4 \to K,$$
we may construct the $K$-quadrilinear map
$$\bar{\phi}: W_1 \times (W_2 \oplus W_3 \oplus W_4) \times(W_2 \oplus W_3 \oplus W_4) \times (W_2 \oplus W_3 \oplus W_4) \to K,$$ 
which is alternating in the last three factors and explicitly given by 
$$\bar{\phi}(r,(s_1,s_2,s_3),(t_1,t_2,t_3),(u_1,u_2,u_3)) = \sum_{\sigma \in S_3} (-1)^\sigma \phi(r,s_{\sigma(1)},t_{\sigma(2)},u_{\sigma(3)}).$$
Thus, we obtain a natural $K$-linear injection
\begin{equation}\label{triplefusion}
\id\otimes\wedge_{2,2,2}:
  K^2\otimes K^2\otimes K^2 \otimes K^2 \rightarrow K^2\otimes \wedge^3(K^2 \oplus K^2\oplus K^2) = K^2\otimes \wedge^3 K^6
\end{equation}
from the space of hypercubes to the space of pairs of senary alternating $3$-forms. In analyzing the $\GL_2(K) \times \GL_6(K)$-orbits of this larger space, one finds that the degree 2 line bundles $L_2, L_3, L_4$ coming from $C_2,C_3,C_4$ are now replaced by a single rank $3$ vector bundle (which splits into the direct sum of these line bundles for elements in the image of $\id \otimes \wedge_{2,2,2}$).  We thus obtain the following:

\begin{thm} \label{thm:3skewHCpreview}
For any field $K$ with $\mathrm{char}(K) \nmid 6$, there is a canonical bijection between nondegenerate $\GL_2(K) \times \GL_6(K)$-orbits on the space $K^2 \otimes \wedge^3 K^6$ of pairs of senary alternating $3$-forms over $K$ and isomorphism classes of
nondegenerate triples $(C,L,M)$, where $C$ is a smooth genus one curve over $K$, $L$ is a degree $2$ line bundle on $C$, and $M$ is a rank $3$ vector bundle of degree $6$ such that $\det M \cong L^{\tns 3}$.
\end{thm}

\subsection{A simultaneous generalization: triply Hermitian hypercubes} \label{sec:3hermHCpreview}
 
Many of the orbit parametrizations we have discussed in this section can be unified and generalized, via a process we call {\it Hermitianization}.  Just as one can consider square matrices that are Hermitian over a quadratic algebra, we may consider cubical matrices that are Hermitian over a cubic algebra.  The most convenient notion for this purpose is Springer's definition of a {\it cubic Jordan algebra}, which we discuss in more detail in \S \ref{sec:cubicjordan}.  For now, it suffices to say that a cubic Jordan algebra is a generalization of the notion of a cubic field extension, where each element of the cubic Jordan algebra has two other formal conjugates, and there is a well-defined notion of a characteristic polynomial that these conjugate elements all satisfy.  Some simple examples of cubic Jordan algebras over a field $K$ include $K^3$, $K^2$, $K$, cubic field extensions of $K$, $K\times \Mat_{2\times2}(K)$, and $\Mat_{3\times 3}(K)$.

A {\em triply Hermitian} {hypercube} for the cubic algebra $J$ over $K$ is one of the form
\begin{equation}\label{hermhyper}
\raisebox{-2\baselineskip}{
\cube a b {b'} {c''} {b''} {c'} c d \qquad \qquad \qquad \cube e f {f'} {g''} {f''} {g'} g h
} \qquad 
\end{equation}
where $a,d,e,h\in K$ and $b,c,f,g\in J$, and $b',b''$ and $c',c''$ are formal conjugates of $b$ and $c$, respectively.  We denote the space of all triply Hermitian hypercubes for such a cubic Jordan algebra $J/K$ by $\mathscr{C}_2(J)$.  

For each such $J$, there is a natural group $\SL_2(J)$ acting by linear transformations but preserving a certain discriminant quartic form on the space $\mathscr{C}_2(J)$.  Then we obtain a parametrization for the orbits of $\GL_2(K) \times \SL_2(J)$ on $\mathscr{C}_2(J)$ (a full version of the theorem may be found in \S\ref{sec:deg2moduli}):

\begin{thm}\label{triplehermpar}
For any field $K$ with $\mathrm{char}(K) \nmid 6$, and a cubic Jordan algebra $J$ over $K$, there is a canonical bijection between nondegenerate $\GL_2(K) \times \SL_2(J)$-orbits on the space $\mathscr{C}_2(J)$ of triply Hermitian hypercubes for $J$ over $K$ and isomorphism classes of triples $(C,L,\mathcal{F})$, where $C$ is a smooth genus one curve over $K$, \;$L$ is a degree~$2$ line bundle on~$C$, and $\mathcal{F}$ is a flag of vector bundles on $C$ with additional structure coming from $J$, subject to a relation between $L$ and~$\mathcal{F}$.
\end{thm}

In each case, one of the key components of the construction of the correspending geometric~data is a flag variety related to the representation of $\SL_2(J)$ on $\mathscr{C}_2(J)$. Each element of $\mathscr{C}_2(J)$ produces a map from a genus one curve $C$ to this flag variety, thereby giving the flag $\mathcal{F}$ on $C$ described in the theorem. See \S\ref{sec:deg2moduli} for full details.

The cases $J=K^3$, $K\times K$, $K$, $K\times \Mat_{2 \times 2}(K)$, and $\Mat_{3\times 3}(K)$ yield regular hypercubes, triply symmetric hypercubes, doubly symmetric hypercubes, doubly skew-symmetric hypercubes, and triply skew-symmetric hypercubes, respectively.  From this perspective, one may also obtain moduli descriptions of some more exotic spaces, e.g., $\GL_2(K) \times \Sp_6(K)$-orbits on $K^2 \tns \wedge^3_0(K^6)$, $\GL_2(K) \times \Spin_{12}(K)$-orbits on $K^2 \tns K^{32}$, and $\GL_2(K) \times E_7(K)$-orbits on $K^2 \tns K^{56}$.


\section{Main results II: Genus one curves and \texorpdfstring{$3\times 3\times 3$}{333} Rubik's cubes} \label{sec:RCpreview}
 
In this section, we discuss the space of $3\times 3\times 3$ cubical
matrices (``Rubik's cubes'') over $K$, and describe the various
parametrizations of genus one curves with extra data that can be obtained from this
perspective.  Again, no proofs or details are given in this section.
Further details, more basis-free constructions, and proofs may be found in
Section~\ref{sec:hermRC}.

\subsection{On Rubik's cubes}\label{sec:RCslicing}

Let $K$ be a field of characteristic not dividing $6$.  Analogous to
the space of $2 \times 2 \times 2$ cubes from \S \ref{sec:cubes},
we may consider the space of $3 \times 3 \times 3$ cubes, which we call
{\em Rubik's cubes}.  Let $\mathcal{C}_3(K)$ denote the space $K^3
\tns K^3 \tns K^3$.  Then each element of $\mathcal{C}_3(K)$ may be
represented as a  $3 \times 3 \times 3$ 
cubical matrix $B = (b_{ijk})_{1 \leq i,j,k \leq 3}$
with entries in $K$:

\vspace{-.0875in}
\begin{equation} \label{eq:RCpicture}
\raisebox{5\baselineskip}{
\xymatrix@!0@=13pt{
 & & & & b_{311} \ar@{-}[rrr] \ar@{-}[lldd] \ar@{--}[dddddd] & & & b_{312}  \ar@{-}[rrr]   \ar@{-}[lldd] & & & b_{313} \ar@{-}[ddd] \ar@{-}[lldd] \\ 
 & & & & & & & & & & \\
 & & b_{211}  \ar@{-}[rrr]  \ar@{-}[lldd]& & & b_{212}  \ar@{-}[rrr]  \ar@{-}[lldd]& & & b_{213} \ar@{-}[ddd] \ar@{-}[lldd] & & \\
 & & & &  & & &  & & & b_{323} \ar@{-}[ddd]  \ar@{-}[lldd] \\ 
 b_{111}  \ar@{-}[rrr] \ar@{-}[ddd] & & & b_{112}  \ar@{-}[rrr] \ar@{-}[ddd] & & & b_{113} \ar@{-}[ddd]  & & & & \\
 & &  & & &  & & & b_{223} \ar@{-}[ddd]  \ar@{-}[lldd] & & \\
 & & & & b_{331} \ar@{--}[rrrrrr] \ar@{--}[lllldddd] & & & & & & b_{333}  \ar@{-}[lldd] \\ 
 b_{121}  \ar@{-}[rrr] \ar@{-}[ddd] & & & b_{122}  \ar@{-}[rrr] \ar@{-}[ddd] & & & b_{123} \ar@{-}[ddd] & & & & \\
 & &  & & & & & & b_{233}  \ar@{-}[lldd] & & \\
  & & & & & & & & & & \\
 b_{131}  \ar@{-}[rrr] & & & b_{132}  \ar@{-}[rrr] & & & b_{133} & & & & 
}
} \qquad .
\end{equation}
If we denote by $\{e_1,e_2,e_3\}$ the standard basis of $K^3$, then
the element of $\mathcal{C}_3(K)$ described by the cubical matrix 
$B = (b_{ijk})_{1 \leq i,j,k \leq 3}$ above 
is $$\sum_{1 \leq i, j, k \leq 3} b_{ijk}\, e_i \tns e_j \tns e_k.$$
Thus we may 
identify $\mC_3(K)$ with the space of $3\times 3\times 3$ matrices
with entries in $K$ or, simply, the space of {\it Rubik's cubes over $K$}.

As in \S \ref{sec:cubes}, a Rubik's cube can naturally be partitioned 
into three $3 \times 3$ matrices in three
distinct ways.  Namely, the cube $B=(b_{ijk})$ given by 
\eqref{eq:RCpicture} can be sliced into the three $3\times 3\times 3$
matrices $M_\ell$, $N_\ell$, and $P_\ell$, for $\ell=1,2,3$, as follows:
\begin{enumerate}
\item[1)] $M_1 = (b_{1jk})$ is the front face, $N_1 = (b_{2jk})$ is the middle slice,  and $P_1 = (b_{3jk})$ is the back~face;
\item[2)] $M_2 = (b_{i1k})$ is the top face, $N_2 = (b_{i2k})$ is the middle slice,  and $P_2 = (b_{i3k})$ is the bottom~face;
\item[3)] $M_3 = (b_{ij1})$ is the left face, $N_3 = (b_{ij2})$ is the middle slice,  and $P_3 = (b_{ij3})$ is the right~face.
\end{enumerate}
There is also a natural action of $\SL_3(K)^3$ on $\mathcal{C}_3(K)$,
analogous to the action of $\SL_2(K)^3$ on the space
$\mathcal{C}_2(K)$ of $2 \times 2 \times 2$ cubes.  Namely, for any
$i \in \{1,2,3\}$, the element $g = (g_{\ell m}) \in \SL_3(K)$ in the
$i$th factor acts on the cube $B$ by replacing $(M_i,N_i,P_i)$ by
$(g_{11} M_i + g_{12} N_i + g_{13} P_i, \,g_{21} M_i + g_{22} N_i +
g_{23} P_i, \,g_{31} M_i + g_{32} N_i + g_{33} P_i)$.  These three
$\SL_3(K)$-actions commute, giving a well-defined action of
$\SL_3(K)^3$ on $\mathcal{C}_3(K)$.  

Recall that a $2 \times 2 \times 2$ cube naturally gave rise to 
three binary quadratic forms, by slicing the cube and taking
determinants as in \eqref{bqfdet}.  
In the analogous manner, a Rubik's cube $B = (b_{ijk})$
naturally gives rise to three ternary cubic forms
\begin{equation}\label{tcfdet}
f_i(x,y,z) = \det(M_i x + N_i y + P_i z)
\end{equation}
for $1 \leq i \leq 3$.  The ternary cubic form $f_1$ is invariant
under the action of the subgroup $\id \times \SL_3(K)^2
\subset \SL_3(K)^3$ on $B\in \mC_3(K)$.  The remaining factor
of $\SL_3(K)$ acts in the standard way on the ternary cubic form
$f_1$, and it is well known that this action has two independent
polynomial invariants, which are traditionally called $S(f_1)$ and
$T(f_1)$ (see \S\ref{sec:ternarycubics} for more details on
ternary cubic forms).  These invariants have degrees $4$ and $6$,
respectively, in the coefficients of $f_1$.  Analogous to the
situation with $2\times 2\times 2\times 2$ hypercubes, one 
checks that $f_2$ and $f_3$ also have the same invariants as $f_1$,
and so we have produced well-defined
$\SL_3(K)^3$-invariants $S(B)$ and $T(B)$ for Rubik's cubes $B$,
having degrees $12$ and $18$, respectively.

The {\it discriminant} $\disc(f)$ of a ternary cubic form $f$ is defined by 
\begin{equation}
\disc(f) = \frac{1}{1728}(S(f)^3 - T(f)^2),
\end{equation}
and the discriminant of the ternary cubic form $f$ is nonzero
precisely when it cuts out a smooth curve in $\PP^2$; such a
ternary cubic form is called {\em nondegenerate}.  Since the $S$ and
$T$ invariants are common to all the $f_i$, their discriminants are
also all the same, and thus we may naturally define the {\it
  discriminant} of a Rubik's cube $B$ to be
\begin{equation}\disc(B) := \frac{1}{1728}(S(B)^3 - T(B)^2),\end{equation}
which is nonzero precisely when any of the ternary cubic forms $f_i$
associated to $B$ cut out a smooth curve in $\PP^2$.
We say then that the Rubik's cube $B$ is   {\it nondegenerate} if its discriminant is nonzero.

We give a conceptual explanation as to why $S(f_i)=S(f_j)$ and $T(f_i)=T(f_j)$ (and thus $\disc(f_i)=\disc(f_j)$) for all $i$ and $j$ in the next subsection.

\subsection{Genus one curves from Rubik's cubes} \label{sec:g1fromRC}

We now explain how a nondegenerate Rubik's cube $B$ naturally gives rise to
a number of genus one curves $C_i$ ($1\leq i\leq 3$) and $C_{ij}$
($1\leq i< j\leq 3$).  We also discuss how these genus one
curves are related to each other, and the resulting description of the
nondegenerate orbits of $\SL_3(K)^3$ on the space $K^3\otimes
K^3\otimes K^3$ of Rubik's cubes over $K$.

\subsubsection{Genus one curves in \texorpdfstring{$\PP^2$}{P2}}

Given a nondegenerate ternary cubic form $f$ over $K$, we may
attach to $f$ a genus one curve $C(f)$ over $K$, namely the
zero locus of the polynomial $f$ in $\PP^2$.  
It is known (see, e.g., \cite{ArtinRodriguezVillegasTate,ankim}) that the 
Jacobian of the curve $C(f)$ may be written as a Weierstrass
elliptic curve with coefficients expressed in terms of the invariants
$S(f)$ and $T(f)$, namely as
$$E(f): y^2 = x^3 - 27 S(f) x - 54 T(f).$$
We always take $E(f)$ as our model for the Jacobian of $C(f)$.

Now given a nondegenerate Rubik's cube $B\in \mC_3(K)$, we have seen that
we naturally obtain three ternary cubic forms $f_1$, $f_2$, $f_3$,
over $K$ from $B$.  Thus each Rubik's cube $B\in\mC_3(K)$ yields
three corresponding genus one curves $C_1$, $C_2$, $C_3$ over
$K$, where $C_i=C(f_i)$.

\subsubsection{Genus one curves in \texorpdfstring{$\PP^2\times\PP^2$}{P2P2}}

The genus one curves we have obtained from a nondegenerate Rubik's cube
$B\in\mC_3(K)$ may also be embedded naturally in $\PP^2 \times \PP^2$.
Let us first identify $\mC_3(K)$  with the space of trilinear forms
on $W_1 \times W_2 \times W_3$, where each $W_i$ ($i\in\{1,2,3\}$) is
a $3$-dimensional $K$-vector space.  (In this identification, when we
write $\mathcal{C}_3(K) = K^3 \tns K^3 \tns K^3$, then the $i$th factor of
$K^3$ is the $K$-vector space dual to $W_i$.)  Then for any $B
\in \mathcal{C}_3(K)$, viewed as a trilinear form, we consider the set
$$C_{12}(K) := \{(w_1,w_2) \in \PP(W_1) \times \PP(W_2) : B(w_1,w_2,
\cdot) \equiv 0 \} \subset \PP(W_1) \times \PP(W_2) \cong \PP^2 \times
\PP^2.$$ 
If $B$ is nondegenerate, then $C_{12}$ in fact yields the graph of an
isomorphism between $C_1$ and $C_2$.  Indeed, for any point
$w_1 \in C_1$, the form $B(w_1, \cdot, \cdot)$ is singular with a
one-dimensional kernel in $W_2$; this gives a point of $\PP(W_2)$ that
then must lie on $C_2$!

As this process is reversible, it follows that $C_{12}$ is isomorphic
to both $C_1$ and $C_2$ via projection, and hence all the curves
$C_i$ are isomorphic to each other: for $1\leq i< j \leq 3$, if we
define $C_{ij}\subset \PP^2\times \PP^2$ in the analogous manner, then we have
natural isomorphisms
 $$ C_i\cong C_{ij} \cong C_j.$$
It also follows (by the same argument as given at the end of \S\ref{p1xp1}) that all three ternary cubic forms $f_i$
have the same values for the invariants $S$ and $T$, as was claimed at the end of \S\ref{sec:RCslicing}.  

\subsubsection{The fundamental triangle of isomorphisms}

From a nondegenerate Rubik's cube $B$, we have thus constructed three genus
one curves $C_{12}$, $C_{23}$, and $C_{13}$, which are all isomorphic.
In fact, we may construct explicit and natural isomorphisms between
them, e.g.,
$$\tau_{12}^{23}: C_{12} \to C_2 \to C_{23},$$
given by projection and un-projection.  More
explicitly, given a point $(w_1, w_2) \in C_{12}$, the bilinear form
$B(\cdot,w_2,\cdot)$ is singular and has rank $2$, so there exists a
unique $w_3 \in \PP(W_3)$ such that $B(\cdot, w_2, w_3) \equiv 0$.
Then $(w_2, w_3)$ is a point on $C_{23}$, giving the map
$\tau_{12}^{23}$.  Clearly, all such maps are invertible, e.g.,
$\tau_{23}^{12} = (\tau_{12}^{23})^{-1}$.

We thus obtain a triangle of maps
\begin{equation} \label{eq:triangleintro}
\raisebox{2\baselineskip}{
\xymatrix@C=10pt@R=30pt{
& C_{12} \ar@{<->}[rd] \ar@{<->}[ld] & \\
C_{23} \ar@{<->}[rr] && C_{13}
}} \qquad .
\end{equation}

However, this triangle does not commute!  Composing the three maps of
this triangle in a clockwise direction, starting from say $C=C_{12}$, yields an automorphism of $C$ given 
by translation by a point $P$ on the Jacobian $E=\Jac(C)$ of $C$.  
Composing the three maps in the counterclockwise
direction then yields the automorphism of $C$ given by translation by $P'=-P$.  As
before, these points $P$, $P'$ are in fact in the period-index subgroup 
$\Jac_C^3(K)$.  Indeed, if $D_1$ and $D_2$ are the degree~$3$ divisors on $C$ corresponding
to the embeddings into $\PP(W_1)$ and $\PP(W_2)$, respectively, then we find that the difference $D_2 - D_1$
corresponds to the point $P$ on the Jacobian of $C$.

In summary, from a nondegenerate Rubik's cube $B \in
\mathcal{C}_3(K)$, we obtain an elliptic curve
$$E(B): y^2 = x^3 - 27 S(B)x - 54 T(B),$$
with
\begin{equation*}
S(B) = S(f_i) \qquad \textrm{and} \qquad T(B) = T(f_i)
\end{equation*}
for $1 \leq i \leq 3$, where $f_1, f_2$, and $f_3$ are the ternary
cubic forms naturally arising from $B$.  Furthermore, the elliptic
curve $E := E(B)$ is canonically the Jacobian of each of the genus one
curves $C_i$ ($1 \leq i \leq 3$) and $C_{ij}$ ($1 \leq i \leq j \leq
3$) arising from $B$.  Finally, there is a natural triangle
\eqref{eq:triangleintro} of isomorphisms among the $C_{ij}$, which
does not commute.
We thus obtain, in addition to a degree 3 divisor $D_1$ on the curve $C_{12}$, a pair of points
$P$, $P'$ in the period-index subgroup $\Jac^3_{C_{12}}(K)$ for the curve $C_{12}$ that sum to zero.

\subsection{Orbit classification of Rubik's cubes}\label{sec:ocrc}

We will show in \S \ref{sec:333} that the data of a genus one
curve $C=C_{12}$, the equivalence class of a degree $3$ rational divisor $D=D_1$ on $C$, and a
pair $P,P'$ of nonzero points summing to zero in the 
period-index subgroup $\Jac^3_{C}(K)$ of the Jacobian of $C$ is in
fact sufficient to recover the orbit of a Rubik's cube $B$.  We have:

\begin{thm} \label{thm:RCparam1} For any field $K$ with
  $\mathrm{char}(K) \nmid 6$, there is a canonical bijection between
  nondegenerate $\GL_3(K)^3$-orbits on the space $K^3 \tns K^3
  \tns K^3$ of Rubik's cubes over $K$ and isomorphism classes of
  triples $(C,L,(P,P'))$, where $C$ is a smooth genus one curve
  over~$K$, $L$ is a degree $3$ line bundle on $C$, and $P$ and
  $P'$ are nonzero $K$-points that sum to zero in the degree $3$ period-index subgroup
  $\Jac^3_C(K)$ of the group of $K$-points of the Jacobian of $C$.
\end{thm}

It is known (see, e.g., \cite{vinberg}) that the ring of polynomial invariants for the action of $\SL_3(K)^3$ on the space $K^3 \tns K^3 \tns K^3$ is freely generated by three polynomials 
$d_6$, $d_9$, and $d_{12}=-27S$, having degrees $6$, $9$, and $12$, respectively, in the entries of the Rubik's cube.  
In terms of the geometric data in Theorem \ref{thm:RCparam1}, these three invariants have geometric meaning.  We may write the Jacobian of the curve $C$ in Weierstrass form as
$$E: y^2 = x^3 + d_{12} x + d_{18},$$ 
where $d_{18}=-54T$ is a degree $18$ polynomial in the entries of the Rubik's cube, and where the points $P$ and $P'$ are given by 
$(x,y) = (d_6, \pm d_9)$ on the model $E$. 
It is clear then that $d_{18}$ may be expressed in terms of the generators $d_6$, $d_9$, and $d_{12}$.  
In conclusion, $d_6$, $d_9$, $d_{12}$ and $d_{18}$ are all fundamental and important 
polynomial invariants for the action of $\SL_3(K)^3$ on $K^3\otimes K^3 \otimes K^3$; they all have key geometric interpretations and can be expressed as simple polynomials in the three basic invariants $d_6$, $d_9$, and $d_{12}$. 

\subsection{Triple symmetrization}\label{sec:symRCpreview}

Analogous to the case of symmetric hypercubes discussed in \S \ref{sec:symHCpreview},
we may consider symmetric Rubik's cubes.  If
we ask for complete symmetry under the action of $S_3$, the Rubik's
cubes in question will be of the form
\begin{equation}\label{symrubik}
\begin{pmatrix}
a & b & c \\
b & d & e \\
c & e & f
\end{pmatrix}
\qquad
\begin{pmatrix}
b & d & e \\
d & g & h \\
e & h & i
\end{pmatrix}
\qquad
\begin{pmatrix}
c & e & f \\
e & h & i \\
f & i & j \\
\end{pmatrix}
\end{equation}
where each $3 \times 3$ matrix represents one ``slice.''  
Just as quadruply-symmetric hypercubes~\eqref{sym4hyper} could be identified with binary
quartic forms \eqref{symquartic}, in the same way the triply-symmetric Rubik's cube
(\ref{symrubik}) may be identified with the ternary cubic form
\begin{equation}ax^3+3bx^2y+3cx^2z+3dxy^2+6exyz+3fxz^2+gy^3+3hy^2z+3iyz^2+jz^3.\end{equation}
This identification corresponds to the natural inclusion
$$\Sym_3 (K^3) \hookrightarrow K^3 \tns K^3 \tns K^3$$
of the space of triply-symmetric Rubik's cubes into the space of
all Rubik's cubes.

Such Rubik's cubes lead to geometric data $(C,L,(P,P'))$ as in
Theorem~\ref{thm:RCparam1},
but due to the symmetry we also have $P=P'$. Since $P+P'=0$, we see that $P$ is a 2-torsion point on the Jacobian of $C$. Conversely, we will show in \S \ref{sec:3symRC} that this is the only constraint on $P$; thus we obtain the following theorem classifying the orbits of
$\GL_3(K)$ on $\Sym_3 K^3$:

\begin{thm} \label{thm:3symRCpreview}
  For any field $K$ with $\mathrm{char}(K) \nmid 6$, there is a
  natural bijection between nondegenerate $\GL_3(K)$-orbits on the
  space $\Sym_3 K^3$ of ternary cubic forms over $K$ and isomorphism
  classes of triples $(C,L,P)$, where $C$ is a smooth genus one curve
  over $K$, $L$ is a degree $3$ line bundle on $C$, and $P$ is a
  nonzero $2$-torsion point on the Jacobian of $C$ defined over $K$.
\end{thm}

We have already noted in (see \S\ref{sec:ternarycubics} for further details) that certain $\GL_3(K)$-orbits on $\Sym^3K^3$ correspond to pairs $(X,L)$, where $X$ is a genus one curve and $L$ is a degree 3 line bundle on $X$.   When char$(K)\nmid6$, these two spaces $\Sym_3K^3$ and $\Sym^3K^3$ are naturally identified, so we obtain two ``dual'' moduli interpretations of the space of ternary cubics in terms of genus one curves.
However, as in the case of symmetric hypercubes viewed as binary quartics in \S\ref{sec:symHCpreview}, these two genus one curves arising from a ternary cubic are {\it not} the same; they are related by the classical Hessian construction (see \S \ref{sec:3symRC}).

\subsection{Double symmetrization}

The orbit description for ternary cubic forms in \S\ref{sec:symRCpreview} was obtained by imposing a symmetry condition on the orbit description for Rubik's cubes.
Rather than imposing a threefold symmetry, we can impose only a double symmetry to obtain Rubik's cubes of the form
\begin{equation}\label{dsrc}
\begin{pmatrix}
a_1 & b_1 & c_1 \\
b_1 & d_1 & e_1\\
c_1 & e_1 & f_1
\end{pmatrix}
\qquad
\begin{pmatrix}
a_2 & b_2 & c_2 \\
b_2 & d_2 & e_2 \\
c_2 & e_2 & f_2
\end{pmatrix}
\qquad
\begin{pmatrix}
a_3 & b_3 & c_3 \\
b_3 & d_3 & e_3 \\
c_3 & e_3 & f_3
\end{pmatrix},
\end{equation}
where again, each $3 \times 3$ matrix represents a slice of the
Rubik's cube.  Since a symmetric $3 \times 3$ matrix represents a
ternary quadratic form, a doubly symmetric Rubik's cube may be viewed as a triple of ternary quadratic forms
\begin{align}\nonumber
\!\!(a_1 x^2 + 2 b_1 x y + 2 c_1 x z + d_1 y^2 + e_1 y z+ f_1 z^2, a_2 x^2 + 2 b_2 x y + 2 c_2 x z + d_2 y^2 + e_2 y z+ f_2 z^2, \\ \label{ttqfs}
a_3 x^2 + 2 b_3 x y + 2 c_3 x z + d_3 y^2 + e_3 y z+ f_3 z^2). \!\!
\end{align}
The above association of the triple (\ref{ttqfs}) of ternary quadratic forms with the doubly
symmetric Rubik's cube (\ref{dsrc}) corresponds to the natural inclusion
$$K^3 \tns \Sym_2 K^3 \hookrightarrow K^3 \tns K^3 \tns K^3.$$

To such nondegenerate doubly symmetric Rubik's cubes, 
we can associate the usual geometric data $(C,L,(P,P'))$ as in Theorem~\ref{thm:RCparam1}, and as in the fully symmetric case,
the symmetry implies that $P$ is a 2-torsion point.
We will show in \S \ref{sec:3symRC} that this is again the only constraint on $P$, and so we obtain the
following theorem classifying the orbits of $\GL_3(K)^2$ on $K^3 \otimes \Sym_2K^3$:

\begin{thm} \label{thm:2symRCpreview}
For any field $K$ with $\mathrm{char}(K) \nmid 6$, there is a natural bijection between  nondegenerate $\GL_3(K)^2$-orbits on the space $K^3 \tns \Sym_2 K^3$ of triples of ternary quadratic forms over $K$ and isomorphism classes of triples $(C,L,P)$, where $C$ is a smooth genus one curved over $K$, $L$ is a degree $3$ line bundle on $C$, and $P$ is a nonzero $2$-torsion point on the Jacobian of $C$ defined over~$K$.
\end{thm}

Note that the data parametrized by triply symmetric and doubly symmetric Rubik's cubes is the same!  The two orbit parametrizations in fact allow us to construct an explicit linear transformation taking any given nondegenerate element of the space $K^3 \tns \Sym_2 K^3$ to an element of $\Sym_3 K^3$.

\subsection{Double skew-symmetrization}

Instead of imposing conditions of symmetry, one may impose conditions of skew-symmetry on Rubik's cubes.
Let us view again our Rubik's cube space $K^3\otimes K^3\otimes K^3$ as the space of $K$-trilinear map $W_1 \times W_2 \times W_3 \to K$, where $W_1$, $W_2$, and $W_3$ are $3$-dimensional $K$-vector spaces.  Then given such a $K$-trilinear map $\phi$, one may construct another $K$-trilinear map
$$\bar{\phi} : W_1 \times (W_2 \oplus W_3) \times (W_2 \oplus W_3)$$
that is skew-symmetric in the last two variables and is given by
$$\bar{\phi}(r,(s,t),(u,v)) = \phi(r,s,v)-\phi(r,u,t).$$
This gives a natural $K$-linear injection
$$\id \tns \wedge_{3,3} : K^3 \tns K^3 \tns K^3 \to K^3 \tns \wedge^2(K^6)$$
taking Rubik's cubes to triples of alternating $2$-forms in six variables.

In analyzing the $\GL_3(K) \times \GL_6(K)$-orbits of this larger space, one finds that the degree~3 line bundles $L_2$ and $L_3$ coming from $C_2$ and  $C_3$ are now replaced by a single rank 2 vector bundle (which splits into the direct sum of these line bundles for elements in the image of $\id\otimes\wedge_{3,3}$). We thus obtain the following (see \S \ref{sec:deg3special} for details):

\begin{thm} \label{thm:2skewRCpreview}
For any field $K$ with $\mathrm{char}(K) \nmid 6$, there is a natural bijection between  nondegenerate $\GL_3(K) \times \GL_6(K)$-orbits on the space $K^3 \tns \wedge^2 K^6$ of triples of alternating senary $2$-forms over $K$ and isomorphism classes of nondegenerate triples $(C,L,M)$, where $C$ is a smooth genus one curved over $K$, $L$ is a degree $3$ line bundle on $C$, and $M$ is a rank $2$ degree $6$ vector bundle on $C$ with $L^{\tns 2} \cong \det M$.
\end{thm}

\noindent
The nondegeneracy condition on the geometric data is a mild open condition and will be discussed in \S\ref{sec:deg3moduli}.

\subsection{A simultaneous generalization: doubly Hermitian Rubik's cubes}

Many of the orbit parametrizations we have discussed in this section can be unified and generalized, by a Hermitianiziation process that is analogous to the one introduced in \S \ref{sec:3hermHCpreview}.  Recall that a Rubik's cube
may be seen as a triple of $3 \times 3$ matrices.  To define a {\em doubly Hermitian Rubik's cube}, we replace the triple of standard $3 \times 3$ matrices by a triple of $3 \times 3$ matrices that are Hermitian with respect to some quadratic algebra $A$ over $K$.   By a quadratic algebra $A$, we mean an algebra $A$ over $K$ such that each element $a$ of such a quadratic algebra $A$ has a natural conjugate $\bar{a}\in A$ such that both $a$ and~$\bar{a}$ satisfy a common (quadratic) characteristic polynomial over $K$.  
A Hermitian $3 \times 3$ matrix $M$ for the quadratic algebra 
$A$ over $K$ is one of the form
\begin{equation*}
\begin{pmatrix}
a & d & e \\
\bar{d} & b & f \\
\bar{e} & \bar{f} & c
\end{pmatrix}
\end{equation*}
where $a, b, c \in K$ and $d, e, f \in A$.  We denote the space of all doubly Hermitian Rubik's cubes for such a quadratic  algebra $A$ by $\mathscr{C}_3(A)$.  Examples of suitable quadratic algebras $A$ include $K$ itself, $K^2$, $\Mat_{2 \times 2}(K)$, or general quaternion or octonion algebras over $K$.  The various quadratic algebras $A$ used in this paper are discussed in \S \ref{sec:cubicjordan}.

For each such quadratic algebra $A$, there is a natural group $\SL_3(A)$ that acts on $3 \times 3$ Hermitian matrices by linear transformations that preserve the determinant.  Then we may classify the orbits of the action of $\GL_3(K) \times \SL_3(A)$ on $\mathscr{C}_3(A)$ as follows (a full version of this theorem may be found in \S \ref{sec:deg3moduli}):

\begin{thm} \label{thm:2hermRCpreview}
For any field $K$ with $\mathrm{char}(K) \nmid 6$, and a quadratic algebra $A$ over $K$, there is a canonical bijection between nondegenerate $\GL_3(K) \times \SL_3(A)$-orbits on the space $\mathscr{C}_3(A)$ of doubly Hermitian Rubik's cubes for $A$ over $K$ and isomorphism classes of triples $(C,L,M)$, where $C$ is a smooth genus one curve over $K$, \;$L$ is a degree~$3$ line bundle on~$C$, and $M$ is a vector bundle of rank equal to the dimension of $A$ over $K$, with a global faithful $A$-action and other structure coming from $A$, subject to a relation between $L$ and $M$.
\end{thm}

In each case, the vector bundle $M$ arises via a natural map from the curve $C$ to the variety of rank one $3 \times 3$ Hermitian matrices up to scaling.  The rank one Hermitian matrices form a flag variety in the space of all Hermitian matrices, and the vector bundle
$M$ on $C$ is part of the flag on~$C$ coming from the pullback of the universal flag.
See \S \ref{sec:deg3moduli} for details.

The cases $A = K \times K$, $K$, $\Mat_{2 \times 2}(K)$ yield regular Rubik's cubes, doubly symmetric Rubik's cubes, and doubly skew-symmetric Rubik's cubes, respectively.  Other examples include twists of these spaces, as well as the space $K^3 \tns K^{27}$ under the action of $\GL_3(K) \times \mathrm{E}_6(K)$, which arises when $A$ is an octonion algebra.

\subsection*{Preliminaries and notation}
Let $K$ be a field not of characteristic $2$ or $3$.  We will work over the field $K$ for the
majority of the paper, but many of the results have analogues over a $\ZZ[\frac{1}{6}]$-scheme as well.

We use the convention that the projectivization of a vector space
parametrizes lines, not hyperplanes.  For example, a basepoint-free
line bundle $L$ on a variety $X$ induces a natural map $\phi_L : X \to
\PP(H^0(X,L)^\vee)$.
  
A {\em genus one curve} means a proper, smooth, geometrically
connected curve with arithmetic genus $1$, and an {\em elliptic curve}
is such a genus one curve equipped with a base point.

An isomorphism of sets of data $D_1$ and $D_2$, where $D_i$ consists
of a genus one curve $C_i$ and vector bundles for $i = 1$ or $2$, is
an isomorphism $C_1 \to C_2$ such that the pullback of the vector
bundles on $C_2$ are isomorphic to the respective bundles on $C_1$.

If $A$ is an element in a tensor product of vector spaces, we use the notation $A(\cdot,\ldots,\cdot)$ to
denote the multilinear form, where the dots may be replaced by substituting elements of the respective
dual vector spaces.  For example, for vector spaces $V_1$, $V_2$, and $V_3$, 
if $A \in V_1 \tns V_2 \tns V_3$ and $v \in V_1^\vee$,
the notation $A(v,\cdot,\cdot)$ will refer to the evaluation $A \subs v$ of the trilinear form $A$ on $v$,
which gives an element of $V_2 \tns V_3$.  By a slight abuse of notation, we will also use this notation to specify whether  $A(v,\cdot,\cdot)$ vanishes for $v \in \PP(V_1^\vee)$, for example.


\section{Genus one curves with degree \texorpdfstring{$2$, $3$, $4$, or $5$}{2, 3, 4, or 5} line bundles} \label{sec:classical}

In this section, we describe representations $V$ of algebraic groups $G$ whose orbits correspond to genus one curves with degree $d$ line bundles for $2 \leq d \leq 5$. For any field $K$ not of characteristic $2$, $3$, or $5$, there is a natural bijection between nondegenerate orbits in $V(K)/G(K)$ and isomorphism classes of genus one curves defined over $K$ equipped with degree $d$ line bundles.

Most of these quotient descriptions are classical (or at least fairly well known) over an algebraically closed field.   For example, for $d = 2$, the space under consideration is that of binary quartic forms, and for $d = 3$, ternary cubic forms.  What is new here is that, for each case, we also show that with the right forms of the group $G$ acting on the representation $V$, the stabilizer of a generic element of the representation agrees with the automorphism group of the curve and the line bundle.  Thus we obtain an equivalence of moduli stacks, which is important in the arithmetic applications (e.g., in~\cite{arulmanjul-bqcount,arulmanjul-tccount}).  We suspect that some of this section is known to the experts, but has not previously been stated explicitly; see also the related work of Cremona, Fisher, and Stoll \cite{cremonafisher, fisher-ternarycubics, cremonafisherstoll, fisher-pfaffianECs}.

\subsection{Binary quartic forms} \label{sec:binaryquartics}

A {\em binary quartic form} over $K$ is a two-dimensional vector space
$V$ over $K$ equipped with an element $q$ of $\Sym^4 V$.  With a choice of basis
$\{w_1, w_2\}$ for $V^\vee$, a binary quartic form over $K$ may be represented as a
homogeneous degree $4$ polynomial
	\begin{equation} \label{eq:BQpoly}
		q(w_1,w_2) = a w_1^4 + b w_1^3 w_2 + c w_1^2 w_2^2 + d w_1 w_2^3 + e w_2^4,
	\end{equation}
where $a, b, c, d, e \in K$.  The group $\GL(V)$ acts on $\Sym^4 V$ by
acting on $V$ in the standard way.  The ring of $\SL(V)$-invariants of
a binary quartic form $q$ as in equation \eqref{eq:BQpoly} is a
polynomial ring, generated by the two invariants
	\begin{equation*}
		I(q) = 12 a e - 3 b d + c^2 \qquad \textrm{and}\qquad
		J(q) = 72 a c e + 9 b c d - 27 a d^2 - 27 e b^2 - 2 c^3.
	\end{equation*}
The coarse moduli space $\Sym^4 V /\!\!/ \SL_2(V)$ is thus birational to the
affine plane, with coordinates given by the invariants $I$ and $J$.  There is also
a natural notion of the discriminant $\Delta(q) = 4I(q)^3 - J(q)^2$ of a binary quartic form.  
The nonvanishing of the discriminant $\Delta(q)$ corresponds to $q$ having four distinct roots
over the separable closure of $K$; we call such binary quartic forms
{\em nondegenerate}.

We may also consider the following twisted action of $g \in \GL(V)$ on $q \in \Sym^4 V$:
	\begin{equation}
	g \cdot q(w_1,w_2) = (\det g)^{-2} q((w_1,w_2)g)
	\end{equation}
for which the $\SL(V)$-invariants described above are preserved as well.  This representation
is $\GL(V)$ acting on $\Sym^4 V \tns (\wedge^2 V)^{-2}$; we will sometimes denote it by
the action of $\GL(V)^{(-2)}$ on binary quartic forms.  Note that the stabilizer of this
twisted action contains the diagonal $\Gm$ of $\GL(V)$ (e.g., scalar matrices), thereby inducing
an action of $\PGL(V)$ on the space of binary quartic forms.

There is one more action we will consider, which is scaling by squares, \ie $\gamma \in \Gm$ sends
$q \in \Sym^4 V \tns (\wedge^2 V)^{-2}$ to $\gamma^2 q$.  We will denote this as the action of $2 \Gm$
on binary quartics.

A nondegenerate binary quartic form $q$ may be associated to a genus one
curve $C$ in the weighted projective space $\PP(1,1,2)$ by the equation
	\begin{equation} \label{eq:y2=bq}
		y^2 = q(w_1,w_2) = a w_1^4 + b w_1^3 w_2 + c w_1^2 w_2^2 + d w_1 w_2^3 + e w_2^4,
	\end{equation}
where $w_1$ and $w_2$ each have degree $1$ and $y$ has degree $2$.  Over the algebraic
closure, the four roots of the binary quartic correspond to the
four points of $\PP(V^\vee)$ over which $C$ ramifies; in other words, the subscheme of $\PP(V^\vee) = \Pone$ cut out by $q$ is the ramification locus of the two-to-one map $C \to \PP(V^\vee)$.  Nondegeneracy is clearly preserved by all of the group actions above.

From a nondegenerate binary quartic $q$, we thus obtain a smooth irreducible
genus one curve $C$, as well as a degree $2$ line bundle $L$ on $C$, which is the pullback of
$\OO_{\PP(V^\vee)}(1)$ to $C$.  Then the space of sections $H^0(C,L)$ may be identified with the vector space $V$.
The $\GL(V)^{(-2)} \times 2 \Gm$-action on a binary quartic does not change the isomorphism class of the curve $C$ obtained in this way.  
The Jacobian of the curve $C$ associated to $q$ has Weierstrass form 
$$y^2 = x^3 - 27 I(q) x - 27 J(q)$$
(see, e.g., \cite{ankim, cremonafisher}).

Conversely, given a smooth irreducible genus one curve $C$ over
$K$ and a degree~$2$ line bundle $L$, the hyperelliptic map $\eta:
C \to \PP(H^0(C,L)^\vee)$ given by the complete linear series $\left| L \right|$
has a ramification divisor of degree $4$ by the Riemann--Hurwitz Theorem.  The branch divisor is a degree $4$
subscheme of $\PP(H^0(C,L)^\vee)$ defined over $K$, which recovers a
binary quartic form over $K$, up to scaling.  We next compute this scaling factor more precisely,
in order to keep track of the group actions for Theorem \ref{thm:bqorbit} below.

Given the genus one curve $\pi: C \to \Spec K$ and a degree $2$ line bundle $L$ on $C$, we have the
exact sequence
	\begin{equation*}
		0 \ra \eta^* \Omega^1_{\PP(H^0(C,L)^\vee)} \ra \Omega^1_C \ra \Omega^1_{C/\PP(H^0(C,L)^\vee)} \ra 0,
	\end{equation*}
and taking the pushforward under $\eta$ gives the exact sequence
	\begin{equation} \label{eq:bqdefseq}
		0 \ra \Omega^1_{\PP(H^0(C,L)^\vee)} \tns \eta_* \OO_C \ra \eta_* \Omega^1_C \ra \eta_* \Omega^1_{C/\PP(H^0(C,L)^\vee)} \ra 0.
	\end{equation}
The first two terms of the sequence \eqref{eq:bqdefseq} are rank two bundles whose
determinants have degrees $-6$ and $-2$, respectively.  Taking determinants and twisting
gives the map
	\begin{equation*}
	\OO \to (\Omega^1_{\PP(H^0(C,L)^\vee)})^{\tns -2} \tns \omega_{C}^{\tns 2}
	\end{equation*}
where $\omega_C := \pi_* \Omega^1_C$ is the Hodge bundle for $C$.
The induced map on cohomology is
	\begin{equation} \label{eq:bqconstr}
		K \ra \Sym^4 (H^0(C,L)) \tns (\wedge^2 (H^0(C,L)))^{\tns (-2)} \tns \omega_{C}^{\tns 2}
	\end{equation}
and the image of $1 \in K$ is the desired binary quartic form.  If we were not
allowing the action of $2\Gm$ on binary quartics as well, then we would need to
specify a differential to pin down the scaling of the binary quartic form.

In other words, there is an isomorphism between the substack of $\Sym^4 V$ of nondegenerate binary quartic forms
and the moduli problem for $(C,L,\phi,\delta)$, where $C$ is a smooth irreducible genus one curve, $L$ is a degree
$2$ line bundle on $C$, $\phi : H^0(C,L) \to V$ is an isomorphism, and $\delta$ is a differential on $C$.  By
the computation \eqref{eq:bqconstr}, this isomorphism is $\GL(V) \times \Gm$-equivariant: the group acts on binary quartic forms as described above (namely, as $\GL(V)^{(-2)} \times 2\Gm$), while on the geometric data, $\GL(V)$ goes through the isomorphism $\phi$ to act on $H^0(C,L)$, and $\Gm$ acts by scaling on $\delta$.

Thus, by descending to the quotient by these group actions, we obtain the correspondence of Theorem~\ref{thm:bqorbit} below.  The stabilizer of a binary quartic form here is exactly the automorphism group of the geometric data corresponding to the form.  In order to describe the stabilizer, we recall the definition
of the Heisenberg group $\Theta_{E,2}$ for an elliptic curve $E$ as the extension given by the following commutative diagram:
\begin{equation}  \label{eq:Heisenberggroup}
		\xymatrix{
			0 \ar[r] & \Gm \ar[r] \ar@{=}[d] & \Theta_{E,2} \ar[r] \ar[d] & E[2] \ar[r] \ar[d] & 0 \\
			0 \ar[r] & \Gm \ar[r]            & \GL_2        \ar[r]        & \PGL_2 \ar[r]  & 0.\!\!
		}
	\end{equation}
More generally, for any $n$, we may define $\Theta_{E,n}$ analogously as an extension of $E[n]$ by $\Gm$.

\begin{thm} \label{thm:bqorbit}
Let $V$ be a $2$-dimensional vector space over $K$.  Then nondegenerate $\GL(V)^{(-2)} \times 2 \Gm$-orbits on $\Sym^4 V$ are in bijection with isomorphism classes of pairs $(C,L)$, where $C$ is a smooth irreducible genus one curve over $K$ and $L$ is a degree $2$ line bundle on $C$.
The stabilizer group $($as a $K$-scheme$)$ of a nondegenerate element of $\Sym^4 V$ corresponding to $(C,L)$ is an extension of $\Aut(\Jac(C))$ by $\Theta_{\Jac(C),2}$, where $\Jac(C)$ is the Jacobian of the curve $C$, $\Aut(\Jac(C))$ is its automorphism group as an elliptic curve, and $\Theta_{\Jac(C),2}$ is the Heisenberg group as defined in $(\ref{eq:Heisenberggroup})$.
\end{thm}

\begin{remark}
When the Jacobian of $C$ does not have $j$-invariant $0$ or $1728$, the automorphism group of $(C,L)$ is in fact just the direct product $\Theta_{\Jac(C),2} \times \ZZ/2\ZZ \subset \GL(V) \times \Gm$.  More generally, the automorphism group scheme described in Theorem \ref{thm:bqorbit} is not necessarily a split extension.
\end{remark}

The isomorphism class of the pair $(C,L)$ is a torsor for $(E, \OO(3O))$, where $E$ is the Jacobian of $C$ and $O$ is the identity point of $E$. In the language of \cite{cathy-periodindex, explicitndescentI} (see Appendix~\ref{appendix:torsors}), the pair $(C,L)$ represents an element of the kernel of the obstruction map and
may be identified with an element of $H^1(K,\Theta_{E,2})$, where $E$
is the Jacobian of $C$.  Therefore, given an elliptic curve
	$$E: y^2 = x^3 - 27 I x - 27 J,$$
the set $H^1(K,\Theta_{E,2})$ is in correspondence with the set of $\GL(V)^{(-2)}$-orbits of binary quartic
forms $q$ having invariants $I(q) = I$ and $J(q) = J$.

\begin{remark} \label{rmk:bqdivisors}
Theorem \ref{thm:bqorbit} may also be rephrased in terms of {\em divisors} on the genus one curves instead of {\em line bundles}; we also replace the group $\GL(V)^{(-2)} \times 2 \Gm$ with $\PGL(V) \times 2 \Gm$.  In this case, the stabilizer of a binary quartic corresponding to a pair $(C,[D])$ for a genus one curve $C$ with a degree $2$ $K$-rational divisor $D$ (of equivalence class $[D]$) is the group of $K$-points of the corresponding extension of $\Aut(\Jac(C))$ by $\Jac(C)[2]$.  For example, if the $j$-invariant of $\Jac(C)$ is not $0$ or $1728$, then the stabilizer is just $\ZZ/2\ZZ \times \Jac(C)(K)[2]$.  While the nondegenerate orbits of $\Sym^4 V$ under our original group $\GL(V)^{(-2)} \times 2 \Gm$ and the group $\PGL(V) \times 2 \Gm$ are identical, with the latter group action the stabilizer of an element matches the naturally defined automorphism group of the corresponding pair $(C,[D])$.

This revision of Theorem~\ref{thm:bqorbit}---i.e., the bijection between $\PGL(V) \times 2 \Gm$-orbits of binary quartic forms and isomorphism classes of pairs $(C,[D])$---is used in \cite{arulmanjul-bqcount} to prove that the
average size of $2$-Selmer groups for elliptic curves over $\mathbb{Q}$, ordered
by height, is $3$, which in turn implies that the average rank of elliptic
curves (ordered in the same way) is bounded by $1.5$.

\end{remark}

\begin{remark} \label{rmk:bqM11}
By generalizing the above constructions to base schemes over $\ZZ[1/6]$ and letting $V$ be a rank $2$ free module over $\ZZ[1/6]$, we obtain an isomorphism between the nondegenerate substack of the double quotient stack $[2 \Gm \setminus \Sym^4 V \tns (\wedge^2 V)^{-2} / \GL(V)]$ and the quotient stack $[\mathscr{M}_{1,1} / \Theta_{E^{\univ},2}]$, where $\mathscr{M}_{1,1}$ is the moduli space of elliptic curves and $\Theta_{E^{\univ},2}$ is the theta group scheme for the universal elliptic curve $E^{\univ}$ over $\mathscr{M}_{1,1}$.  Here, the $T$-points of the double quotient stack are triples $(\mathcal{V}, L_T, s)$, where $p: \mathcal{V} \to T$ is a rank $2$ vector bundle over $T$, $L_T$ is a line bundle on $T$, and $s$ is a map $L_T^{\tns 2} \to \Sym^4 (\mathcal{V}) \tns (\wedge^2 \mathcal{V})^{-2}$.

Analogous statements will be true for all of the other cases discussed in this section; we discuss this further in \S \ref{sec:diffbases}.
\end{remark}

\subsection{Ternary cubic forms} \label{sec:ternarycubics}

A {\em ternary cubic form} over $K$ is a three-dimensional vector
space $V$ and an element $f$ of $\Sym^3 V$; with a choice of basis
$\mathfrak{B} = \{ w_1, w_2, w_3 \}$ for $V^\vee$, such a form may
be represented as a homogeneous degree $3$ polynomial
	\begin{align} \label{eq:ternarycubic}
		f(w_1,w_2,w_3) = a w_1^3 &+ b w_2^3 + c w_3^3 + a_2 w_1^2 w_2 + a_3 w_1^2 w_3 \\
		 &+ b_1 w_1 w_2^2 + b_3 w_2^2 w_3 + c_1 w_1 w_3^2 + c_2 w_2 w_3^2 + m w_1 w_2 w_3 \nonumber
	\end{align}
with coefficients in $K$.
There is a natural action of $\GL(V)$ on the space of
all ternary cubic forms by the standard action of $\GL(V)$ on $V$.  The ring of
$\SL(V)$-invariants of the space of ternary cubic forms is a polynomial ring
generated by a degree $4$ invariant $d_4=S$ and a degree $6$ invariant
$d_6=T$, and they may be computed by classical formulas (see
\cite{fisher-ternarycubics}, for our choice of scaling).  Thus, the coarse
moduli space $\Sym^3 V /\!\!/ \SL(V)$ is
birational to the affine plane $\mathbb{A}^2$ with coordinates~$d_4$ and~$d_6$.

For the twisted action of $\GL(V)$ on $\Sym^3 V$ where $g \in \GL(V)$ sends 
$f(X,Y,Z) \in \Sym^3 V$ to $(\det g)^{-1} f((X,Y,Z)g)$, the $\SL(V)$-invariants
described above are also preserved.  This representation
is described more accurately as the action of $\GL(V)$ on $\Sym^3 V \tns (\wedge^3 V)^{-1}$,
and we will also sometimes denote it by the action of $\GL(V)^{(-1)}$ on $\Sym^3 V$ to indicate the $-1$-twist
of the determinant.  Note that this representation has a nontrivial kernel, namely the
diagonal $\Gm$ of $\GL(V)$ (e.g., scalar matrices), so it induces a natural $\PGL(V)$-action
on the space of ternary cubic forms.

We will also consider the action of $\Gm$ on $\Sym^3 V$ (or on $\Sym^3 V \tns (\wedge^3 V)^{-1}$)
by scaling.  Scaling the form $f$ by $\gamma \in \Gm$ scales each $\GL(V)^{(-1)}$-invariant by $\gamma^d$,
where $d$ is the degree of the invariant.

We claim that the nondegenerate subset of ternary cubic forms, 
up to the $\GL(V)^{(-1)} \times \Gm$-action,
parametrizes genus one curves equipped with degree $3$ line
bundles, up to isomorphisms.  In particular, a ternary cubic form $f$
defines a curve $\iota: C := \{f = 0\} \hookrightarrow \PP(V^\vee)$.
We say that $f$ a {\em nondegenerate} ternary cubic form if $C$ is smooth,
which occurs if and only if the degree $12$ discriminant $\Delta(f) :=
(d_4^3-d_6^2)/1728$ of $f$ is nonzero.  In this case, the curve $C$ has genus one, and
the pullback $\iota^* \OO_{\PP(V^\vee)}(1)$ is a degree $3$ line
bundle on $C$.

On the other hand, given a (smooth irreducible) genus one curve $\pi: C \to \Spec K$ and a degree~$3$ line
bundle $L$ on $C$, the embedding of $C$ into $\PP(H^0(C,L)^\vee) \cong \PP^2$
gives rise to the exact sequence of sheaves
	\begin{equation*}
		0 \ra \mathcal{I}_C \ra \OO_{\PP(H^0(C,L)^\vee)} \ra \OO_C \ra 0
	\end{equation*}
on $\PP(H^0(C,L)^\vee)$, where $\mathcal{I}_C$ is the ideal defining the curve
$C$.  Tensoring with $\OO_{\PP(H^0(C,L)^\vee)}(3)$, taking cohomology, and tensoring with $H^0(\PP(H^0(C,L)^\vee),\mathcal{I}_C(3))^\vee$ gives the map
	\begin{equation*}
		K \ra H^0(\PP(H^0(C,L)^\vee),\OO(3)) \tns (\wedge^3(H^0(C,L)))^{-1} \tns \omega_C,
	\end{equation*}
where $\omega_C := \pi_* \Omega^1_C$ is the Hodge bundle for the curve $C$.  The image of $1 \in K$ is an element of
$\Sym^3(H^0(C,L)) \tns (\wedge^3(H^0(C,L)))^{-1}$, \ie a ternary cubic form with $V := H^0(C,L)$.  Although
the Hodge bundle $\omega_C$ is trivial over a field, if we did not include the $\Gm$-action here, we would
need to specify a differential to pin down the scaling of the ternary cubic form.

These two functors between ternary cubic forms and pairs $(C,L)$ are inverse to one another.

As in the binary quartic case, there is in fact an isomorphism between the nondegenerate subset of $\Sym^3 V$ and the
moduli problem for $(C,L,\phi,\delta)$, where $C$ is a genus one curve, $L$ is a degree $3$ line bundle,  $\phi: H^0(C,L) \to V$ is an isomorphism, and $\delta$ is a differential on $C$.  This isomorphism is $\GL(V) \times \Gm$-equivariant, and so we obtain the following:

\begin{thm} \label{thm:tcorbit}
Let $V$ be a $3$-dimensional $K$-vector space.  Then nondegenerate $\GL(V)^{(-1)} \times \Gm$-orbits of $\Sym^3 V$
are in bijection with isomorphism classes of pairs $(C,L)$, where $C$ is a smooth irreducible genus one curve over $K$ and 
$L$ is a degree $3$ line bundle on $C$.
The stabilizer group $($as a $K$-scheme$)$ of a nondegenerate element of $\Sym^3 V$ corresponding to $(C,L)$ is an extension of $\Aut(\Jac(C))$ by $\Theta_{\Jac(C),3}$, where $\Jac(C)$ is the Jacobian of $C$, $\Aut(\Jac(C))$ is its automorphism group as an elliptic curve, and $\Theta_{\Jac(C),3}$ is the degree $3$ Heisenberg group of $\Jac(C)$.
\end{thm}

\begin{remark}
More precisely, the stabilizer group of a nondegenerate element of $\Sym^3 V$ coincides with the automorphism group of the corresponding pair $(C,L)$. Here, the automorphism group of the pair $(C,L)$ consists of the $K$-points of the group scheme given by a possibly non-split extension of $\Aut(\Jac(C))$ by $\Theta_{\Jac(C),3}$.  

For example, if $C$ has a point $O$, and $L$ is the line bundle $\mathcal{O}(3O)$, then the extension is indeed split, since for all automorphisms of $C$ fixing $O$, the pullback of $L$ is isomorphic to $L$.  However, in general this extension will not split.  For example, if $C$ is a nontrivial torsor of its Jacobian, then the automorphism group of $(C,L)$ is just $\Theta_{\Jac(C),3}$.

Note that the automorphism group of $(C,L)$ as a {\em torsor} for $(\Jac(C), \mathcal{O}(3O))$ is always~$\Theta_{\Jac(C),3}$.
\end{remark}

Given  a ternary cubic form $f$, the associated genus one curve $C$ has
Jacobian $E := \Jac(C)$ which is determined by the $\SL(V)$-invariants $d_4 = d_4(f)$ and $d_6 = d_6(f)$ (see \cite{ankim} for details).  In
Weierstrass form, the elliptic curve $E$ may be expressed as
	\begin{equation} \label{eq:JacTC}
		E: y^2 = x^3 - 27 d_4 x - 54 d_6.
	\end{equation}
The discriminant $\Delta(E)$ is then given by the formula
	$$1728 \Delta(E) = d_4^3 - d_6^2.$$
Note that $d_4$ and $d_6$ are scaled by the usual action of $\Gm$ on $f$, so they are only relative invariants for the action of the group $\GL(V)^{(-1)} \times \Gm$ on the orbit of $f$; however, all elliptic curves in this orbit
are isomorphic, as $E$ is isomorphic to the elliptic curve given by $y^2 = x^3 - 27 \lambda^4 d_4 x - 54 \lambda^6 d_6$
for any $\lambda \in K^*$.

Conversely, given two numbers $d_4, d_6 \in K$ such that $d_4^3 - d_6^2 \neq 0$, the $\GL(V)^{(-1)} \times \Gm$-orbits of $\Sym^3V$
having {\em relative} invariants $d_4$ and $d_6$ are made up of ternary cubic forms having invariants
$\gamma^4 d_4$ and $\gamma^6 d_6$ for $\gamma \in K^*$.  Over an algebraically closed field,
these ternary cubic forms comprise exactly one orbit, but over a general field they may break up into many $K$-orbits.

Given invariants $d_4, d_6 \in K$, one may specify an elliptic curve $E$, say 
in the form of \eqref{eq:JacTC}.  Then the $\GL(V)^{(-1)} \times \Gm$-orbits of ternary cubic forms over $K$
having associated Jacobian $E$ correspond to pairs $(C,L)$ with Jacobian isomorphic to $E$.  
As described more carefully in Appendix~\ref{appendix:torsors}, such pairs $(C,L)$
are twists of the elliptic curve $E$ and the standard degree $3$
line bundle $\OO(3 \cdot O)$ where $O$ is the identity point on $E$.  
The $K$-rational pairs $(C,L)$ with a choice of isomorphism $\Aut^0(C) \stackrel{\cong}{\ra} E$ are exactly parametrized up to 
isomorphism by $H^1(K,\Theta_{E,3})$.
Thus, the $\GL(V)^{(-1)}$-orbits of $\Sym^3V$ with invariants $d_4$ and $d_6$
are in bijection with the pointed set $H^1(K,\Theta_{E,3})$, where $E$ is the elliptic curve in \eqref{eq:JacTC}.
We recover the following proposition, which is Theorem 2.5 of \cite{fisher-ternarycubics}:

\begin{prop} \label{prop:ternarycubictorsors}
	Let $E$ be an elliptic curve over $K$ with Weierstrass form
		$$y^2 = x^3 - 27 d_4 x - 54 d_6.$$
	Then the set $H^1(K,\Theta_{E,3})$ parametrizes 
	$\GL(V)^{(-1)}$-equivalence classes of $\Sym^3 V$ with invariants $d_4$ and $d_6$.
\end{prop}

\begin{remark}
The difference between the $\GL(V)^{(-1)}$-orbits of $\Sym^3V$ with invariants $d_4$ and $d_6$ and the $\GL(V)^{(-1)} \times \Gm$-orbits with the same relative invariants is subtle.  For example, under the first action, a ternary cubic form $f$ and its negative $-f$ are not generally in the same orbit (but they are in the same $\GL(V)^{(-1)} \times \Gm$-orbit).  While $f$ and $-f$ cut out the same genus one curve $C$ in the plane, they correspond to inverse elements in $H^1(K, \Theta_{\Jac(C),3})$; this exactly reflects the $\ZZ/2\ZZ$ in the automorphism group of any elliptic curve.
\end{remark}

\begin{remark}
Analogously to Remark \ref{rmk:bqdivisors}, we may replace line bundles with equivalence classes of divisors in the statement of Theorem \ref{thm:tcorbit} and the group $\GL(V)^{(-1)} \times \Gm$ with $\PGL(V) \times \Gm$. Then we obtain a bijection between $\PGL(V) \times \Gm$-orbits of $\Sym^3 V$ and isomorphism classes of pairs $(C,[D])$, where $C$ is a genus one curve with a rational degree $3$ divisor $D$.  The stabilizer of a nondegenerate element of $\Sym^3 V$ corresponding to $(C,[D])$ is the automorphism group of $(C,[D])$, namely the group of $K$-points of a possibly non-split extension of $\Aut(\Jac(C))$ by $\Jac(C)[3]$.

In \cite{arulmanjul-tccount}, this bijection is used to prove that the average size of $3$-Selmer groups for elliptic curves over $\Q$, ordered by height, is $4$, implying an improved upper bound of $7/6$ for the average rank of elliptic curves over $\Q$.
\end{remark}

\begin{remark}
As in Remark \ref{rmk:bqM11}, we actually have an isomorphism between the nondegenerate substack of the quotient stack $[\Sym^3 V \tns (\wedge^3 V)^{-1} / \GL(V) \times \Gm]$ and the quotient stack $[\Moneone/\Theta_{\Euniv,3}]$, where $\Theta_{\Euniv,3}$ is the theta group scheme for the universal elliptic curve $\Euniv$ over $\Moneone$.  See \S \ref{sec:diffbases} for more details.
\end{remark}

\subsection{Pairs of quaternary quadratic forms} \label{sec:deg4}

Next, let $V$ and $W$ be vector spaces of dimensions $4$ and $2$, respectively, over $K$.  We study the space $W \tns \Sym^2 V$.  With a choice of bases for both $V$ and $W$, any element of $W \tns \Sym^2 V$ may be represented as a pair of symmetric $4 \times 4$ matrices, say $A$ and $B$.  There is a natural action of $\GL(W) \times \GL(V)$ on this space, and the $\SL(W) \times \SL(V)$-invariants for this space form a polynomial ring generated by two invariants $d_8$ and $d_{12}$ of degrees $8$ and $12$, respectively \cite{ankim}.  In particular, the $\SL(V)$-invariants of the element represented by the pair $(A,B)$ of symmetric matrices are the coefficients of the binary quartic form $q(x,y) = \det (Ax + By)$ (see, e.g., \cite{ankim,merrimansikseksmart}), and the $\SL(W)$-invariants of this binary quartic form $q(x,y)$ are the polynomials $I(q)$ and $J(q)$ from~\S\ref{sec:binaryquartics}. 

We describe briefly how nondegenerate elements of this space naturally give genus one curves with degree $4$ line bundles (see \cite{merrimansikseksmart} for an excellent exposition of the details).  An element of $W \tns \Sym^2 V$ represents a pencil (parametrized by $\PP(W^\vee)$) of quadrics in $\PP(V^\vee)$.  If this pencil is nontrivial, its base locus is a curve $C$, which is of genus one if smooth, by adjunction.  The curve $C$ is smooth exactly when the discriminant $\Delta(q) = 4I(q)^3 - J(q)^2$ is nonzero, just as for binary quartic forms, and elements of $W \tns \Sym^2 V$ giving smooth curves are called {\em nondegenerate}.  Pulling back $\OO_{\PP(V^\vee)}(1)$ to the genus one curve $C$ gives a degree $4$ line bundle.  Furthermore, the rulings on these quadrics parametrize a double cover $D$ of $\PP(W^\vee)$ ramified at the degree $4$ subscheme given by the binary quartic form $q(x,y)$.  This curve $D$ is also a genus one curve, and in fact, as elements of $H^1(K,E)$, the class of $D$ is double that of $C$.

On the other hand, starting with a genus one curve $C$ and a degree $4$ line bundle $L$, the embedding of $C$ into $\PP(H^0(C,L)^\vee) \cong \PP^3$
gives rise to the exact sequence of sheaves
	\begin{equation*}
		0 \ra \mathcal{I}_C \ra \OO_{\PP(H^0(C,L)^\vee)} \ra \OO_C \ra 0,
	\end{equation*}
and twisting by $\OO_{\PP(H^0(C,L)^\vee)}(2)$ shows that $H^0(\PP(H^0(C,L)^\vee),\mathcal{I}_C(2))$ is at least $2$-dimensional.  It is easy to check
that for $C$ to be a smooth irreducible genus one curve, this space is exactly $2$-dimensional (e.g., by computing the free resolution for $\OO_C$ over $\PP(H^0(C,L)^\vee) \cong \PP^3$).  We thus obtain a $2$-dimensional subspace of $\OO_{\PP(H^0(C,L)^\vee)}(2)$ (the space of quadratic forms on $\PP^3$) as desired.

Therefore, we have an isomorphism between nondegenerate elements of $W \tns \Sym^2 V$ and the moduli problem for triples $(C,L,\phi, \psi)$, where $C$ is a genus one curve, $L$ is a degree $4$ line bundle, and $\phi: H^0(C,L) \to V$ and $\psi: H^0(\PP(H^0(C,L)^\vee),\mathcal{I}_C) \to W^\vee$  are isomorphisms.  The group $\GL(W) \times \GL(V)$ acts on $W \tns \Sym^2 V$ in the standard way, and it acts on $\phi$ and $\psi$ by the actions on $V$ and $W^\vee$, respectively.  Therefore, taking the quotients of both sides of this isomorphism by $\GL(W) \times \GL(V)$ gives a correspondence between the orbits and the isomorphism class of pairs $(C,L)$.

\begin{thm} \label{thm:deg4orbit}
Let $W$ and $V$ be $K$-vector spaces of dimensions $2$ and $4$, respectively.  There exists a bijection from nondegenerate
$\GL(W) \times \GL(V)$-orbits of $W \tns \Sym^2 V$ and isomorphism classes of pairs $(C,L)$, where $C$ is a smooth irreducible genus one curve and $L$ is
a degree $4$ line bundle on $C$.  The stabilizer group $($as a $K$-scheme$)$ of a nondegenerate element of $W \tns \Sym^2 V$ corresponding to $(C,L)$ is an extension of $\Aut(\Jac(C))$ by $\Theta_{\Jac(C),4}$, where $\Jac(C)$ is the Jacobian of $C$, $\Aut(\Jac(C))$ is its automorphism group as an elliptic curve and $\Theta_{\Jac(C),4}$ is the degree $4$ Heisenberg group of $\Jac(C)$.
\end{thm}

\begin{remark}
As in the ternary cubic case, the automorphism group of a pair $(C,L)$ is made up of the $K$-points of the group scheme given by a possibly non-split extension of $\Aut(\Jac(C))$ by $\Theta_{\Jac(C),4}$.  In this case, part of the extension splits more often, e.g., if the curve $C$ has a degree $2$ line bundle $M$, then the pair $(C, M^{\tns 2})$ has automorphism group $\Theta_{\Jac(C),4} \rtimes \ZZ/2\ZZ$ if $\Jac(C)$ does not have $j$-invariant $0$ or $1728$.
\end{remark}

\begin{remark} \label{rmk:deg4divisors}
Again, in Theorem \ref{thm:deg4orbit}, we may replace the line bundle $L$ with the equivalence class of a rational degree $4$ divisor $D$, and the group $\GL(W) \times \GL(V)$ with its quotient by the kernel of the multiplication map $\Gm \times \Gm \to \Gm$, sending $(\gamma_1, \gamma_2)$ to $\gamma_1 \gamma_2^2$.  The stabilizer of a nondegenerate element of $W \tns \Sym^2 V$ corresponding to $(C,[D])$ then consists of the $K$-points of the group scheme given by the induced extension of $\Aut(\Jac(C))$ by $\Jac(C)[4]$.

This correspondence is used in \cite{arulmanjul-4Sel} to compute the average size of the $4$-Selmer group of elliptic curves over $\Q$.
\end{remark}

As in the previous cases, the Jacobian of the curve $C$ corresponding to an element of $W \tns \Sym^2 V$ is given by its $\SL(W) \times \SL(V)$-invariants $d_8$ and $d_{12}$:
	$$\Jac(C): y^2 = x^3 - 27 d_8 x - 27 d_{12}.$$
This follows easily from the fact that $d_8$ and $d_{12}$ are the $\SL(W)$-invariants $I(q)$ and $J(q)$ of the binary quartic form $q$.

\subsection{Quintuples of \texorpdfstring{$5 \times 5$}{5x5} skew-symmetric matrices} \label{sec:deg5}

The degree $5$ problem was studied extensively by Fisher in \cite{fisher-genus1pf, fisher-invts} and much of the following can be deduced from the work of  Buchsbaum-Eisenbud in \cite{buchsbaumeisenbud1}.  For completeness, we very briefly remind the reader of the constructions involved.  In this section, let $K$ be a field not of characteristic $2$, $3$, or $5$.

Let $V$ and $W$ be $5$-dimensional $K$-vector spaces.  We consider the space $V \tns \wedge^2 W$, which has $\SL(V) \times \SL(W)$-invariant ring generated by two generators $d_{20}$ and $d_{30}$ of degrees $20$ and $30$, respectively.  An element of $V \tns \wedge^2 W$, with a choice of bases for $V$ and $W$, may be thought of as five $5 \times 5$ skew-symmetric matrices, or a single $5 \times 5$ skew-symmetric matrix of linear forms on $V^\vee$.  Then generically the $4 \times 4$ Pfaffians intersect in a genus one curve in $\PP(V^\vee) = \PP^4$, and pulling back $\OO_{\PP(V^\vee)}(1)$ to this curve gives a degree $5$ line bundle on the curve.  The isomorphism class of the curve and the degree $5$ line bundle are clearly $\GL(V) \times \GL(W)$-equivariant.  The nondegeneracy required here is the nonvanishing of the degree $60$ discriminant formed in the usual way from the generators of the invariant ring.

To construct an element of $V \tns \wedge^2 W$ from a smooth irreducible genus one curve $C$ and a degree $5$ line bundle $L$, one identifies $V = H^0(C,L)$ and $W = H^0(\PP(V^\vee),\mathcal{I}_C(2))$, where $\mathcal{I}_C$ is the ideal sheaf for $C$.  Then one immediately obtains a section of $\OO_{\PP(V^\vee)}(1) \tns \wedge^2 W$ from the free resolution 
$$0 \to \OO(5) \to \OO(-3) \tns W^\vee \to \OO(-2) \tns W \to \OO \to \OO_C$$
of $\OO_C$ over $\PP(H^0(C,L)^\vee)$.

Therefore, there is an isomorphism between elements of $V \tns \wedge^2 W$ and the moduli problem parametrizing $(C,L,\phi,\psi)$, where $C$ is a genus one curve, $L$ is a degree $5$ line bundle, and $\phi: H^0(C,L) \to V$ and $\psi: H^0(\PP(H^0(C,L)^\vee), \mathcal{I}_C(2)) \to W$ are isomorphisms.  Quotienting both sides by the natural actions of $\GL(V) \times \GL(W)$ gives

\begin{thm} \label{thm:deg5moduli}
Let $V$ and $W$ be $5$-dimensional $K$-vector spaces.  Then there is a canonical bijection between nondegenerate $\GL(V) \times \GL(W)$-orbits of $V \tns \wedge^2 W$ and isomorphism classes of pairs $(C,L)$, where $C$ is a genus one curve and $L$ is a degree $5$ line bundle on $C$.  The stabilizer group $($as a $K$-scheme$)$ of a nondegenerate element of $V \tns \wedge^2 W$ corresponding to $(C,L)$ is an extension of $\Aut(\Jac(C))$ by $\Theta_{\Jac(C),5}$, where $\Jac(C)$ is the Jacobian of $C$, $\Aut(\Jac(C))$ is its automorphism group as an elliptic curve, and $\Theta_{\Jac(C),5}$ is the degree $5$ Heisenberg group of $\Jac(C)$.
\end{thm}

\begin{remark} \label{rmk:deg5divisors}
As in the previous three cases, Theorem \ref{thm:deg5moduli} may be rephrased using isomorphism classes of pairs $(C,[D])$, where $D$ is a degree $5$ rational divisor on $C$, where the group $\GL(V) \times \GL(W)$ is replaced by its quotient by the kernel of the multiplication map $\Gm \times \Gm \to \Gm$ sending $(\gamma_1, \gamma_2)$ to $\gamma_1 \gamma_2^2$.  This correspondence is used in \cite{arulmanjul-5Sel} to determine the average size of the $5$-Selmer group of elliptic curves over $\Q$.
\end{remark}

Again, the Jacobian of the curve $C$ associated to a nondegenerate element of $V \tns \wedge^2 W$ is given by the $\SL(V) \times \SL(W)$-invariants $d_{20}$ and $d_{30}$:
	$$\Jac(C) : y^2 = x^3 + d_{20} x + d_{30}.$$
See \cite{fisher-invts} for details and methods for computing these invariants.

\subsection{Some remarks on different bases} \label{sec:diffbases}

This short section may be safely skipped for readers interested in the main theorems of this paper over fields $K$ (as opposed to more general base rings or schemes).  Here, we simply comment on how one can vary the base schemes to study the moduli problems in Section \ref{sec:classical}.

All of the constructions discussed in this section may be made precise over arbitrary $\ZZ[1/30]$-schemes $S$.  In particular, let $\mathcal{M}_d$ denote the moduli stack of genus one curves with degree $d$ line bundles.  Then for $2 \leq d \leq 5$, we claim that $\mathcal{M}_d$ is isomorphic to an open substack of a certain quotient stack $[V/G]$ for a group $G$ and representation $V$ of $G$.  The constructions are straightforward generalizations of those over fields in \S \ref{sec:binaryquartics} through \S \ref{sec:deg5}, e.g., as in Remark \ref{rmk:bqM11}.

For example, the case of ternary cubics is discussed in detail in \cite[\S 2.A.2]{wei-thesis}: the idea is that the $T$-points  of the double quotient $[\Gm \setminus \Sym^3 V \tns (\wedge^3 V)^{-1} / \GL(V)]$ are triples $(\VV, L_T, f)$, where $\VV$ is a rank $3$ vector bundle over $T$, $L_T$ is a line bundle on $T$, and $f$ is a section of $\Sym^3 \VV \tns (\wedge^3 \VV)^{-1} \tns L_T$, and nondegenerate triples correspond to genus one curves over $T$ with degree $3$ line bundles.  The curve $C$ is the zero set of the section of $f$ in $\PP(\VV^\vee)$, and the line bundle $L$ is the pullback of $\OO_{\PP(\VV^\vee)}(1)$ via $C \to \PP(\VV^\vee)$; conversely, given the curve $\pi: C \to T$ and $L$, the construction gives the vector bundle $\pi_*L$ with $L_T := \pi_* \Omega^1_{C/T}$ and an appropriate section.

In each of these cases $2 \leq d \leq 5$, because the groups $G$ appearing have vanishing Galois cohomology group $H^1(K,G)$, this description of $\mathcal{M}_d$ as an open substack of $[V/G]$ immediately implies that isomorphism classes of objects parametrized by $\mathcal{M}_d(K)$ are in bijection with the nondegenerate elements of $V(K)/G(K)$.  The situation is quite different for the moduli space $\mathcal{N}_d$ of genus one curves with degree $d$ rational divisor classes.  For $2 \leq d \leq 5$, it is possible to also describe $\mathcal{N}_d$ as an open substack of a quotient stack $[V/G]$; however, there are genus one curves over $K$ with degree $d$ rational divisor classes that cannot be represented by an element of $V(K)$.  That is, the groups $G$ no longer have connected centers, implying that the quotient map $V \to \mathcal{N}_d$ is not necessarily surjective on $K$-points.

In the remainder of this paper, even though we restrict our attention to working over a base field $K$ not of characteristic $2$ or $3$, most of the constructions we describe may be generalized to other base schemes.  Just as for $\mathcal{M}_d$ here, one may show that the moduli spaces we encounter, of genus one curves with extra data, may be given as open substacks of quotient stacks.


\section{Representations associated to degree \texorpdfstring{$3$}{3} line bundles}
\label{sec:hermRC}

In this section, we study a class of representations whose orbits are related to genus one curves with degree $3$ line bundles.  The main results in this section are summarized in Section \ref{sec:RCpreview}.

We begin by studying the space of Rubik's cubes, which is one of the fundamental examples of this paper, and some of its simpler variants.  The orbit space for Rubik's cubes (also called $3 \times 3 \times 3$ boxes) is related to the moduli space of genus one curves with two non-isomorphic degree~$3$ line bundles.  This may also be identified with the moduli space of genus one curves with a single degree~$3$ line bundle and one nonzero point on the Jacobian (lying in the appropriate period-index subgroup).

In \S \ref{sec:cubicjordan}, we then recall the theory of cubic Jordan algebras, in preparation for the general case.  The main general theorem is in \S \ref{sec:deg3moduli}, where we describe how the orbit spaces of these representations, built from cubic Jordan algebras, are moduli spaces for genus one curves with degree $3$ line bundles and additional vector bundles.
In later subsections, we then explain how to recover the earlier cases from the general theorem, and we also describe another special case that gives rank $2$
vector bundles on genus one curves.

Many of the orbit problems described in this section
are used in \cite{cofreecounting}
to determine average sizes of $3$-Selmer groups in certain families of elliptic curves over $\Q$.


\subsection{Rubik's cubes, or \texorpdfstring{$3 \times 3 \times 3$}{3x3x3} boxes} \label{sec:333}

We study an important example of a representation whose orbits naturally produce
genus one curves with degree $3$ line bundles.  This representation was studied by K.\;Ng in
\cite{ng} over $\CC$, and we extend his results to more general fields $K$.  
The main theorems of this section (Theorem \ref{thm:333bij} and Corollary \ref{cor:333CLP})
are joint with C.\;O'Neil.

Let $V_1$, $V_2$, and $V_3$ be three-dimensional $K$-vector spaces, so $G := \GL(V_1) \times \GL(V_2) \times \GL(V_3)$
acts on the representation $V := V_1 \tns V_2 \tns V_3$.
The following theorem, which is proved for $K=\CC$ in \cite{ng}, 
describes the orbits of these ``Rubik's cubes'' over more general fields $K$.
Here, nondegeneracy is equivalent to the nonvanishing of a degree $36$ invariant, 
which we will describe in more detail below.

\begin{thm} \label{thm:333bij}
	Let $V_1$, $V_2$, and $V_3$ be $3$-dimensional vector spaces over $K$.  Then
	nondegenerate $G$-orbits of $V_1 \tns V_2 \tns V_3$ are in bijection
	with isomorphism classes of quadruples $(C,L_1,L_2,L_3)$, where $C$ is a genus one curve over $K$ and $L_1$, $L_2$, and $L_3$
	are degree $3$ line bundles on $C$ with $L_1$ not isomorphic to $L_2$ or $L_3$ and satisfying $L_1^{\tns 2} \cong L_2 \tns L_3$.
\end{thm}

Note that the action of the group $G$ on $V$ is clearly not faithful: the kernel of the multiplication map
	$$\Gm(V_1) \times \Gm(V_2) \times \Gm(V_3) \to \Gm \subset \GL(V),$$
where each $\Gm(V_i)$ is the group of scalar transformations of $V_i$ for $i = 1,2,3$, fixes
every element in $V_1 \tns V_2 \tns V_3$. For the sole purpose of classifying the orbits, this does not 
matter, since the orbits of $G$ and of the quotient by this kernel on the representation $V$ are identical.
However, this kernel, isomorphic to $\Gm^2$, shows up as part of the automorphism group of the geometric data; in particular, the stabilizer in $G$ of a generic nondegenerate element of $V_1 \tns V_2 \tns V_3$ giving the curve $C$ corresponds to the $K$-points of an extension $H$ of the group scheme $\Jac(C)[3]$ by this $\Gm^2$.  More generally, the stabilizer of a nondegenerate element consists of the $K$-points of a possibly non-split extension of $\Aut(\Jac(C))$ by the group scheme $H$.
For each nondegenerate element, the stabilizer group coincides exactly with the group of automorphisms of the geometric data (if we also record the isomorphism $L_1^{\tns 2} \cong L_2 \tns L_3$ in Theorem \ref{thm:333bij}).

The rest of this section is devoted to describing the construction of the genus one curve and its line bundles from an orbit.  In particular, we prove Theorem \ref{thm:333bij}.

\subsubsection{Geometric construction} \label{sec:RCgeom}

We first describe the construction of the genus one curve and the line bundles from a $G$-orbit.
Let $\AA \in V = V_1 \tns V_2 \tns V_3$, so $\AA$ induces a linear map from $V_1^\vee$ to $V_2 \tns V_3$.  There is a natural {\em determinantal} cubic hypersurface $Y$ in $\PP(V_2 \tns V_3)$; with the choice of bases for $V_2$ and $V_3$, it can be described as the vanishing of the determinant of the $3 \times 3$ matrices that comprise $V_2 \tns V_3$.  Then the intersection of $Y$ with the image of $\PP(V_1^\vee) \to \PP(V_2 \tns V_3)$ is generically a cubic curve $C_1$ on the image of $\PP(V_1^\vee)$, given as the vanishing of a covariant ternary cubic form in $\Sym^3 V_1$.

In other words, the curve $C_1$ is a determinantal variety, given by the determinant of a matrix of linear forms on $\PP(V_1^\vee)$.  Explicitly, with a choice of bases for $V_1$, $V_2$, and $V_3$, one can denote $\AA$ as a $3 \times 3 \times 3$ cube $(a_{rst})_{1 \leq r,s,t \leq 3}$ of elements $a_{rst} \in K$.  Then this ternary cubic form in $\Sym^3 V_1$ may be written simply as
	$$f_1(v) := \det (\AA(v,\cdot,\cdot)),$$
for $v \in V_1^\vee$.  One may similarly define cubic curves $C_2 \subset \PP(V_2^\vee)$ and $C_3 \subset \PP(V_3^\vee)$,
cut out by ternary cubic forms $f_2 \in \Sym^3 V_2$ and $f_3 \in \Sym^3 V_3$.

We call a Rubik's cube $\AA$ {\em nondegenerate} if the variety $C_1$
(equivalently, $C_2$ or $C_3$) thus defined is smooth and one-dimensional, which
corresponds to the nonvanishing of a degree $36$ polynomial in
$a_{rst}$.  This polynomial is called the {\em discriminant} of the
Rubik's cube $\AA$, and it coincides with the usual degree $12$
discriminant $\Delta(f_1)$ of the ternary cubic form $f_1$ (which is 
equal to $\Delta(f_2)$ and $\Delta(f_3)$).

If $\AA$ is nondegenerate, then the degree $3$ plane curve $C_1$ is smooth
of genus one.  For all points $w^\dagger \in C_1$, we claim that
the singular matrix $\AA(w^\dagger,\cdot,\cdot)$ has exactly rank $2$.  If
not, then the $2 \times 2$ minors of $\AA(w,\cdot,\cdot)$ would vanish on
$w^\dagger$, and so would all the partial derivatives
\begin{equation*}
  \left. \frac{\partial f}{\partial w_i} \right|_{w=w^\dagger} = \sum_{s,t=1}^3 a_{ist} A_{ij}^*(w^\dagger)
\end{equation*}
where $A_{ij}^*(w^\dagger)$ is the $(i,j)$th $2 \times 2$ minor of
$\AA(w^\dagger,\cdot,\cdot)$.  Thus, since $C_1$ was assumed to be smooth, we see that
the rank of the matrix $\AA(w^\dagger,\cdot,\cdot)$ cannot drop by two.

In other words, the nondegeneracy condition is equivalent to the condition that the image of $\PP(V_1^\vee)$ in $\PP(V_2 \tns V_3)$
not intersect the image of the Segre variety $\PP(V_2) \times \PP(V_3) \hookrightarrow \PP(V_2 \tns V_3)$.
In the sequel, we assume that $\AA$ is nondegenerate.  (Note that nondegeneracy is preserved
by the group action.)

Given a nondegenerate Rubik's cube $\AA$, define the variety
\begin{equation*}
  C_{12} := \{(w,x) \in \PP(V_1^\vee) \times \PP(V_2^\vee) : \AA(w,x,\cdot) = 0 \} \subset \PP(V_1^\vee) \times \PP(V_2^\vee).
\end{equation*}
Because $\AA$ is a trilinear form and the locus on which it
vanishes in $V_1 \times V_2$ is invariant under scaling, this is a
well-defined locus in $\PP(V_1^\vee) \times \PP(V_2^\vee)$.  Since
$\AA$ is nondegenerate, the projection
	\begin{equation*}
	C_{12} \ra \PP(V_1^\vee)
	\end{equation*}
is an isomorphism onto $C_1$.  The inverse map takes a point $w \in
C_1 \subset \PP(V_1^\vee)$ to the pair $(w,x) \in \PP(V_1^\vee) \times
\PP(V_2^\vee)$, where $x$ corresponds to the exactly one-dimensional
kernel of the linear map $\AA(w,\cdot,\cdot) \in V_2 \tns V_3 \cong \Hom
(V_2^\vee,V_3)$.  This map $C_1 \to C_{12}$ is algebraic, as this
kernel is given as a regular map by the $2 \times 2$ minors of the
matrix $\AA \subs w$.  Therefore, by dimension considerations, the
curve $C_{12}$ is the complete intersection of three bidegree $(1,1)$
forms on $\PP(V_1^\vee) \times \PP(V_2^\vee) = \PP^2 \times \PP^2$.
Similarly, the projection from $C_{12}$ to $\PP(V_2^\vee)$ is an
isomorphism onto $C_2$, which shows that $C_1$ and $C_2$ are
isomorphic.

We may also consider the curves
\begin{align*}
  C_{13} &:= \{(w,y) \in \PP(V_1^\vee) \times \PP(V_3^\vee) : \AA(w,\cdot,y) = 0 \} \\
  C_{23} &:= \{(x,y) \in \PP(V_2^\vee) \times \PP(V_3^\vee) : \AA(\cdot, x, y) = 0 \}
\end{align*}
and the analogous maps between $C_i$, $C_3$, and $C_{i3}$ for $i = 1$ or $2$ are also
isomorphisms.  Thus, all the curves $C_1$, $C_2$, $C_3$, $C_{12}$,
$C_{13}$, and $C_{23}$ are isomorphic, and the nondegeneracy condition is equivalent
to the smoothness of any or all of these curves.  The diagram
\begin{equation} \label{eq:RCcurvediag}
\raisebox{5\baselineskip}{
\xymatrix@R=15pt@C=12pt@M=5pt{
	& & \PP(V_1^\vee) \\
	C_{12} \ar[rr]^{\pi_1^2} \ar[rdd]_{\pi_2^1} & & C_1 \ar@{^{(}->}[u]_{\iota_1} \ar@<0.5ex>[ldd]^{\tau_1^2} \ar@<0.5ex>[rdd]^{\tau_1^3} & & C_{13} \ar[ll]_{\pi_1^3} \ar[ldd]^{\pi_3^1} \\
	\\
	& C_2 \ar@{^{(}->}[ld]_{\iota_2} \ar@<0.5ex>[ruu]^{\tau_2^1} \ar@<0.5ex>[rr]^{\tau_2^3} & & C_3 \ar@{^{(}->}[rd]^{\iota_3} \ar@<0.5ex>[luu]^{\tau_3^1} \ar@<0.5ex>[ll]^{\tau_3^2} \\
	\PP(V_2^\vee) & & & & \PP(V_3^\vee) \\
	       & & C_{23} \ar[uul]^{\tau_2^3} \ar[uur]_{\tau_3^2}
}} \end{equation}
summarizes the relationships between these curves.  By construction,
the maps $\tau_i^j$ and $\tau_j^i$ are inverses to one another.  These
maps from the curve $C_1$ to each projective space give the degree $3$
line bundles
\begin{align*}
  L_1 &:= \iota_1^* \OO_{\PP(V_1^\vee)}(1) \\
  L_2 &:= (\iota_2 \circ \tau_1^2)^* \OO_{\PP(V_2^\vee)}(1) \\
  L_3 &:= (\iota_3 \circ \tau_1^3)^* \OO_{\PP(V_3^\vee)}(1)
\end{align*}
on $C_1$.  For $1 \leq i \leq 3$, all three dimensions of sections of
the degree $3$ bundle $L_i$ arise from pulling back sections from
$\OO_{\PP(V_i^\vee)}(1)$.

\begin{lemma} \label{lem:L1notL2orL3}
	The degree $3$ line bundle $L_1$ on $C_1$ is not isomorphic to either
	of the line bundles $L_2$ or $L_3$.
\end{lemma}

\begin{proof}
It suffices, by symmetry, to show that $L_1$ and $L_2$
are not isomorphic line bundles.  If $L_1 \cong L_2$, then the curve
$C_{12}$ would lie on a diagonal of $\PP^2 \times \PP^2 =
\PP(V_1^\vee) \times \PP(V_2^\vee)$, and with an identification
of the bases for $V_1$ and $V_2$, we have $\AA(w,w,\cdot) = 0$ for all $w
\in C_1$.  Because $C_1$ spans $\PP(V_1^\vee)$, we must
have that $\AA(\cdot,\cdot,y)$ is a skew-symmetric $3 \times 3$ matrix for any
$y \in \PP(V_3^\vee)$.  Since odd-dimensional skew-symmetric matrices
have determinant zero, we would have $C_3 = \PP(V_3^\vee)$, which is a
contradiction.
\end{proof}

\begin{lemma}  \label{lem:RCrelation}
	The line bundles $L_1, L_2, L_3$ on $C_1$ defined above
        satisfy the relation
	\begin{equation} \label{eq:RCrelation}
		L_1 \tns L_1 \cong L_2 \tns L_3.
	\end{equation}
\end{lemma}

\begin{proof}
  For $w \in C_1 \subset \PP(V_1^\vee)$, each coordinate of
  $\tau_1^2(w) \in \PP(V_2^\vee)$ is given by the $2 \times 2$ minors
  $A_{ij}^*(w)$ of $\AA(w,\cdot,\cdot)$ for some fixed $j$ where not all
  $A_{ij}^*(w)$ vanish.  Let $D_2$ be an effective degree $3$ divisor
  on $C_1$ such that $\OO(D_2) \cong L_2$.  Then the points of $D_2$
  (defined over an appropriate extension of $K$) are the preimage on $C_1$
  of the intersection of a hyperplane with the image of the curve $C_{12}$
  in $\PP(V_2^\vee)$; in particular,
  we may choose $D_2$, without loss of generality, to be the divisor
  defined by the locus where a particular minor, say $A_{11}^*(w)$,
  vanishes on the curve $C_1$ but at least one $A_{i1}^*(w)$ is
  nonzero.  Similarly, we may choose a divisor $D_3$ such that
  $\OO(D_3) \cong L_3$ to be the points $w \in C_1$ where $A_{11}^*(w)
  = 0$ but not all other $A_{j1}^*(w)$ vanish.  Then 
  the points of the degree $6$ effective divisor $D_2 +
  D_3$ are exactly the intersection of the curve $C_1$ and
  $A_{11}^*(w) = 0$, and the line bundle $\OO(D_2 + D_3)$ is isomorphic
  to the pullback of $\OO_{\PP(V_1^\vee)}(2)$ to $C_1$.
\end{proof}

The composition maps arising from traversing the inner triangle in \eqref{eq:RCcurvediag}, such as
\begin{equation*}
	\alpha_{123} := \tau_3^1 \circ \tau_2^3 \circ \tau_1^2 : C_1 \ra C_1,
\end{equation*}
are not the identity map.  A straightforward calculation
using Lemma \ref{lem:RCrelation} and its symmetric analogues shows that the automorphism
$\alpha_{123}$ of $C_1$ is given by translation by the point $P_{123}$ in $\Jac(C_1)$
corresponding to the degree $0$ line bundle $L_2 \tns L_1^{-1} \in \Pic^0(C_1)$.
More generally, the analogous three-cycle $\alpha_{ijk}$ is the automorphism of $C_i$ given by translation
by a point $P_{ijk} \in \Jac(C_i)$, where $P_{ijk}$ is the image of the point $\mathrm{sgn}(ijk) P_{123} = \pm P_{123} \in \Jac(C_1)$
under the canonical isomorphism $\Jac(C_1) \to \Jac(C_i)$.

\subsubsection{Bijections}

Because the geometric constructions in the previous section are $G$-invariant,
we have shown that a $G$-orbit of $V$ produces a genus one curve $C$ with three
degree $3$ line bundles $L_1$, $L_2$, $L_3$, such that $L_1^2 \cong L_2 \tns L_3$
and $L_1$ is not isomorphic to $L_2$ or $L_3$.  We show that this data exactly
determines a $G$-orbit of $V$.

\begin{proof}[Proof of Theorem $\ref{thm:333bij}$]
  We have already shown that there is a well-defined map $\Phi$ from
  $G$-orbits of nondegenerate elements of $V$ to the listed
  geometric data.  In the other direction, given such a quadruple $(C,
  L_1, L_2, L_3)$, we consider the multiplication map (\ie the cup
  product on cohomology)
  \begin{equation}
    \mu_{12}: H^0(C,L_1) \tns H^0(C,L_2) \ra H^0(C,L_1 \tns L_2).
  \end{equation}
  A simple case of \cite[Theorem 6]{mumford} shows that $\mu$ is
  surjective.
  Thus, by Riemann-Roch, the kernel of $\mu_{12}$ has dimension $9 - 6
  = 3$.  Now let $V_1 := H^0(C,L_1)$, \,$V_2 := H^0(C,L_2)$, and $V_3 :=
  (\ker(\mu_{12}))^\vee$, so the injection
  \begin{equation*}\label{kercube}
    \ker(\mu_{12}) \hookrightarrow H^0(C,L_1) \tns H^0(C,L_2)
  \end{equation*}
  gives an element of $\Hom(\ker(\mu_{12}), H^0(C,L_1) \tns
  H^0(C,L_2)) \cong V_1 \tns V_2 \tns V_3$.

  If quadruples $(C,L_1,L_2,L_3)$ and $(C',L_1',L_2',L_3')$ are equivalent, then there
  is an isomorphism $\sigma: C \to C'$ such
  that $\sigma^* L_i' \cong L_i$ for $1 \leq i \leq 3$.  The isomorphisms
  induced on the spaces of sections, \eg $H^0(C,L_1)
  \stackrel{\cong}{\ra} H^0(C',L_1')$, commute with the multiplication
  maps, so the Rubik's cubes constructed by their kernels differ only
  by choices of bases.

  We check that the two functors between $G$-orbits of $V$ and
  the equivalence classes of quadruples are inverse to one another.
  Given a quadruple $(C, L_1, L_2, L_3)$ of the appropriate type,
  let the images of the natural embeddings be $C_1 \subset \PP(H^0(C,L_1)^\vee)$,
  $C_2 \subset \PP(H^0(C,L_2)^\vee)$, and $C_{12} \subset \PP(H^0(C,L_1)^\vee) \times \PP(H^0(C,L_2)^\vee)$.
  We construct the trilinear form $\AA \in H^0(C,L_1)
  \tns H^0(C,L_2) \tns (\ker \mu_{12})^\vee$ as above.  Now let 
  	\begin{align*}
  		C_1' &:= \{w \in \PP(H^0(C,L_1)^\vee) : \det \AA(w,\cdot,\cdot) = 0 \} &&\subset \PP(H^0(C,L_1)^\vee) \\
  		C_2' &:= \{x \in \PP(H^0(C,L_2)^\vee) : \det \AA(\cdot,x,\cdot) = 0 \} &&\subset \PP(H^0(C,L_2)^\vee) \\
  		C_{12}' &:= \{(w,x) \in \PP(H^0(C,L_1)^\vee) \times \PP(H^0(C,L_2)^\vee)&&\hspace{-20pt} : \AA(w,x,\cdot) = 0 \} \\
						 &&&\subset \PP(H^0(C,L_1)^\vee) \times \PP(H^0(C,L_2)^\vee)
  	\end{align*}
  be the varieties cut out by the trilinear form $\AA(\cdot,\cdot,\cdot)$.
  
  We claim that $C_1 = C_1'$, \,$C_2 = C_2'$, and $C_{12} = C_{12}'$ as sets
  and thus as varieties, which implies that the curve $C_1'$ is
  isomorphic to $C$ and that the line bundles on $C_1'$ defined as
  pullbacks of $\OO(1)$ on $\PP(H^0(C,L_1)^\vee)$ and
  $\PP(H^0(C,L_2)^\vee)$ are isomorphic to the pullbacks of $L_1$ and
  $L_2$, respectively, via the isomorphism $C \stackrel{\cong}{\ra}
  C_1'$.  For all $(w, x) \in C_{12}$, the construction
  of $\AA$ as the kernel of $\mu_{12}$ implies that
  $\AA(w, x,\cdot) = 0$, so $C_{12}'$ contains $C_{12}$ and
  also $C_1' \supset C_1$ and $C_2' \supset C_2$.
  
  Now either the polynomial $\det \AA(w,\cdot,\cdot)$ or $\det \AA(\cdot,x,\cdot)$ is
  not identically $0$.  If they both were identically $0$, then
  $\AA(w,x,\cdot) = 0$ for all $(w,x) \in \PP(H^0(C,L_1)^\vee) \times
  \PP(H^0(C,L_2)^\vee)$, which contradicts the fact that $\AA$ must
  have nonzero tensor rank by construction.  Without loss of
  generality, assume $\det \AA(w,\cdot,\cdot)$ is not identically zero.  Then
  both $C_1'$ and $C_1$ are given by nonzero degree $3$ polynomials and
  thus define the same variety, so $C_1'$ is a smooth irreducible genus
  one curve in $\PP^2 = \PP(H^0(C,L_1)^\vee)$.  Because $C_1'$ is
  smooth, the trilinear form $\AA$ is nondegenerate, and $C_{12}'$ is
  also smooth and irreducible, hence exactly the same variety as $C_{12}$.
  
  It remains to show that the geometric data coming from a Rubik's cube produces the $G$-orbit
  of the same cube again. Given nondegenerate $\AA \in V_1 \tns V_2 \tns V_3$, where $V_i$ are
  $3$-dimensional vector spaces for $1 \leq i \leq 3$, we have
  described the associated quadruple $(C, L_1, L_2, L_3)$ as the image
  of~$\Phi$.  Then the vector spaces $V_i$ and $H^0(C,L_i)$ are
  naturally isomorphic for $i = 1$, $2$, and $V_3^\vee$ can be identified
  with the kernel of $\mu_{12}$.  With
  these identifications, the Rubik's cube constructed from this
  quadruple is well-defined and $G$-equivalent to the original
  $\AA$.
\end{proof}

We may also rewrite the geometric data in Theorem \ref{thm:333bij} in terms of $K$-points
on the Jacobian of~$C$.  Indeed, instead of keeping track of the line bundles
$L_2$ and $L_3$, it suffices to record the difference of the line bundles as a point
in the Jacobian.  However, not all points $P \in \Jac(C)(K)$ arise
in this way---only those that are the difference of two degree $3$ line bundles arise.  The point $P$ below is exactly
the same point on $\Jac(C)$ by which the curve is translated via the automorphism~$\alpha_{123}$:

\begin{cor} \label{cor:333CLP}
	Let $V_1$, $V_2$, and $V_3$ be $3$-dimensional vector spaces over $K$.  Then
	nondegenerate $G$-orbits of $V_1 \tns V_2 \tns V_3$ are in bijection
	with isomorphism classes of triples $(C,L,P)$, where $C$ is a genus one curve over $K$, $L$ is a degree
	$3$ line bundle on $C$, and $0 \neq P \in \Jac(C)(K)$ such that $P \in \Jac(C)(K) = \Pic^0(C)(K)$ is
	the difference of two elements of $\Pic^3(C)(K)$.
\end{cor}

In Appendix~\ref{appendix:torsors}, we show in Proposition \ref{prop:periodindexsubgp} that
such points $P \in \Jac(C)(K)$ are exactly the nonzero points in the period-index subgroup $\Jac_C^3(K)$, 
yielding Theorem \ref{thm:RCparam1}.

\subsubsection{Invariant theory} \label{sec:333invthy}

The ring of $\SL(V_1) \times \SL(V_2) \times \SL(V_3)$-invariants of the representation $V_1 \tns V_2 \tns V_3$
for three-dimensional vector spaces $V_1$, $V_2$, $V_3$ is generated freely by polynomial
invariants of degrees $6$, $9$, and $12$, respectively.   These have interpretations
in terms of the  geometric data described in the bijection of
Corollary \ref{cor:333CLP}.

\begin{prop} \label{prop:333invthy}
	There exists a choice of normalization for the relative invariants $c_6$, $c_9$, $c_{12}$
	such that given a nondegenerate tensor in $V_1 \tns V_2 \tns V_3$
	corresponding to $(C,L,P)$ as in Corollary~$\ref{cor:333CLP}$, the Jacobian of $C$
	may be expressed in generalized Weierstrass form as
		\begin{equation} \label{eq:333EC}
		E : y^2 + c_9  y = x^3 + c_6 x^2 + c_{12} x
		\end{equation}
	and the point $P$ on $E$ is given by $(0,0)$.
\end{prop}

The $\SL(V_1)$-invariants of the ternary cubic $f = f_1$ given by $(C,L)$ are also clearly invariants 
for the whole space $V_1 \tns V_2 \tns V_3$, since $f$ is fixed under the action of 
$\SL(V_2) \times \SL(V_3)$.  The polynomials $d_4(f)$ and $d_6(f)$
have degrees $12$ and $18$ as invariants of $V_1 \tns V_2 \tns V_3$.  One may check
that with the normalizations above, we have
	\begin{eqnarray*}
		d_4(f) = 16 c_6^2 - 48 c_{12} & \textrm{and} & d_6(f) = - 64 c_6^3 - 216 c_9^2 + 288 c_6 c_{12} 
	\end{eqnarray*}
so that $(E,P)$ may be taken by linear changes of variables to the elliptic curve
	$$E': y^2 = x^3 - 27 d_4(f) x - 54 d_6(f)$$
with the point $P$ becoming $(12 c_6, 108 c_9)$ on $E'$.
This recovers the interpretation of the invariants described in~\S\ref{sec:ocrc}.

One proof of the above proposition is obtained by computing the expression for
the point $P$ in terms of the orbit.  We omit these computations, as all the invariants
have a very large number of terms!%
\footnote{The degree $9$ invariant is also known as the Strassen invariant and has a simple
closed form expression \cite{sturmfels}.  If $\phi$ is represented by three 
$3 \times 3$ matrices $M_1$, $M_2$, and $M_3$
by ``slicing'' in any of the three directions with $\det M_2 \neq 0$, then $c_9(\phi)$ may
be given by the expression $(\det M_2^2(M_1 M_2^{-1} M_3 - M_3 M_2^{-1} M_1))$.}
A more abstract argument is simple: any elliptic curve with a non-identity rational point $P$ may be written in the form
\eqref{eq:333EC}, for some $c_6$, $c_9$, $c_{12} \in K$, where $(0,0)$ is the point $P$.  Then the numbers $c_6$, $c_9$, $c_{12}$ are algebraic invariants of the geometric data $(C, L_1, L_2, L_3)$ coming from a Rubik's cube, up to scaling by $\lambda^6$, $\lambda^9$, $\lambda^{12}$, respectively, for some $\lambda \in K^*$.  Thus, these must be relative $G$-invariants
of the representation $V$.

Given this interpretation of the invariants, we may specialize the correspondence
for a particular elliptic curve $E$ over $K$.  
We think of $c_6$, $c_9$, and $c_{12}$
as polynomial functions of the corresponding degree from $V_1 \tns V_2 \tns V_3 \to K$
(which are, of course, $\SL(V_1) \times \SL(V_2) \times \SL(V_3)$-invariant).  Let
$d_{12} = 16 c_6^2 - 48 c_{12}$ and $d_{18} = - 64 c_6^3 - 216 c_9^2 + 288 c_6 c_{12}$.

\begin{cor} \label{cor:333torsors}
	Let $E$ be an elliptic curve over $K$, given in Weierstrass form as
		$$y^2 = x^3 - 27 a_4 x - 54 a_6.$$
	Then the subset of triples $(\alpha_1,\alpha_2,\alpha_3) \in H^1(K,\Theta_{E,3})^3$ such that 
		\begin{enumerate}[{\rm (i)}]
			\item the sum of the images of $\alpha_i$ under the natural map
					$H^1(K,\Theta_{E,3}) \to H^1(K,E[3])$ is zero,
			\item $\alpha_1$ is not equal to $\alpha_2$ or $\alpha_3$, and 
			\item the images of $\alpha_i$ under $H^1(K,\Theta_{E,3}) \to H^1(K,E[3]) \to H^1(K,E)$ all coincide
		\end{enumerate}
	are in bijection with the $\GL(V_1) \times \GL(V_2) \times \GL(V_3)$-orbits
	of $V_1 \tns V_2 \tns V_3$ that have representatives $\AA$ with
	$d_{12}(\AA) = \lambda^{12}a_4$ and $d_{18}(\AA) = \lambda^{18}a_6$, for some nonzero $\lambda \in K^*$.
\end{cor}

The condition on the invariants merely ensures that both sides of the bijection are restricted
to exactly those Rubik's cubes corresponding to curves with Jacobian $E$.

Given any elliptic curve over $K$, there may be no such $(\alpha_1,\alpha_2,\alpha_3)$, for example,
if the elliptic curve does not have any non-identity rational points.  However, given an elliptic curve $E$ over $K$
of the form \eqref{eq:333EC}, there always exists a $G$-orbit of $V$ where $E$ is the Jacobian of the associated genus one curve.  In particular, taking
$C$ to simply be the trivial torsor $E$ with the degree $3$ line bundles $\OO(3\cdot O)$ and
$\OO(2\cdot O + P)$, where $P$ is the point $(0,0)$ on $E$, constructs such an orbit.

\begin{cor} \label{cor:333surjorbits}
The map from nondegenerate orbits $V(K)/G(K)$ to elliptic curves of the form
	$$y^2 + c_9 y = x^3 + c_6 x^2 + c_{12} x$$
with $c_6$, $c_9$, $c_{12} \in K$, obtained by
taking the Jacobian of the genus one curve associated to the orbit, is surjective.
\end{cor}

For a global number or function field $K$, if we restrict to orbits where the curve $C$ is everywhere locally soluble
(meaning that $C$ has a $K_\nu$-point for every place $\nu$ of $K$), 
Corollary \ref{cor:333torsors} simplifies significantly
and yields a bijection between certain orbits and elements of the $3$-Selmer group of elliptic curves of the form 
\eqref{eq:333EC}.  See \cite{cofreecounting} for details and applications.

\subsection{Symmetric Rubik's cubes} \label{sec:sym333}

In this subsection, we study ``symmetrized'' Rubik's cubes.  There is a natural $S_3$-action on each Rubik's cube, obtained by
permuting the factors $V_i$ for $i = 1$, $2$, $3$, and we study the subset of Rubik's cubes
that are invariant under the subgroup $S_2\subset S_3$, or under all of $S_3$.

\subsubsection{Doubly symmetric Rubik's cubes} \label{sec:2symRC}

The simplest case is that of doubly symmetric Rubik's cubes, i.e., triples of $3 \times 3$ symmetric matrices.
This is the subrepresentation $V_1 \tns \Sym_2V_2 \subset V_1 \tns V_2 \tns V_2$ of $\GL(V_1) \times \GL(V_2)$, for three-dimensional $K$-vector spaces
$V_1$ and $V_2$.  Away from characteristic $2$, this is the same as the quotient $V_1 \tns \Sym^2V_2$.  We give the orbit parametrization for this space in the following basis-free version of Theorem \ref{thm:2symRCpreview}:

\begin{thm} \label{thm:2symRC}
	Let $V_1$ and $V_2$ be $3$-dimensional vector spaces over $K$.  Then the nondegenerate
	$\GL(V_1)\times\GL(V_2)$-orbits of $V_1 \tns \Sym_2V_2$ are in bijection with 
	isomorphism classes of triples $(C,L,P)$, where $C$ is a genus one curve
	over $K$, $L$ is a degree $3$ line bundle on $C$, and $P$ is a nonzero $2$-torsion point of $\Jac(C)(K)$.
\end{thm}

\begin{proof}
	Given an element $\AA$ of $V_1 \tns \Sym_2V_2 \subset V_1 \tns V_2 \tns V_2$, we construct the genus one
	curve $C$ and line bundles $L_1,L_2,L_3$ as before.  Because of the symmetry, the line bundles
	$L_2$ and $L_3$ coincide.  The relation $L_1^{\tns 2} \cong L_2 \tns L_3$ and the fact that
	$L_1 \not\cong L_2$ shows that the point $P := L_1 \tns L_2^{-1}$ is a nonzero $2$-torsion point of $\Jac(C)$.
	Note that because $3E(K) \subset \Jac_C^3(K) \subset E(K)$, all $2$-torsion points
	of $E(K)$ are contained in $\Jac_C^3(K)$, so requiring $P$ to be in $\Jac_C^3(K)$ is not an extra condition.

	On the other hand, given such $(C,L,P)$ and setting $L' = L \tns P$, the proof of
	Theorem \ref{thm:333bij} shows that 
	the quadruple $(C,L_1,L_2,L_3) = (C,L,L',L')$ recovers a
	$\GL(U_1) \times \GL(U_2) \times \GL(U_3)$-orbit of $U_1 \tns U_2 \tns U_3$, where
	$U_1 = H^0(C,L_1)$, $U_2 = H^0(C,L_2)$, and $U_3$ is the dual of the three-dimensional
	kernel of the multiplication map
	$\mu_{12} : U_1 \tns U_2 \to H^0(C,L_1 \tns L_2)$.  
	There exists a natural identification $\psi_{12}: U_1 \to U_2$, taking a
	section $s \in U_1$ to the section $t \in U_2$ such that
	$s(x) = t(x)$ for any $x \in C$.  This identification exists
	because the $2$-torsion point $P = 3 P$ induces a linear automorphism of
	$\PP(H^0(C,L_1)^\vee)$ preserving the image of $C$ as a variety.
	The map
	$$\mu_{12} \circ (1 \tns \psi_{12}):  U_1 \tns U_1 \to H^0(C,L_1 \tns L_2).$$
	factors through $\Sym_2U_1$ because for all $s_1, s_2 \in U_1$, the image of
	$s_1 \tns s_2$ and of $s_2 \tns s_1$ evaluated on $x \in C$ 
	are both equal to $s_1(x) s_2(x).$
\end{proof}

The representation $V_1 \tns \Sym_2V_2$ has $\SL(V_1) \times \SL(V_2)$-invariant ring 
generated by the two previously defined polynomials $c_6$ and $c_{12}$ of degrees $6$ and $12$,
respectively.  These have the same interpretation as in \S \ref{sec:333invthy}.  That is,
the Jacobian of $C$ may be written in normal form as
	$$E: y^2 = x^3 + c_6 x^2 + c_{12} x$$
with nonzero $2$-torsion point $P$ having coordinates $(x,y) = (0,0)$.  In other words, the symmetric tensors are
simply the elements of $V_1 \tns V_2 \tns V_2$ where the degree $9$ invariant vanishes.
Nondegeneracy is again given by the same degree $36$ discriminant, which now factors:
$$\Delta = 16 c_{12}^2 (-4 c_{12} + c_6^2).$$
The reduced discriminant has degree $24$, and nondegeneracy is determined by the nonvanishing of $c_{12}$ and of $-4 c_{12} + c_6^2$.

\subsubsection{Triply symmetric Rubik's cubes, or ternary cubic forms again} \label{sec:3symRC}

We now consider triply symmetric Rubik's cubes, \ie orbits of $\GL(V)$ on $\Sym_3V$ for a three-dimensional vector space $V$ over $K$.
Although this space, away from characteristic $3$, is isomorphic to the space 
of ternary cubic forms $\Sym^3V$ discussed in \S \ref{sec:ternarycubics},%
\footnote{More precisely, these representations are dual to one another if $V$ is self-dual.}
we treat this space as a subspace of $V \tns V \tns V$, instead of a quotient, to obtain a different moduli interpretation.

Because $\Sym_3 V \hookrightarrow V \tns \Sym_2 V$, for
each tensor $\AA \in \Sym_3 V$ we may construct the associated genus one curve $C$ and
line bundles on $C$.  
We then find that the orbit spaces of these two representations in fact
correspond to the same moduli problem!  The following is a basis-free version of Theorem \ref{thm:3symRCpreview}:

\begin{thm} \label{thm:3symRC}
	Let $V$ be a three-dimensional vector space over $K$.  Then nondegenerate $\GL(V)$-orbits of
	$\Sym_3V$ are in bijection with isomorphism classes of triples $(C,L,P)$, where $C$ is a genus one curve
	over $K$, $L$ is a degree $3$ line bundle on $C$, and $P$ is a nonzero $2$-torsion point of $\Jac(C)(K)$.
\end{thm}

\begin{proof}
We simply strengthen the proof of Theorem \ref{thm:2symRC} with the observation that all the vector spaces in question may be naturally identified.  As in the proof of Theorem \ref{thm:2symRC}, given a triple $(C,L,P)$, we let $U_1 = H^0(C,L)$ and $U_2 = U_3 = H^0(C,L \tns P)$.  Using the identifications $\psi_{12}$ and $\psi_{13}$, and by taking corresponding identified bases for $U_1$, $U_2$, and $U_3$, we obtain a triply symmetric
element of $U_1 \tns U_2 \tns U_3$.
\end{proof}

This curve $C$ is not the same as the curve $Z$ associated to the
ternary cubic via \S \ref{sec:ternarycubics}.  The curve $C$ has
a degree $36$ discriminant, while $Z$ has a discriminant of degree
$12$.  In fact, it is easy to check that the curve $C$ is the zero locus of
the Hessian of the ternary cubic form defining $Z$!  (Recall that the Hessian of a ternary cubic form is given by the determinant of the matrix of second partial
derivatives, which yields another ternary cubic form.)  Therefore, we have
another proof of the following:

\begin{prop}
Given a ternary cubic form $f$ over $K$, let $H(f)$ denote the Hessian ternary cubic form.  If the discriminant of $H(f)$ is nonzero,
then the Jacobian of the genus one curve cut out by $H(f)$ is an elliptic curve with a nonzero rational $2$-torsion point.
\end{prop}

\noindent
This fact is classically shown by constructing a fixed-point free involution on the curve cut out by $H(f)$; see \cite[Chapter 3]{dolgachevCAG}.

To describe the invariant theory and the relationship with this curve $Z$, let $\AA$ be a nondegenerate element of $\Sym_3 V$.  Recall that the generators $d_4$ and $d_6$ of the $\SL(V)$-invariant ring of this representation are of degrees $4$ and $6$; we maintain the normalization from \S \ref{sec:ternarycubics}.  Then $\AA$, viewed as a ternary cubic form as in \S \ref{sec:ternarycubics}, gives rise to a genus one curve $Z$ in $\PP(V^\vee)$ whose Jacobian is given by \eqref{eq:JacTC}.

On the other hand, let $C$ denote the genus one curve obtained from viewing $\AA$ as a (symmetric) Rubik's cube.  Then an easy computation shows that the Jacobian of $C$ may be given in Weierstrass form as
\begin{equation} \label{eq:3symRCJac}
E: y^2 = x^3 - 72 d_6 x^2 + 1296 d_4^3 x,
\end{equation}
where the point $(0,0)$ is the nonzero $2$-torsion point on $E$.  The discriminant of $E$ factors as a rational multiple of
$$d_4^6 (d_4^3 - d_6^2),$$
so the reduced discriminant has degree $4 + 12 = 16$.  Note that the factor $d_4^3 - d_6^2$ is a scalar multiple of the discriminant of the elliptic curve $\Jac(Z)$.  Thus, requiring nondegeneracy of $\AA$ as a Rubik's cube implies that both the genus one curve $Z$ and the curve $C$ are smooth.  Also, observe that even without the parametrization of Theorem \ref{thm:3symRC}, we know that all elliptic curves over $K$ with a rational $2$-torsion point may be written in the form \eqref{eq:3symRCJac}, since such an elliptic curve may clearly be written as
$$y^2 = x^3 + a x^2 + b x$$
for any $a, b \in K$ with nonzero discriminant (i.e., $b (a^2 - 4 b) \neq 0$). Taking $d_4 = 1296 / b$ and $d_6 = -23328 a / b^2$ gives \eqref{eq:3symRCJac}.


\subsection{Cubic Jordan algebras} \label{sec:cubicjordan}

A {\em Jordan algebra} over $K$ is a $K$-algebra with a commutative
bilinear product $\bullet$ satisfying the Jordan identity $(x^2
\bullet y) \bullet x = x^2 \bullet (y \bullet x)$.  In this section,
we introduce the ``cubic Jordan algebras'' that will be relevant for the more general degree
$3$ moduli problem.  Some of their connections with geometry
and representation theory will prove vital for describing and proving the
related orbit parametrizations in \S \ref{sec:deg3moduli}.

\subsubsection{Jordan cubic forms and Springer's construction} \label{sec:springer}

Following 
\cite{mccrimmon}, we first briefly describe Springer's
construction of cubic Jordan algebras from cubic forms.
Let $U$ be a finite-dimensional vector space over $K$, and let $\Norm$ be
a cubic form on $U$ such that $\Norm(e) = 1$ for some basepoint $e$.
Then there are naturally associated spur and trace forms, and their
(bi)linearizations, denoted by
	\begin{align} \label{eq:jordanalgforms}
		\Spur(x) &:= \Norm(x,x,e)  &  \Spur(x,y) &:= \Norm(x,y,e) = \Spur(x+y) - \Spur(x) - \Spur(y) \\
		\Tr(x) &:= \Norm(e,e,x)    &  \Tr(x,y)  &:= \Tr(x)\Tr(y) - \Spur(x,y). \nonumber
	\end{align}
In general, an {\em adjoint map} $\sharp$ for the cubic form
$(\Norm,e)$ is a quadratic map $\sharp: U \to U$ satisfying
	\begin{align}
   	\Tr(x^\sharp,y) &= \Norm(x,x,y) \label{eq:tracesharpformula} \\
   	(x^\sharp)^\sharp &= N(x) x  \label{eq:adjointidentity}\\
   	e \times x &= \Tr(x)e - x \label{eq:csharp}
	\end{align}
where $\times$ denotes the bilinearization
	\begin{equation} \label{eq:bilinearsharp}
	x \times y := (x+y)^\sharp - x^\sharp - y^\sharp.
	\end{equation}
An {\em admissible} cubic form is a cubic form $\Norm$ with basepoint $e$ and an
associated adjoint map~$\sharp$.  Such a form gives rise to a natural
Jordan algebra structure on $U$ with unit element $e$ and product given by
	\begin{equation*}
		x \bullet y := \frac{1}{2}\left(x \times y + \Tr(x)y + \Tr(y) x - \Spur(x,y)e \right).
	\end{equation*}
We also have the identity
	$$x^3 - \Tr(x) x^2 + \Spur(x) x - \Norm(x) e=0,$$
i.e., the ``characteristic polynomial of $x$'', evaluated at $x$, vanishes for all $x \in U$ (Cayley--Hamilton Theorem).

For cubic forms, there is often a natural choice of an adjoint map.
In particular, a cubic form is called {\em nondegenerate} if its
associated bilinear trace form $\Tr(\cdot,\cdot)$ is nondegenerate, in
which case $U$ and its dual $U^\vee$ may be canonically identified.
In other words, for fixed $x$, the linear functional $y \mapsto \Norm(x,x,y)$
corresponds to an element $x^\sharp \in U$, giving a quadratic map
	\begin{equation} \label{eq:sharpmap} \xymatrix@R=0pt@C=18pt{
	  \sharp:& U \ar[r] & U^\vee \ar[r]^\cong &  U \\
		&x \ar@{|->}[r] & \Norm(x,x,\cdot) \ar@{|->}[r] & x^\sharp.
	}\end{equation}
This map $\sharp$ by definition satisfies
\eqref{eq:tracesharpformula}; if $\sharp$ also satisfies
\eqref{eq:adjointidentity}, then it is easy to check that
\eqref{eq:csharp} holds as well and thus $\sharp$ is an adjoint map
\cite[\S 4.3]{mccrimmon}.  Therefore, given a nondegenerate cubic form $\Norm$
on $U$ with basepoint $e$, there is a natural Jordan algebra structure
on $U$.

By an abuse of notation, sometimes the map $\sharp$ will refer to
just the first map in \eqref{eq:sharpmap}, that is, the map from $x \in U$
to the linear functional $y \mapsto \Norm(x,x,y)$ in $U^\vee$,
since $U$ and $U^\vee$ are naturally identified.

\subsubsection{Composition algebras and Hermitian matrices} \label{sec:compalgs}

We now describe a class of cubic Jordan algebras that will play a crucial role in 
the representations we study.  We begin with some remarks about composition algebras,
which are used to construct these Jordan algebras.

A {\em composition algebra} $A$ over a field $K$ is a $K$-algebra $A$
with identity element $e$ and a nondegenerate quadratic norm form $q$
on $A$ that satisfies
	\begin{eqnarray*}
	q(e) = 1 & \textrm{and} & q(ab) = q(a)q(b)
	\end{eqnarray*}
for any elements $a,b \in A$.
By a theorem of Hurwitz, such an algebra $A$ is 
either $K$ itself, a quadratic \'{e}tale $K$-algebras, a quaternion
algebra over $K$, or a Cayley algebra over $K$.

\begin{example} \label{ex:compalgsKbar}
  If $K = \overline{K}$ is algebraically closed, then the only
  composition algebras are the four {\em split} composition algebras
  of dimensions $1$, $2$, $4$, and $8$, namely the split unarions
  $\mathscr{U}(K) := K$, the split binarions $\mathscr{B}(K) \cong K
  \times K$, the split quaternions $\mathscr{Q}(K) \cong \Mat_{2
    \times 2}(K)$, and the split octonions $\mathscr{O}(K)$ with the
  natural norm forms.
\end{example}

\begin{example}
  Over $K = \mathbb{R}$, in addition to the split composition
  algebras, there exist the usual division algebras of dimensions $2$,
  $4$, and $8$: the complex numbers $\mathbb{C}$, the Hamiltonian
  quaternions $\mathbb{H}$, and the Cayley octonions $\mathbb{O}$,
  respectively.
\end{example}

If the quadratic form $q$ on the composition algebra $A$ is
nondegenerate, then $A$ is alternative separable of degree $2$ and has
an involution $\star$ sending $a \in A$ to $a^{\star} = \overline{a}$
\cite[Prop.~33.9]{bookofinvolutions}.  Any element $a \in A$ hence
satisfies the equation
	\begin{equation*}
		a^2 - \tr_A(a) a + \norm_A(a) e = 0,
	\end{equation*}
where the trace and norm are defined in terms of the quadratic norm
form $q$ as
	\begin{eqnarray*}
		\tr_A(a) := q(a+e) - q(a) - q(e) &\textrm{and} & \norm_A(a):= q(a).
	\end{eqnarray*}

For a composition algebra $A$ over $K$ as above, the Hermitian matrix
algebra $\HH_n(A)$ consists of the $n \times n$ matrices $M =
(m_{ij}) \in \Mat_{n \times n}(A)$ with $M = \overline{M}^t$, or
equivalently, $m_{ji} = m_{ij}^\star$ for $1 \leq i, j \leq n$.  The
multiplicative structure of the algebra $\HH_n(A)$ is
commutative but not associative, given by
	\begin{equation} \label{eq:hermmultiplication}
		M \bullet M' := \frac{1}{2}(M \cdot M' + M' \cdot M)
	\end{equation}
where $\cdot$ denotes usual matrix multiplication.  Under this algebra structure, the
Hermitian matrices $\HH_n(A)$ form a Jordan algebra.

We now restrict to the case of cubic Jordan algebras given as a Hermitian matrix algebra for a composition algebra $A$.  There is one Jordan algebra structure on $\HH_3(A)$ inherited from the composition law
\eqref{eq:hermmultiplication}.  Also, on $\HH_3(A)$, we may
define a natural admissible cubic form $(\Norm, e, \sharp)$:
	\begin{align}
		\Norm \!\begin{pmatrix} c_1 & a_3 & a_2^\star \\ a_3^\star & c_2 & a_1 \\ a_2 & a_1^\star & c_3 \end{pmatrix}
		  &:= c_1 c_2 c_3 - c_1 \norm_A(a_1) - c_2 \norm_A(a_2) - c_3 \norm_A(a_3) + \tr_A(a_1 a_2 a_3) \label{eq:normJA} \\[.02in]
		  		\!e := \!\begin{pmatrix} 1 & & \\ & 1 & \\ & & 1 \end{pmatrix}\:\! & \,\nonumber \\[.0165in]
		\sharp: \!\begin{pmatrix} c_1 & a_3 & a_2^\star \\ a_3^\star & c_2 & a_1 \\ a_2 & a_1^\star & c_3 \end{pmatrix}\!\!\!
		  &\,\,\mapsto \!\begin{pmatrix}
		  	c_2 c_3 - \norm_A(a_1) & a_2^\star a_1^\star - c_3 a_3 & a_3 a_1 - c_2 a_2^\star \\
		  	a_1 a_2 - c_3 a_3^\star & c_1 c_3 - \norm_A(a_2) & a_3^\star a_2^\star - c_1 a_1 \\ 
		  	a_1^\star a_3^\star - c_2 a_2 & a_2 a_3 - c_1 a_1^\star & c_1 c_2 - \norm_A(a_3) 
		  \end{pmatrix} \label{eq:sharpJA}
	\end{align}
for $c_1, c_2, c_3 \in K$ and $a_1, a_2, a_3 \in A$.  For example, if
$A$ were commutative, then the norm form $\Norm$ is the usual
determinant of the matrix and the map $\sharp$ coincides with the
usual adjoint map for $3 \times 3$ matrices.  Springer's construction
then gives a Jordan algebra structure on $\HH_3(A)$ using this
admissible cubic form.

\begin{prop}[\mbox{\cite[\S 4.4]{mccrimmon}}]
	The two Jordan algebra structures on $\HH_3(A)$, as
        defined above, are the same.
\end{prop}

\subsubsection{Hermitian tensor spaces}

In this section, we describe in a basis-free manner the algebra $\HH_n(A)$ 
for {\em associative} composition algebras $A$ over
$K$, \ie composition algebras of dimensions $1$, $2$, or $4$ over $K$.
For such $A$, we introduce the notion of a {\em
  Hermitian tensor space} of an $A$-bimodule $\mathcal{M}$.  Just as
for symmetric and alternating tensor products, the idea is to
construct a subspace (or a quotient space) of a tensor space that
corresponds to Hermitian matrices.

If $\mathfrak{M}$ is an $A$-bimodule, then we define
$\mathfrak{M}^\star$ to be its twist by the involution $\star$ on $A$.
In other words, there is $K$-vector space (but not $A$-module)
isomorphism
	\begin{align*}
		\mathfrak{M} &\ra \mathfrak{M}^\star \\
		m &\longmapsto m^\star
	\end{align*}
but the left and right $A$-actions on $\mathfrak{M}^\star$ are given by
	\begin{align*}
		a(m^\star) = ((m)a^\star)^\star && \textrm{and} && (m^\star)a = (a^\star(m))^\star
	\end{align*}
for all $a \in A$ and $m \in \mathfrak{M}$. For any two $A$-bimodules,
the tensor product is again an $A$-bimodule; in our case, we have that
$\mathfrak{M} \tns_A \mathfrak{M}^\star$ is an $A$-bimodule with left
and right $A$-actions described by
	\begin{align*}
		a(m_1 \tns m_2^\star) = a(m_1) \tns m_2^\star && \textrm{and} && (m_1 \tns m_2^\star)a = m_1 \tns (m_2^\star)a = m_1 \tns (a^\star(m_2))^\star
	\end{align*}
for all $a \in A$ and $m_1, m_2 \in \mathfrak{M}$.  Note that elements
of $A$ can ``pass through'' the tensor product, \ie the relation
$(m_1)a \tns m_2^\star = m_1 \tns a(m_2^\star)$ holds for all $a \in
A$ and $m_1, m_2 \in \mathfrak{M}$.  There is a natural involution
$\tau$ on elements of $\mathfrak{M} \tns_A \mathfrak{M}^\star$ sending
	$$m_1 \tns m_2^\star \mapsto m_2 \tns m_1^\star.$$

\begin{defn} \label{def:hermtensor}
	The {\em Hermitian tensor space} $\Herm^2(\mathfrak{M})$ of
	$\mathfrak{M}$ is the sub-bimodule of $\mathfrak{M}~\tns_A~\mathfrak{M}^\star$
	consisting of elements $M$ satisfying $\tau(M) = M$.
\end{defn}

\begin{remark}
	One could also define a similar Hermitian tensor space as a
	quotient
		$$\mathfrak{M} \tns_A \mathfrak{M}^\star / \mathcal{I}$$
	where $\mathcal{I}$ is the submodule generated by all tensors
	of the form $m_1 \tns m_2^\star - m_2 \tns m_1^\star$ for $m,
	m' \in \mathfrak{M}$.  Over a field of characteristic not $2$,
	these notions are the same (just like for symmetric tensors).
	For our purposes, the subspace defined in Definition
  \ref{def:hermtensor} is more useful.
\end{remark}

For any $\mathfrak{M}$, there is a Segre-like map
	\begin{align} \label{eq:definesegre}
		\Seg: \mathfrak{M} &\ra \Herm^2(\mathfrak{M}) \subset \mathfrak{M} \tns_A \mathfrak{M}^\star \\
		m &\longmapsto m \tns m^\star \nonumber
	\end{align}
whose image consists precisely of the ``rank one'' tensors in $\mathfrak{M}
\tns_A \mathfrak{M}^\star$.  This is, of course, not a linear map.  In
fact, the right $A$-action on $\mathfrak{M}$ scales the image of
$\Seg$ by elements of the field $K$; in particular,
	\begin{equation} \label{eq:SegAaction}
		\Seg((m)a) = (m)a \tns ((m)a)^\star = (m)a \tns a^\star(m^\star) = \norm_A(a) (m \tns m^\star).
	\end{equation}

If $\mathfrak{M}$ is a free rank $r$ $A$-bimodule with a choice of
basis, then the Hermitian tensor space $\Herm^2(\mathfrak{M})$ is
visibly just the space $\HH_r(A)$ of $r \times r$ Hermitian
matrices over $A$, as the involution $\tau$ corresponds to taking the
conjugate transpose of a matrix.

In the sequel, let $\mathfrak{M}$ denote a free self-dual rank $3$
$A$-bimodule, and let $\JA := \Herm^2(\mathfrak{M})$.  Then $\JA$ also
has the structure of a cubic Jordan algebra as in \S
\ref{sec:compalgs}.  When we refer to $\JA$ for associative $A$, the module
$\mathfrak{M}$ will be assumed.  For an octonion algebra $A$ over $K$, the notation
$\JA$ will refer to the exceptional Jordan algebra $\HH_3(A)$.

\begin{remark}
	Because $\mathfrak{M}$ is assumed to be self-dual,
	we may view elements of the tensor space $\mathfrak{M} \tns_A
	\mathfrak{M}^\star$ as maps from $\mathfrak{M}^\vee \cong
	\mathfrak{M}$ to $\mathfrak{M}^\star$.  So there are also
	basis-free definitions of the norm, trace, and spur forms, as
	well as the adjoint map $\sharp$ in \eqref{eq:sharpmap}, for $\JA$.
\end{remark}

\subsubsection{Rank} \label{sec:rank}

Elements of $\JA$ inherit the notion of {\em rank} from the ambient
tensor space $\mathfrak{M} \tns \mathfrak{M}^\star$.  In this section, we discuss the stratification of
$\JA$ by rank, which is closely related to Severi varieties and their
tangent and secant varieties.

By forgetting the Jordan algebra and the $A$-module structure on
$\JA$, it makes sense to think of $\JA$ as a $K$-vector space of
dimension $3 \dim(A) + 3$.  The projective space $\PP(\JA)$ will denote
the space of $K$-lines in $\JA$ as a $K$-vector space and thus has
dimension $3 \dim A + 2$ over $K$.

Let $X_A \in \PP(\JA)$ correspond to the rank one elements of $\JA$,
so $X_A$ is cut out by quadrics in $\PP(\JA)$.  Let $Y_A \in \PP(\JA)$
consist of the elements of $\JA$ having rank at most two. Then $Y_A$ is visibly the secant
variety of $X_A$ in $\PP(\JA)$.  On the other hand, the variety $Y_A$
is defined by the cubic norm form $\Norm$ on $\JA$, so it is a cubic
hypersurface in $\PP(\JA)$.

If the composition algebra $A$ has dimension $d$ over $K$, then $X_A$,
$Y_A$, and $\PP(\JA)$ have dimensions $2d$, $3d+1$, and $3d+2$,
respectively.  Since $d = 1$, $2$, $4$, or $8$, in all these cases, the
secant variety $Y_A$ of $X_A$ is defective, and by a theorem of
Zak \cite{zak-book}, the secant variety is also the tangent variety.

\begin{example}
	Over an algebraically closed field of characteristic $0$, the varieties $X_A \subset \PP(\JA)$ for the
	four composition algebras $A$ (see Example \ref{ex:compalgsKbar})
	are exactly the four Severi varieties \cite{zak-severivarieties}:
		\begin{enumerate}[(i)]
			\item the Veronese surface $\PP^2 \subset \PP^5$,
			\item the Segre fourfold $\PP^2 \times \PP^2 \subset \PP^8$,
			\item the Grassmannian $\Gr(2,6) \subset \PP^{14}$, and
			\item the $16$-dimensional variety $E^{16} \subset \PP^{26}$ discovered by Lazarsfeld~\cite{lazarsfeldE16}.
		\end{enumerate}
	Over a general field $K$, the varieties that arise are twisted forms of these.
\end{example}

From the geometric perspective, the adjoint map $\sharp: \JA \to
\JA^\vee$ is essentially (up to scaling) the birational map
	\begin{equation} \label{eq:beta}
		\beta_A: \PP(\JA) \dashrightarrow \PP(\JA^\vee)
	\end{equation}
given by the linear system of quadrics passing through $X_A$, or in
other words, the partial derivatives of the norm form $\Norm$ on $\JA$
(see \cite{ein-shepherdbarron} or \cite{zak-severivarieties}).  By definition, the map $\beta_A$
blows down $Y_A$ to $X_A$, and the inverse blows up $X_A$ to $Y_A$, so $X_A$ is
naturally isomorphic to the dual variety of $Y_A$ and vice versa.

These varieties $X_A$ have a simple moduli interpretation, based on
their definition as rank one elements of $\JA$ up to $K$-scaling.

\begin{lemma}\label{compar}
  For any composition algebra $A$ of dimension $1, 2$, or $4$ over
  $K$, the variety $X_A$ parametrizes elements of $\mathfrak{M}$
  up to right $A$-scaling.
\end{lemma}

\begin{proof}
	The map $\Seg$ defined in \eqref{eq:definesegre} descends to a
        well-defined map
		\begin{equation*}
			(\mathfrak{M} \setminus \{0\}) / A \ra \PP(\JA)
		\end{equation*}
	by the computation in \eqref{eq:SegAaction}.  (Here
	$(\mathfrak{M} \setminus \{0\}) / A$ denotes nonzero elements
	of the module $\mathfrak{M}$ up to right $A$-scaling.)  The
	image of this map is by definition $X_A$, and it is easy to
	check that this map is injective.
\end{proof}

In fact, the variety $X_A$ for all composition algebras $A$ (including
those of dimension $8$) over an algebraically closed field of
characteristic $0$ is often considered an embedding of the projective
plane $\PP^2(A)$ over $A$ into $\PP(\JA)$ \cite{zak-severivarieties}.
Here, we work over a more general base field $K$, but the above
argument only holds for associative composition algebras.  For
octonion algebras $A$ over $K$, an explicit computation shows that the
variety $X_A$ here shares the same points as the usual definition of
an octonionic projective plane \cite[Chapter 12]{conway-quatoct},
which can also be described as right (or left) $A$-lines in $A^3$.

\subsubsection{Linear transformations of Jordan algebras}

Let $\SL(J)$ be the group of norm-preserving $K$-linear automorphisms
of a Jordan algebra $J$.  In this way, the Jordan algebra $J$ may
be thought of as a representation of the group $\SL(J)$.  In the case $J=\mathcal H_3(A)$,
we also denote $\SL(J)$ by $\SL_3(A)$.

Over an algebraically closed field $K = \overline{K}$, the cubic
Jordan algebras $\JA$ built from the four split composition algebras
$A$ correspond to the following groups and representations:
\begin{center}
	\begin{tabular}{c|c|c}
	$A$ & $\SL(\JA)$ &$\JA$ \\
	\hline
	$\mathscr{U}(K) = K$ & $\SL_3(K)$ & $\Sym^2(3)$  \\
	$\mathscr{B}(K) \cong K \times K$ & $\SL_3(K) \times \SL_3(K)$ & $3 \tns 3$\\
	$\mathscr{Q}(K) \cong \Mat_{2 \times 2}(K)$ & $\SL_6(K)$ & $\wedge^2(6)$ \\
	$\mathscr{O}(K)$ & $E_6(K)$ & $27$ \\
	\end{tabular}
\end{center}
Over a general field, we may consider various forms of these groups and the corresponding representations.

For the cubic Jordan algebras $\JA$ described in \S
\ref{sec:rank}, the cone on $X_A$ --- that is, the set of rank one
tensors --- in $\JA$ is exactly the orbit of the highest weight vector
of the representation $\JA$ of $\SL(\JA)$ (see \cite[Chapter
  3]{zak-book} for an explanation over an algebraically closed field).  
Over the algebraic closure, it is the unique closed orbit of the action of $\SL(\JA)$ on
$\PP(\JA)$.  Thus, the variety $X_A \subset \PP(\JA)$ is isomorphic
to the flag variety $\SL(\JA)/P$, where $P$ is the parabolic subgroup of $\JA$
that stabilizes the highest weight vector.

More generally, over the algebraic closure, the orbits of the action of $\SL(\JA)$ on $\JA$
give a stratification of $\JA$ by rank. We obtain another description of the rank two 
tensors of the representation $\JA$, which together form another orbit of $\SL(\JA)$ on $\PP(\JA)$.

Note that this description of $X_A$ automatically gives a
moduli interpretation of $X_A$, since it is a generalized flag variety;
the moduli interpretation given in Lemma~\ref{compar} is slightly
stronger, since it also includes the action of the composition algebra $A$ on the
vector bundle.  


\subsection{Doubly Hermitian Rubik's cubes} \label{sec:deg3moduli}

Our goal in this section is to study, in a uniform way, the orbits of
the representations $V \tns \JA$ of the group $G_\JA := \GL(V) \times
\SL(\JA)$, where $V$ is a three-dimensional vector space over the
field $K$ and $J(A)$ is the cubic Jordan algebra obtained from a composition algebra $A$ as
defined in \S \ref{sec:cubicjordan}.  
We will restrict ourselves to nondegenerate elements of
the tensor space, as these will correspond exactly to nonsingular curves.

\begin{defn}
	An element $\phi \in V \tns \JA$ is called {\em nondegenerate}
        if the induced composition map
	 $$\sharp \circ \phi : V^\vee \to \JA \to \JA^\vee$$
	is injective.
\end{defn}

We will show that nondegenerate elements of $V \tns \JA$ correspond to
genus one curves with extra data, including what we call an
$A$-line on the curve.  Intuitively, an $A$-line on a variety $Z$ is
like a rank one (left or right) $A$-module, if the notion of rank were
well defined for noncommutative rings.

\begin{defn}
	Let $A$ be a dimension $d$ composition algebra over $K$ and
	let $Z$ be a variety defined over $K$.  Then an {\em $A$-line}
	over $Z$ is a rank $d$ vector bundle $E$ over $Z$ with a
	global faithful right $A$-action.  
\end{defn}

We say that an $A$-line $E$ on $Z$ has {\em size $s$} if $E$ is a subbundle
of the trivial rank $ds$ bundle $B$ over $Z$, such that $B$ has
a global faithful right $A$-action that restricts on $E$ to
the given $A$-action on $E$.  A size $3$ $A$-line $E$ on $Z$ is {\em very ample} if there
is an immersion $\kappa: Z \to X_A$ such that the pullback of the
universal $A$-line on $X_A$ is isomorphic to $E$.  Finally, for a very ample size $3$
$A$-line $E$ on $Z$, we denote by $\lin E$ the line bundle on $Z$ given by pulling back 
$\OO_{\PP(\JA)}(1)$ to $Z$ via the composition $Z \stackrel{\kappa}{\ra} X_A \hookrightarrow \PP(\JA)$.
We will show that $\lin E$ is closely related to the determinant bundle $\det E$ (of $E$ as a
vector bundle), in each case.

From an element of $V \tns \JA$, we will construct a genus one curve with a degree~$3$ line bundle and an $A$-line.  This construction will be automatically invariant under the action of $G_\JA$.  In fact, this is
exactly the data that determines a $G_\JA$-orbit!

\begin{thm} \label{thm:hermRC}
	The nondegenerate $G_\JA$-orbits of $V \tns \JA$ are in
	bijection with isomorphism classes of nondegenerate triples $(C,L,E)$, where
	$C$ is a smooth genus one curve over $K$, $L$ is a degree~$3$
	line bundle on $C$, and $E$ is a very ample size $3$
	$A$-line over $C$ satisfying $\lin E \cong L^{\tns 2}$.
\end{thm}

The nondegeneracy condition for a triple $(C,L,E)$ will be discussed more in the proof; it is an open condition, so the theorem
may be rephrased as a bijection between the orbits of $V \tns \JA$ and the $K$-points of an open substack of the
moduli space of such triples (with the isomorphism between $\lin E$ and $L^{\tns 2}$).  For some choices of $A$, we will work out a relatively simple interpretation of this condition.

\subsubsection{Geometric construction}

We now show that a nondegenerate element $\phi$ of $V
\tns \JA$ naturally gives rise to the geometric data of a genus one
curve $C$, a degree $3$ line bundle on $C$, and an $A$-bundle $E$.
Let $d$ be the dimension of the composition algebra $A$ over $K$.

Given $\phi \in V \tns \JA$, we may instead think of $\phi$ as a
linear map in $\Hom(V^\vee,\JA)$.  Nondegeneracy of $\phi$ immediately
implies that this map is injective, so we obtain a linear map
	\begin{equation*}
	\PP(\phi) : \PP(V^\vee) \ra \PP(\JA).
	\end{equation*}
Let $W$ be the image of $\PP(\phi)$ in $\PP(\JA)$.  Then the secant
variety $Y_A$ of $X_A$ cuts out a cubic plane curve $C$ on $W$.  In
other words, the curve $\iota: C \hookrightarrow \PP(\JA)$ is defined by the vanishing of the cubic
norm form $\Norm$ on the plane $W$.  Let $L$ be the pullback of
$\OO(1)$ from the projective plane $W$ to $C$, so $L$ is a degree $3$
line bundle on $C$.

We claim that for $\phi$ nondegenerate, this curve $C$ is smooth and
irreducible, and thus of genus one.  Nondegeneracy of $\phi$ implies
that $W$ does not intersect the variety $X_A$, which is the base locus
for the rational map $\beta_A$ in \eqref{eq:beta}, or equivalently,
that the partials of the norm form $\Norm$ do not simultaneously
vanish.  In this case, the curve $C$ is nonsingular.

\begin{lemma}
	The image $C^\sharp$ of the curve $C \subset Y_A$ under the
	adjoint map $\beta_A$ is isomorphic to $C$, and hence is
	a smooth irreducible genus one curve in $X_A \subset \PP(\JA^\vee)$.
\end{lemma}

\begin{proof}
	Recall that $\JA$ and its dual may be naturally identified, and
	from \S \ref{sec:rank}, the adjoint map
	$\beta_A$ is a birational map from $\PP(\JA)$ to $\PP(\JA^\vee)$,
	which blows down $Y_A \setminus X_A$ to $X_A$.  The image of
	the plane $W$ under $\beta_A$ is birational on $W \setminus C$.  In other words, the generic
	fiber of $\beta_A$ restricted to $W$ is connected, so by
	Zariski's connectedness theorem, all the special fibers are
	connected; since there are no contractible curves in $W \cong
	\PP^2$, each fiber is a single point.  Thus, the image
	$C^\sharp$ of $C$ is also a smooth irreducible genus one
	curve, and it is contained in $X_A$.
\end{proof}

By the moduli interpretation of the variety $X_A$, the closed immersion
$C^\sharp \hookrightarrow X_A$ is equivalent to the data of a
very ample size $3$ $A$-line on $C^\sharp$, which pulls back to a very ample
size $3$ $A$-line $E$ on $C$.  In other words, we have
produced the triple $(C,L,E)$ as desired.

It is clear that the isomorphism class of the triple $(C,L,E)$ is preserved under
the action of $G_{\JA}$, where two such triples $(C,L,E)$ and $(C',L',E')$ are isomorphic
if there is an isomorphism $\sigma: C \to C'$ such that $\sigma^* L' = L$ and $\sigma^* E' = E$
and the $A$-actions on $E$ and $E'$ are related by $\sigma$.
The groups $\GL(V)$ and $\SL(\JA)$ act by linear transformations on $V$ and on $\JA$, respectively,
and the action of $\SL(\JA)$ fixes the varieties $X_A$ and $Y_A$.  Thus, both actions do 
not affect the geometric data, up to isomorphism.

Now we show that the triple $(C,L,E)$ satisfies the last condition of the theorem.

\begin{lemma}
	There is an isomorphism
		$$\lin E \cong L^{\tns 2}$$
	of line bundles on the curve $C$.
\end{lemma}

\begin{proof}
	This relation follows from the fact that the map $\beta_A$ is given by quadratic polynomials.
	The line bundle $\lin E$ on $C$ is the pullback of $\OO_{\PP(\JA)}(1)$ via
		\begin{equation*}
			C \stackrel{\beta_A}{\ra} C^\sharp \hookrightarrow X_A \hookrightarrow \PP(\JA^\vee) \cong \PP(\JA),
		\end{equation*}
	which is isomorphic to $\iota^*\OO_{\PP(\JA)}(2)$.  On the other hand, the
	line bundle $L$ is defined as the pullback of $\OO_W(1)$ to $C$, 
	and since $W$ lies linearly in $\PP(\JA)$, in fact $L$ is isomorphic
	to $\iota^*\OO_{\PP(\JA)}(1)$.
\end{proof}

This line bundle $\lin E$ on $C$ is closely related to the determinant bundle of $E$.
\begin{lemma}
	For a very ample size $3$ $A$-line $E$ on a projective variety $Z$, where $A$ is a composition
	algebra over $K$ of dimension $d$, if $\dim A = 1$, then 
	$$(\det E)^{\tns 2} \cong \lin E$$
	and if $\dim A = 2$, $4$, or $8$, then
	$$\det E \cong (\lin E)^{\tns (d/2)}.$$
\end{lemma}

\begin{proof}
	It suffices to prove this lemma for the variety $X_A$ itself.  Recall that $X_A$
	is a homogeneous variety in $\PP(\JA)$ given as the projectivization of the orbit
	of the highest weight vector of the representation $\JA$ of $\SL(\JA)$.  It is well known
	that the pullback of $\OO_{\PP(\JA)}(1)$ to $X_A$ is the product of the determinants of
	the vector bundles in the universal flag on $X_A$ \cite[Chapter 9]{fulton-youngtableaux}.
	Comparing the universal $A$-line $E$ on $X_A$ to the vector bundles in the universal flag
	under the typical moduli interpretation gives the lemma.  (For example, when $A$ is the split
	quaternions over $K$, our $A$-line $E$ is a rank $4$-vector bundle on $X_A = \Gr(2,6)$ and is the second
	tensor power of the typical universal rank $2$ bundle defined on the Grassmannian.)
\end{proof}

\subsubsection{Bijection}

The geometric construction described above gives one direction of the bijection.
We now prove Theorem \ref{thm:hermRC} (and its weaker version, Theorem \ref{thm:2hermRCpreview}).

\begin{proof}[Proof of Theorem $\ref{thm:hermRC}$]
	Suppose $(C,L,E)$ is a triple satisfying the conditions of the theorem.  We wish to recover
	a plane $W$ in $\PP(\JA)$ such that $C$ is isomorphic to the curve cut out by $W$ and
	the cubic hypersurface $Y_A$.
	
	The degree $3$ line bundle $L$ on $C$ gives a closed immersion $\eta: C \hookrightarrow \PP(H^0(C,L)^\vee)
	\cong \PP^2$.	  The $A$-line $E$ on $C$ implies that there is a closed immersion
		\begin{equation} \label{eq:kappa}
			\kappa: C \to X_A \hookrightarrow \PP(\JA) \cong \PP(\JA^\vee)
		\end{equation}
	given by the sections of $\lin E$.  More precisely, since $\lin E \cong \kappa^* \OO_{\PP(\JA)}(1)$
	and $\deg (\lin E) = 6$ by assumption, the image of $\kappa$ lies in a
	$\PP^5$ lying linearly in $\PP(\JA)$, namely the image of the map $\lambda: \PP(H^0(C,\lin E)^\vee) \hookrightarrow \PP(\JA)$.
	
	Recall that we have an isomorphism $L^{\tns 2} \cong \lin E$.  The multiplication map
		$$H^0(C,L) \tns H^0(C,L) \ra H^0(C,L^{\tns 2}) \cong H^0(C,\lin E)$$
	is surjective and factors through $\Sym^2 H^0(C,L)$, so a dimension count shows that
	the spaces $H^0(C,\lin E)$ and $\Sym^2 H^0(C,L)$ are naturally isomorphic.  Thus,
	we have a natural quadratic map
		\begin{equation}
			\rho: \PP(H^0(C,L)^\vee) \ra \PP(\Sym^2 H^0(C,L)^\vee) \stackrel{\cong}{\ra} \PP(H^0(C,\lin E)^\vee).
		\end{equation}
	Pulling back the line bundle $\OO_{\PP(\JA)}(1)$
	to the curve $C$ via the composition $\lambda \circ \rho \circ \eta$ and via $\kappa$
	give isomorphic line bundles on $C$, so under an appropriate change of basis, the
	images of $C$ in $\PP(\JA)$ via these two maps coincide.  In other words, the following diagram commutes:
		\begin{equation*}
		\raisebox{2\baselineskip}{ \xymatrix@C=0.6in{
			C \ar[d]^{\eta} \ar[r] & X_A \ar[d] \\
			\PP(H^0(C,L)^\vee) \ar[r]^-{\lambda \circ \rho}_-{\mathrm{quadratic}} & \PP(\JA)
		}} \qquad .\end{equation*}	

	The image of $\PP(H^0(C,L)^\vee)$ under $\rho$ is a Veronese surface $\mathscr{V}$
	in $\PP(H^0(C,\lin E)^\vee) \subset \PP(\JA)$.  The nondegeneracy condition we require
	is that this Veronese $\mathscr{V}$ is not contained in $Y_A$.  Under the inverse of the adjoint
	map $\beta_A$, we claim that $\mathscr{V}$ gives a $\PP^2$ lying linearly in $\PP(\JA)$.	
	Recall that $\beta_A \circ \beta_A$ is the identity on $\PP(\JA) \setminus Y_A$.
	By assumption, outside of the image of $C$ on $\mathscr{V}$, the map $\beta_A$ is birational,
	so by the valuative criterion of properness, there is a well-defined map from
	all of the surface $\mathscr{V}$ to $\PP(\JA)$, whose image is a linear subspace of $\PP(\JA)$.
	This plane may be identified with $\PP(H^0(C,L)^\vee)$ under some choice of basis.
	
	Note that the nondegeneracy condition on the triple $(C,L,E)$ is satisfied when constructed from
	an element of $V \tns \JA$.  Finally, the constructions in each direction are inverse to one another, 
	since $\beta_A \circ \beta_A$ is the identity on an open set of $\PP(\JA)$.
\end{proof}

\subsection{Specializations} \label{sec:deg3special}

For particular choices of $A$, Theorem \ref{thm:hermRC} recovers some of the spaces considered earlier,
while some other choices of $A$ give new parametrization theorems.  In this section, we describe
some of the cases where $A$ is split.

For example, the case where $A = K \times K$ and $d = 2$ recovers the space of Rubik's cubes studied
in \S \ref{sec:333}.  For $A = K \times K$, it is easy to check that the Jordan algebra $\JA$ is
isomorphic to the Jordan algebra $\Mat_{3 \times 3}(K)$ of $3 \times 3$ matrices
over $K$, where the norm, spur, and trace forms coming from the characteristic polynomial in the standard
way. The variety $X_A$ is isomorphic to the Segre fourfold $\PP^2 \times \PP^2 \hookrightarrow \PP^8$, and its secant
variety $Y_A$ is the cubic hypersurface given by the vanishing of the determinant.  Then the curve obtained as a intersection of a plane and 
$Y_A$ in $\PP(\JA)$ may be thought of as a determinantal variety, just as before.

The specialization of Theorem \ref{thm:hermRC} to $A = K$ and $d = 1$ gives the space of doubly symmetric
Rubik's cubes, which we studied in \S \ref{sec:2symRC}.  Here, the variety $X_A$
is the Veronese surface $\PP^2 \hookrightarrow \PP^5$.  The $A$-line $E$ is another degree $3$ line bundle, 
given as the pullback of $\OO(1)$ from the Veronese surface to the curve.  The squares of the
line bundles $L$ and $E$ are isomorphic, and thus their difference is a $2$-torsion point on the Jacobian of the curve,
as described in Theorem \ref{thm:2symRC}.

We next study Theorem \ref{thm:hermRC} in the case where $d = 4$ and the composition algebra $A$ is the algebra of 
split quaternions over $K$, i.e., the algebra $\Mat_{2 \times 2}(K)$ of $2 \times 2$ matrices over $K$, in order
to recover Theorem \ref{thm:2skewRCpreview}.
In this case, the algebra $\JA$ is isomorphic to the algebra of $6 \times 6$ skew-symmetric matrices
over $K$, where the cubic form is the degree $3$ Pfaffian of such a matrix.  The moduli problem
becomes one of determining so-called {\em Pfaffian representations}, which have been studied
over an algebraically closed field in \cite{buckley-pfaff}.

\begin{thm}
	Let $V$ and $W$ be $K$-vector spaces of dimensions $3$ and $6$, respectively.  Then nondegenerate
	$\GL(V) \times \SL(W)$-orbits of $V \tns \wedge^2 W$ are in bijection with isomorphism classes
	of nondegenerate triples $(C,L,F)$, where $C$ is a genus one curve over $K$, $L$ is a degree $3$ line bundle on $C$,
	and $F$ is a rank $2$ vector bundle on $C$, with $\det F \cong L^{\tns 2}$.
\end{thm}

The interpretation of $X_A$ as a moduli space of rank $2$ vector bundles is straightforward.
Here $X_A \hookrightarrow \PP(\JA)$ is the Pl\"{u}cker embedding of the Grassmannian $\Gr(2,W)$ in
$\PP(\wedge^2(W)) \cong \PP^{14}$.  The $A$-line $E$ coming from Theorem \ref{thm:hermRC} is a 
rank $4$ --- not rank $2$! --- vector bundle over $C$, but there is an equivalence of categories 
between modules $E$ over rank $4$ Azumaya algebras and rank $2$ vector bundles $F$.  This
phenomenon is special to the case of the split quaternion algebra; for nonsplit quaternion algebras, the minimal rank vector
bundle recovered from this data will have rank $4$.

We may also set $A$ to be the split octonion algebra $\mathscr{O}$ over $K$, although the data is certainly less familiar
to most people!  Then the algebra $\JA$ is the exceptional Jordan algebra, and the variety $X_A$ is the fourth Severi variety
found by Lazarsfeld \cite{lazarsfeldE16, zak-severivarieties}.  The variety $X_A$ is $16$-dimensional, and
it is often interpreted as a projective plane $\PP^2(\mathscr{O})$ over $\mathscr{O}$ \cite{compprojplanes}.  This interpretation
is exactly how we recover a $\mathscr{O}$-line on a curve $C$ from a map $C \to \PP^2(\mathscr{O})$.

\begin{thm}
	Let $V$ and $W$ be $K$-vector spaces of dimensions $3$ and $27$, respectively,
	with the split algebraic group $E_6$ acting on $W$.
	Then nondegenerate $\GL(V) \times E_6$-orbits of $V \tns W$ are in bijection with isomorphism classes of 
	nondegenerate quadruples $(C,L,\xi, E)$, where $C$ is a genus one curve over $K$, 
	$L$ is a degree $3$ line bundle on $C$, $\xi: C \to \PP^2(\mathscr{O}) \subset \PP(W)$
	is a map, and $E$ is the rank $8$ vector bundle on $C$ obtained from pulling back the universal bundle on $\PP^2(\mathscr{O})$,
	and $\det E \cong L^{\tns 8}$.
\end{thm}

One may also take $A$ to be a nonsplit quadratic algebra over $K$.  We plan to discuss some applications of these cases in future work.


\section{Representations associated to degree \texorpdfstring{$2$}{2} line bundles}
\label{sec:hermHC}

In this section, analogously to Section \ref{sec:hermRC}, 
we study a class of representations whose orbits are related to genus one curves with {degree $2$} line bundles. The main theorems in this section are summarized in Section~\ref{sec:HCpreview}.

We begin in \S\ref{sec:bideg22forms} by considering the space of bidegree $(2,2)$ forms on $\PP^1 \times \PP^1$.
We show that the orbits here
correspond to genus one curves equipped with a degree $2$ line bundle
and a nonzero point on the Jacobian.  (This is analogous to the interpretation of orbits for the space of Rubik's cubes!)
In~\S\ref{sec:hypercube}--\ref{sec:symhypercubes}, we then examine the space of hypercubes ($2 \times 2 \times 2 \times 2$ matrices) and some of its simpler variants; this space of hypercubes is the fundamental space for many degree $2$ moduli problems, just as the
space of Rubik's cubes, considered in the previous section, was the fundamental space for many degree $3$ moduli problems. 

In preparation for the general case, in \S\ref{subsec:Hcubes}
we then introduce the notion of ``triply Hermitian cubes'' with respect to a cubic Jordan algebra $J$, which form a vector space $\mathscr{C}=\mathscr{C}(J)$, and we describe a flag variety inside
this space up to scaling.  In~\S\ref{sec:deg2moduli}, given an element of the space $V \tns \mathscr{C}$ where $V$ is a two-dimensional vector space, we then 
construct genus one curves with a projection to that flag variety.  This yields bijections between orbits on $V\tns\mathscr{C}$ (i.e., the space of ``triply Hermitian hypercubes'') and isomorphism classes of genus one curves equipped with degree 2 line bundles and additional vector bundles.    
After uniformly treating the bijections for such spaces, 
we then specialize to several of the split cases, for which the geometric data becomes easier to describe; many of these are related to interesting arithmetic structures.

Many of the orbit problems described in
this section are used in \cite{cofreecounting} to determine average sizes of $2$-Selmer groups for certain families of elliptic curves over $\Q$.

\subsection{Bidegree \texorpdfstring{$(2,2)$}{(2,2)} forms} \label{sec:bideg22forms}

Let $V_1$ and $V_2$ be two-dimensional vector spaces over $K$.  A {\em
  $(2,2)$ form} $f$ over $K$ is an element of $\Sym^2 V_1 \tns \Sym^2
V_2$.  With a choice of bases $\{w_1,w_2\}$ and $\{x_1,x_2\}$ of $V_1$ and $V_2$,
respectively, such a form $f$ may be represented as a polynomial
	\begin{align} \label{eq:22form}
		f(w_1,w_2,x_1,x_2) = a_{22} w_1^2 x_1^2 &+ a_{32} w_1 w_2 x_1^2 + a_{42} w_2^2 x_1^2
												 + a_{23} w_1^2 x_1 x_2 + a_{33} w_1 w_2 x_1 x_2 \\ &+ a_{43} w_2^2 x_1 x_2
												 + a_{24} w_1^2 x_2^2 + a_{34} w_1 w_2 x_2^2 + a_{44} w_2^2 x_2^2 . \nonumber
	\end{align}
The group $\GL(V_1) \times \GL(V_2)$ acts on the space
of $(2,2)$ forms by the standard action on each factor.  We will also consider a
twisted action of $(g_1,g_2) \in \GL(V_1) \times \GL(V_2)$:
	\begin{equation*}
		(g_1,g_2) f(w,x) = \det(g_1)^{-1} \det(g_2)^{-1} f(g_1(w), g_2(x)).
	\end{equation*}
This is the representation $\Sym^2 V_1 \tns \Sym^2 V_2 \tns (\wedge^2 V_1 \tns \wedge^2 V_2)^{-1}$
of $\GL(V_1) \times \GL(V_2)$; by abuse of notation, we will refer to this as the twisted
action on $(2,2)$ forms.  This twisted action is not faithful; for example, the scalars $\Gm(V_i)$ of $\GL(V_i)$
act trivially on all $(2,2)$ forms.  Finally, the standard
scaling action of $\Gm$ on such bidegree $(2,2)$ forms $f$
will be relevant in the sequel.  The group $G$ for the moduli problem here will be the product of the scaling $\Gm$ and $\GL(V_1) \times \GL(V_2)$ acting by the twisted action.

\subsubsection{Geometric construction and bijection}

The $(2,2)$ form $f$ cuts out a bidegree $(2,2)$ curve $C$ in
$\PP(V_1^\vee) \times \PP(V_2^\vee)$.  If the curve $C$ is smooth, then a
standard computation shows that $C$ has genus $(2-1)(2-1) =
1$.  Pulling back line bundles via the embedding $\iota: C
\hookrightarrow \PP(V_1^\vee) \times \PP(V_2^\vee)$ gives two degree
$2$ line bundles on $C$, namely 
	\begin{align*}
		L_1 := \iota^* \OO_{\PP(V_1^\vee) \times \PP(V_2^\vee)}(1,0) && \textrm{and} &&
		L_2 := \iota^* \OO_{\PP(V_1^\vee) \times \PP(V_2^\vee)}(0,1).
	\end{align*}
Each of the projection maps $\mathrm{pr}_i : C \to \PP(V_i^\vee)$, for
$i = 1$ or $2$, is a degree $2$ cover of $\PP(V_i^\vee)$, ramified at
four points over a separable closure of $K$.  A binary quartic $q_1$
on $V_1$ associated to the ramification locus in $\PP(V_1^\vee)$
may be computed by taking the discriminant of $f$ as a quadratic polynomial on $V_2$:
	\begin{align} \label{eq:bqof22form}
		q_1(w_1,w_2) := \disc (f(x_1,x_2)) = &(a_{23} w_1^2 + a_{33} w_1 w_2 + a_{43} w_2^2)^2 \\
									 &- (a_{22} w_1^2 + a_{32} w_1 w_2 + a_{42} w_2^2)(a_{24} w_1^2 + a_{34} w_1 w_2 + a_{44} w_2^2), \nonumber
	\end{align}
and similarly for $q_2(x_1, x_2)$ as a binary quartic form on $V_2$.
The nonsingular genus one curve obtained from each of these binary quartics,
as in \S \ref{sec:binaryquartics}, is isomorphic to the curve $C$.  
Via those isomorphisms, the line bundles $L_1$ and $L_2$ coincide with 
the natural degree $2$ line bundles
on the genus one curves coming from these binary quartics.

We call a $(2,2)$ form $f$ or its associated curve $C$ {\em
  nondegenerate} if both of the associated binary quartics are
nondegenerate, \ie have four distinct roots over a separable closure.
For each of the binary quartics, this condition is given by the
nonvanishing of the discriminant $\Delta(q_i)$.  As the binary quartic
$q_i$ is invariant under the action of $\SL(V_j)$ on $f$, the
discriminant $\Delta(q_i)$ is a degree $12$ $\SL(V_i) \times
\SL(V_j)$-invariant for $f$.  Moreover, it is easy to check that
$I(q_1) = I(q_2)$ and $J(q_1) = J(q_2)$, so $\Delta(q_1) =
\Delta(q_2)$.  Thus, the polynomials $I(f) := I(q_i)$ and $J(f) :=
J(q_i)$ for $i=1$ or $2$ are 
$\SL(V_1) \times
\SL(V_2)$-invariants of $f$ having degrees $4$ and $6$, respectively.  The {\em discriminant}
$\Delta(f) = \Delta(q_i)$ of the $(2,2)$ form $f$ is a degree $12$
invariant, and a {\em nondegenerate} $(2,2)$ form is one with nonzero
discriminant.%
\footnote{Neither $J(f)$ nor $\Delta(f)$ is a generator for the ring of
$\SL(V_1) \times \SL(V_2)$-invariants of $\Sym^2 V_1 \tns \Sym^2 V_2$.  The invariant
ring is a polynomial ring with generators in degrees $2$, $3$, and $4$, and will be discussed
more carefully in \S \ref{sec:bideg22invtheory}.}
The nonvanishing of this discriminant is also equivalent to the condition that
the curve $C$ cut out by $f$ is nonsingular.

Thus, from a nondegenerate $(2,2)$ form $f$, we have constructed a
genus one curve in $\Poneone$, and the $G$-action preserves
the isomorphism class of this curve and the line bundles.  Conversely, given a genus one curve $C$ and two
degree $2$ line bundles $L_1$ and $L_2$ on $C$, there are natural
degree $2$ maps $\eta_i : C \to \PP(H^0(C,L_i)^\vee) = \Pone$ and the product map
	\begin{equation*}
	\xymatrix{
		(\eta_1,\eta_2) : C \ar[r]& \PP(H^0(C,L_1)^\vee) \times \PP(H^0(C,L_2)^\vee)
	}.
	\end{equation*}
If $L_1 \cong L_2$, then $(\eta_1,\eta_2)$ is a degree $2$ cover of a diagonal
in $\Poneone$, \ie the image of this map is isomorphic to $\Pone$.
Otherwise, we claim that this map is a closed immersion.

\begin{lemma} \label{lem:ijembedding}
For a smooth irreducible genus one curve $C$ and
  non-isomorphic degree $2$ line bundles $L_1$ and $L_2$ on $C$,
  the composition
		\begin{equation*}
			\xymatrix{
					\kappa: C \ar[rr]^-{(\eta_1,\eta_2)}&& \PP(H^0(C,L_1)^\vee) \times \PP(H^0(C,L_2)^\vee)
						\ar @{^{(}->}[r]^-{\mathrm{Segre}} & \PP(H^0(C,L_1)^\vee \tns H^0(C,L_2)^\vee)
			}
		\end{equation*}
	is a closed immersion.
\end{lemma}

\begin{proof}
	By Riemann-Roch, the spaces of sections $H^0(C,L_1),
  H^0(C,L_2)$, and $H^0(C,L_1 \tns L_2)$ have dimensions $2$, $2$,
  and $4$, respectively.  The multiplication map
		$$\mu_{12}: H^0(C,L_1) \tns H^0(C,L_2) \ra H^0(C,L_1 \tns L_2)$$
	is an isomorphism, by Castelnuovo's basepoint-free pencil trick
	(see \cite[p.\ 126]{acgh}).
	Since $\deg (L_1 \tns L_2) = 4$, the curve $C$ is isomorphic to its image 
	in $\PP(H^0(C,L_1 \tns L_2)^\vee) = \PP^3$.  Since $\kappa$ is the composition
	of this map to $\PP(H^0(C,L_1 \tns L_2)^\vee)$ with the isomorphism $\PP(\mu_{12}^\vee)$,
	it is a closed immersion.
\end{proof}

The image $C_{12}$ of the curve $C$ in $\PP(H^0(C,L_1)^\vee) \times
\PP(H^0(C,L_2)^\vee)$ is cut out by a $(2,2)$ form, via the exact
sequence 
$$0 \to \mathcal{I}_{C} \to \OO_{\PP(H^0(C,L_1)^\vee) \times \PP(H^0(C,L_2)^\vee)} \to \OO_{C} \to 0$$
where $\mathcal{I}_{C}$ is the ideal defining $C_{12}$.  Tensoring with $\OO(2,2)$, taking cohomology,
and tensoring by the dual of $H^0(\PP(H^0(C,L_1)^\vee) \times \PP(H^0(C,L_2)^\vee),\mathcal{I}_{C}(2,2))$
gives a map from $K$ to 
	\begin{align*}
	 H^0(\PP(H^0(C,L_1)^\vee) &\times \PP(H^0(C,L_2)^\vee),\OO(2,2)) \\
	 & \qquad \tns (\wedge^2(H^0(C,L_1)))^{-1} \tns (\wedge^2(H^0(C,L_2)))^{-1} \tns \omega_C,
	\end{align*}
where $\omega_C$ is the usual Hodge bundle for $C$.  We thus obtain a bidegree $(2,2)$ form, i.e.,
an element of $\Sym^2 (H^0(C,L_1)) \tns \Sym^2 (H^0(C,L_2)) \tns (\wedge^2(H^0(C,L_1)) \tns \wedge^2(H^0(C,L_2)))^{-1}$.  The factor $\omega_C$ fixes the scaling, just as in the cases in Section \ref{sec:classical}.

Thus, a genus one curve and two nonisomorphic degree $2$ line bundles $L_1$ and $L_2$ give rise to a $(2,2)$ form.

\begin{thm} \label{thm:22curvesbij}
	The nondegenerate $G$-orbits of $\Sym^2 V_1 \tns \Sym^2 V_2$ for two-dimensional vector spaces $V_1$ and $V_2$ are in bijection
	with isomorphism classes of triples $(C,L_1,L_2)$, where $C$ is a genus one curve and $L_1$
	and $L_2$ are nonisomorphic degree $2$ line bundles on $C$.
\end{thm}

The stabilizer of the $G$-action corresponds to the automorphism group of the triple $(C,L_1,L_2)$, which
for a generic nondegenerate $(2,2)$ form is the $K$-points of an extension of $\Jac(C)[2]$ by $\Gm^2$.
In general, the stabilizer consists of the $K$-points of a possibly non-split extension of this group
scheme by $\Aut(\Jac(C))$.

By the same argument as in Corollary \ref{cor:333CLP}, we may keep track of the difference $L_2 \tns L_1^{-1}$
instead of $L_2$.  The difference corresponds to a point on the Jacobian of $C$, and in fact, in the
period-index subgroup $\Jac^2_C(K)$.

\begin{cor} \label{cor:22formsCLP}
	The nondegenerate $G$-orbits of $\Sym^2 V_1 \tns \Sym^2 V_2$
	for two-dimensional vector spaces $V_1$ and $V_2$ are in bijection
	with isomorphism classes of triples $(C,L,P)$, where $C$ is a genus one curve, $L$
	is a degree $2$ line bundle on $C$, and $0 \neq P \in \Jac^2_C(K)$.
\end{cor}

\subsubsection{Invariant theory} \label{sec:bideg22invtheory}

The $\SL(V_1) \times \SL(V_2)$-invariants of the representation $\Sym^2 V_1 \tns \Sym^2 V_2$ form
a polynomial ring generated by invariants $\delta_2$, $\delta_3$, and $\delta_4$ of degrees $2$, $3$, and $4$
(see \cite{vinberg}, for example).  These same polynomials are relative invariants under the standard
action of $\GL(V_1) \times \GL(V_2)$, but are invariant under the twisted action of $\GL(V_1) \times \GL(V_2)$ described above.
When $f$ is a $(2,2)$ form given by \eqref{eq:22form}, we may take the generators to be
	\begin{align*}
		\delta_2 &= a_{33}^2 - 4 a_{32} a_{34} + 8 a_{24} a_{42} - 4 a_{23} a_{43} + 8 a_{22} a_{44} \\
		\delta_3 &= a_{24} a_{33} a_{42} - a_{23} a_{34} a_{42} - a_{24} a_{32} a_{43} + a_{22} a_{34} a_{43} + a_{23} a_{32} a_{44} - a_{22} a_{33} a_{44} \\
		\delta_4 &= I(f)
	\end{align*}
although any linear combination of $I(f)$ and $\delta_2^2$ is a degree $4$ generator of the invariant ring.  
Given a $(2,2)$ form $f$, these invariants are essentially described by the triple $(C,L,P)$ of Corollary \ref{cor:22formsCLP},
or in particular, the Jacobian $E$ of $C$ and the coordinates of $P$ on some form of $E$.

Recall that $I(f)$ and $J(f)$ are $\SL(V_1) \times \SL(V_2)$-invariants of degrees $4$ and $6$,
obtained from the binary quartics $q_i$ for $i = 1$ or $2$.  The Jacobian of the curve $C$
given by $f = 0$ is the Jacobian of the genus one curve associated to $q_i$, and it can be written
in Weierstrass form as
	\begin{equation*}
		y^2 = x^3 - 27 I(f) x - 27 J(f).
	\end{equation*}
The nonzero point $P$ has coordinates $(x(f),y(f))$ satisfying $E$, which are $\SL(V_1) \times \SL(V_2)$-invariants
of degrees $2$ and $3$, respectively, \ie scalar multiples of the generators $\delta_2$ and $\delta_3$ of the invariant ring!  The relation
  $$(108 \delta_3)^2 = (3 \delta_2)^3 - 27 I(f) (3 \delta_2) - 27 J(f)$$
shows that $(x(f),y(f)) = (3 \delta_2, 108 \delta_3)$.

We may also write the Jacobian of the genus one curve!$C$ in generalized Weierstrass form as
	\begin{equation} \label{eq:weierstrass234}
		y^2 + a_3 y = x^3 + a_2 x^2 + a_4 x
	\end{equation}
where the coefficients satisfy
	\begin{align*}
		a_2 = 9 \delta_2,  &&
		a_3 = 216 \delta_3, &&
		a_4 = 27 \delta_2^2 - 27 \delta_4,
	\end{align*}
and are different generators of the invariant ring (since we are working over a field $K$ not of characteristic $2$ or $3$).


\subsection{Hypercubes} \label{sec:hypercube}

We now consider $2 \times 2 \times 2 \times 2$ boxes, or hypercubes.  
This space is 
the fundamental representation for the degree $2$ cases,
and we will study a number of variants and generalizations in the subsections that follow.

The representation in question is
$V := V_1 \tns V_2 \tns V_3 \tns V_4$, where each $V_i$ is a $2$-dimensional $K$-vector space,
with the natural action by $G := \GL(V_1) \times \GL(V_2) \times \GL(V_3) \times \GL(V_4)$.
We prove the following theorem, where nondegeneracy corresponds to the nonvanishing of
a certain degree $24$ invariant, described in more detail below.

\begin{thm} \label{thm:hypercube}
  Let $V_1$, $V_2$, $V_3$, $V_4$ be $2$-dimensional vector spaces over $K$.
  Then nondegenerate $\GL(V_1) \times \GL(V_2) \times \GL(V_3) \times \GL(V_4)$-orbits
  of $V_1 \tns V_2 \tns V_3 \tns V_4$ are in bijection
  with isomorphism classes of quintuples $(C, L_1, L_2, L_3, L_4)$, where
  $C$ is a genus one curve over $K$, and the $L_i$ are degree $2$
  line bundles on $C$, satisfying $L_1 \tns L_2 \cong L_3 \tns
  L_4$ and $L_i \not\cong L_j$ for $i \neq j, 1 \leq i \leq 2, 1 \leq j \leq 4$.
\end{thm}

The stabilizer of a nondegenerate hypercube giving the genus one curve $C$ is exactly the automorphism
group of the quintuple, provided that we record the isomorphism $L_1 \tns L_2 \cong L_3 \tns L_4$.  In particular,
let $H$ be the extension of $\Jac(C)[2]$ by the kernel of the multiplication map 
$\Gm(V_1) \times \Gm(V_2) \times \Gm(V_3) \times \Gm(V_4) \to \Gm$, 
where each $\Gm(V_i)$ is the set of scalar transformations of $V_i$ for $1 \leq i \leq 4$.  Then the
stabilizer consists of the $K$-points of a possibly non-split extension of $\Aut(\Jac(C))$ by this group scheme $H$.
For example, if the $j$-invariant of $\Jac(C)$ is not $0$ or $1728$, and $C$ has a rational point $O$ and
full rational $2$-torsion with respect to $O$, generated by $P_1$ and $P_2$, then the hypercube corresponding to
the quintuple $(C,\mathcal{O}(2O),\mathcal{O}(O+P_1), \mathcal{O}(O+P_2), \mathcal{O}(O+P_1 + P_2))$
will have stabilizer group $H \times \ZZ/2\ZZ$, that is, an extension of $(\ZZ/2\ZZ)^3$ by $\Gm^3$.

\subsubsection{Geometric construction} \label{sec:HCgeom}

We describe how to construct the genus one curve and degree $2$ line bundles from a hypercube;
any $G$-equivalent hypercube will produce an isomorphic curve and line bundles.

Let $\AA \in V_1 \tns V_2 \tns V_3 \tns V_4$, so $\AA$ induces a linear map from 
$V_1^\vee \tns V_2^\vee$ to $V_3 \tns V_4$ and thus a linear map
$$\PP(V_1^\vee \tns V_2^\vee) \to \PP(V_3 \tns V_4).$$
There is a natural determinantal quadric in $\PP(V_3 \tns V_4)$; with
choices of bases for $V_3$ and $V_4$, it consists of nonzero $2 \times 2$ matrices,
up to scaling, which have determinant $0$. Let $C_{12}$ be the intersection of this
quadric with the image of the Segre map $\PP(V_1^\vee) \times \PP(V_2^\vee) \to \PP(V_1^\vee \tns V_2^\vee)$,
composed with the linear map given by $\AA$.  Then $C_{12}$ is generically
a curve.

More explicitly, the curve $C_{12}$ is a determinantal variety, given
by the vanishing of the determinant of a $2 \times 2$ matrix of bidegree $(1,1)$
forms on $\PP(V_1^\vee) \times \PP(V_2^\vee) = \Poneone$.  With choices of
bases for all the vector spaces $V_i$, the hypercube may be written as 
a $2 \times 2 \times 2 \times 2$ array $(a_{rstu})_{1 \leq r, s, t, u \leq 2}$ with
$a_{rstu} \in K$.  Then the curve $C_{12}$ is given as the vanishing of
the determinant of $\AA(w,x,\cdot,\cdot)$, which may be represented as the matrix
	\begin{small}
	$$\begin{pmatrix} a_{1111} w_1 x_1 + a_{1211} w_1 x_2 + a_{2111} w_2 x_1 + a_{2211} w_2 x_2 &
	 a_{1112} w_1 x_1 + a_{1212} w_1 x_2 + a_{2112} w_2 x_1 + a_{2212} w_2 x_2   \\
	 a_{1121} w_1 x_1 + a_{1221} w_1 x_2 + a_{2121} w_2 x_1 + a_{2221} w_2 x_2  &
	 a_{1122} w_1 x_1 + a_{1222} w_1 x_2 + a_{2122} w_2 x_1 + a_{2222} w_2 x_2  \end{pmatrix},$$
	 \end{small}%
where $\{w_1, w_2 \}$ and $\{x_1, x_2\}$ are the bases for $V_1^\vee$ and $V_2^\vee$, respectively.  This
determinant $f_{12}(w,x)$ is a $(2,2)$ form, i.e., an element of $\Sym^2 V_1 \tns \Sym^2 V_2$, and it is invariant
under the action $\SL(V_3) \times \SL(V_4)$.  One may similarly define the varieties $C_{ij}$ for
any $1 \leq i \neq j \leq 4$.  (We identify $C_{ij}$ and $C_{ji}$.)

We call a hypercube {\em nondegenerate} if the variety $C_{12}$ is smooth and one-dimensional, which
by \S \ref{sec:bideg22forms} is given by the nonvanishing of a $\SL(V_1) \times \SL(V_2)$ invariant
of degree $12$ in the coefficients of $f_{12}$ and thus degree $24$ in the hypercube.  This degree $24$
polynomial is invariant under all $\SL(V_i)$ for $1 \leq i \leq 4$.  By symmetry (or explicit computation),
this polynomial is the discriminant for all the $f_{ij}$, and we call it the {\em discriminant} of the
hypercube $\AA$ itself.  Nondegeneracy is preserved by the action of $G$.  In the sequel, we will only
work with nondegenerate hypercubes.

If $\AA$ is nondegenerate, then for all points $(w,x) \in C_{12}$, the matrix $\AA(w,x,\cdot,\cdot)$
is not the zero matrix and thus has rank exactly $1$.  (If it were the zero matrix, then all the partial
derivatives would vanish at the point $(w,x)$, so $C_{12}$ would not be smooth.)  So for all $(w,x) \in C_{12}$,
there is exactly one dimension of vectors $y \in V_3$ such that $\AA(w,x,y,\cdot) = 0$ (and similarly, one
dimension of vectors $z \in V_4$ with $\AA(w,x,\cdot,z) = 0$).

Given a nondegenerate hypercube $\AA$, it turns out that all of the resulting curves $C_{ij}$ are isomorphic!
To see this, define the variety
	$$C_{123} := \{(w,x,y) \in \PP(V_1^\vee) \times \PP(V_2^\vee) \times \PP(V_3^\vee) : \AA(w,x,y,\cdot) = 0 \}.$$
The projection $\pi_{123}^{12}$ of this variety $C_{123}$ onto $C_{12}$ is an isomorphism, where the inverse map $\rho_{12}^{123}$ is given by taking the point in $\PP(V_3^\vee)$ corresponding to the exactly one-dimensional kernel of $\AA(w,x,\cdot,\cdot)$.
By symmetry, the curves $C_{13}$ and $C_{23}$ are also isomorphic to $C_{123}$.  We may similarly define the curves
$C_{ijk} \subset \PP(V_i^\vee) \times \PP(V_j^\vee) \times \PP(V_k^\vee)$ for any $\{i,j,k\} \subset \{1,2,3,4\}$,
and they are all isomorphic to their projections to any two factors.
The $C_{ij}$ are all smooth irreducible isomorphic genus one curves, related by natural isomorphisms
	$$\tau_{ij}^{jk} : C_{ij} \stackrel{\rho_{ij}^{ijk}}{\ra} C_{ijk} \stackrel{\pi_{ijk}^{ik}}{\ra} C_{jk}$$
for $\{i,j,k\} \subset \{1,2,3,4\}$, where each $\pi$ is the projection and $\rho$ is the natural inverse ``un-projection'' map (where each map has domain given by the subscript and target given by the superscript).  We also obtain isomorphisms of the form
	$$\tau_{ijk}^{jkl} : C_{ijk} \to C_{jk} \to C_{jkl}$$
for $\{i,j,k,l\} = \{1,2,3,4\}$.  Finally, by the definition of these curves, we have natural projection maps
	\begin{align*}
	\pi_{ij}^i : C_{ij} \to \PP(V_i^\vee) \!\!\!\!&&\!\!\!\! \textrm{\!\!\!\!\!\!\!\!\!\!\!and\!\!\!\!\!\!\!\!\!\!\!} 
	\!\!\!\!&&\!\!\!\! \pi_{ijk}^i : C_{ijk} \to \PP(V_i^\vee). \!\!\!\!\!\!\!\!\!\!\!\!\!\!\!\!\!\!\!\!\!\!\!
	\end{align*}
It is clear that $\tau_{ijk}^{jkl}$ and $\tau_{jkl}^{ijk}$ are inverse
maps, as are $\tau_{ij}^{jk}$ and $\tau_{jk}^{ij}$.  However, composing more than two such
maps in sequence will not always give identity maps on these curves.
There are two types of interesting cycles we can obtain from composing the $\tau_{ijk}^{jkl}$ maps.  These are best exemplified
by arranging the curves $C_{ijk}$ in a tetrahedron as in \eqref{eq:tet}: there are four triangles (the faces
of the tetrahedron) and three four-cycles.

\begin{lemma} \label{lem:threecycle}
	The triangle of maps on each of the faces of the tetrahedron \eqref{eq:tet} 
	is not the identity map, but composing it twice gives the identity map.  In particular,
	the composition $C_{ijk} \to C_{ikl} \to C_{ijl} \to C_{ijk}$ is the hyperelliptic involution for $C_{ijk} \to \PP(V_i^\vee)$,
	sending any point $(w,x_1,y_1)$ to the point $(w,x_2,y_2)$, where
	$\{x_1, x_2\}$ and $\{y_1, y_2\}$ are the $($not necessarily distinct$)$ solutions to 
	$\AA \subs (w \tns x) = 0$ and $\AA \subs (w \tns y) = 0$, respectively.
\end{lemma}

\begin{proof}
Given a point $w \in \PP(V_1^\vee)$ not in the ramification locus of
the projection from $C_{12}$, there are two distinct points $x_1, x_2 \in
\PP(V_2^\vee)$ such that $\det \AA(w,x_i,\cdot,\cdot) = 0$.  For $i = 1$ or $2$, we have $\AA(w, x_i, y_i,\cdot) = 0$
for exactly one point $y_i \in \PP(V_3^\vee)$.   Since $\AA(w, x_2, y_1, z) = 0$
for some $z \in \PP(V_4^\vee)$, the linear form $\AA(w, \cdot, y_1 ,z)$
vanishes when evaluated at both $x_1$ and $x_2$, so it is identically
$0$; similarly, the linear form $\AA(w, x_2, \cdot, z)$ is identically zero.  Thus, we have the composition
	\begin{equation*}
	\raisebox{\baselineskip}{
	\xymatrix @R=0pt{ 
		C_{123} \ar[r]^{\tau_{123}^{134}} & C_{134} \ar[r]^{\tau_{134}^{124}} & C_{124} \ar[r]^{\tau_{124}^{123}} & C_{123} \\
		               (w, x_1,y_1)  \ar@{|->}[r] & (w,y_1,z) \ar@{|->}[r] & (w, x_2, z) \ar@{|->}[r] & (w, x_2, y_2).
	} 
	}
	\qedhere
	\end{equation*}
\end{proof}

Four-cycles of maps $\tau_{ijk}^{jkl}$ are also not the identity; we will show this by proving a relation among
degree $2$ line bundles defined on each of the curves.  For simplicity of notation, choose one curve, say $C_{12}$, to be the
primary curve under consideration.  This choice matters in the
definitions and constructions we will make in the sequel, but all
choices are equivalent.

Define four line bundles $L_i$ on $C_{12}$ by pulling back the line
bundle $\OO(1)$ from each $\PP(V_i^\vee)$.  Of course, it is important
through which maps we choose to pullback the bundle:
	\begin{align} \label{eq:C12linebundles}
		L_1 &:= (\pi_{12}^1)^* \OO_{\PP(V_1^\vee)}(1) \nonumber\\
		L_2 &:= (\pi_{12}^2)^* \OO_{\PP(V_2^\vee)}(1) \\
		L_3 &:= (\pi_{123}^3 \circ \rho_{12}^{123})^* \OO_{\PP(V_3^\vee)}(1) \nonumber\\
		L_4 &:= (\pi_{124}^4 \circ \rho_{12}^{124})^* \OO_{\PP(V_4^\vee)}(1). \nonumber
	\end{align}
That is, $L_1$ and $L_2$ come directly from the maps $C_{12} \to
\PP(V_1^\vee)$ and $C_{12} \to \PP(V_2^\vee)$, and for $i = 3$ or $4$, the
line bundle $L_i$ is pulled back via the simplest maps $C_{12} \to C_{12i} \to C_{2i} \to
\PP(V_i^\vee)$. Since all the curves $C_{ij}$ are defined
by bidegree $(2,2)$ equations, each of these line bundles on $C_{12}$
have degree $2$.

By Lemma \ref{lem:ijembedding}, the line bundles $L_1$ and $L_2$ are not isomorphic, since $C_{12}$
is a smooth irreducible genus one curve given by a nondegenerate $(2,2)$ form.
Similarly, since $C_{ij}$ is also a smooth irreducible genus one curve for
$i = 1$ or $2$ and $j=3$ or $4$, the line bundles
$(\tau_{ij}^{12})^* L_i = (\pi_{ij}^i)^* \OO_{\PP(V_i^\vee)}(1)$ and
$(\tau_{ij}^{12})^* L_j = (\pi_{ij}^j)^* \OO_{\PP(V_j^\vee)}(1)$ on $C_{ij}$ are not isomorphic,
so $L_i$ and $L_j$ are not isomorphic bundles on $C_{12}$.  Thus, the four line
bundles defined in \eqref{eq:C12linebundles} are all pairwise nonisomorphic, except
possibly $L_3$ and $L_4$.

\begin{lemma} \label{lem:HCrelation}
For the line bundles on $C_{12}$ defined above, we have the relation
  \begin{equation} \label{eq:HCrelation}
   L_1 \tns L_2 \cong L_3 \tns L_4.
  \end{equation}
\end{lemma}

\begin{proof}
	With a choice of a basis for $V_i$, points of the projective spaces
	$\PP(V_i^\vee)$ may be represented as $[a:b]$.  Let
	$\mathcal{D}(L)$ be the linear equivalence class of divisors
	corresponding to a line bundle $L$.  A representative $D_3$ of
        $\mathcal{D}(L_3)$ is (the formal sum of the points in) the preimage of a fixed point, say
        $[1:0]$, in $\PP(V_3^\vee)$, and similarly, we may choose a
        divisor $D_4$ in the class of $\mathcal{D}(L_4)$ as the
        preimage of $[1:0] \in \PP(V_4^\vee)$.  Let $\AA(w,x,\cdot,\cdot)$ be
        denoted by the matrix
		\begin{equation*}
			\begin{pmatrix} \AA_{11}(w,x) & \AA_{12}(w,x) \\ \AA_{21}(w,x) & \AA_{22}(w,x) \end{pmatrix} \in V_3 \tns V_4.
		\end{equation*}
	Then $D_3 + D_4$ is the sum of the four points that are
        solutions (counted up to multiplicity) of the system
  	\begin{equation*}
  	\left\{  \begin{matrix}
  		\AA_{11}(w,x) = 0 \\
  		\det \AA(w,x,\cdot,\cdot) = 0
  	\end{matrix} \right\}.
  	\end{equation*}
	Interpreted in another way, these four points of intersection
        are exactly the points of intersection of $C_{12}$ and the
        bidegree $(1,1)$ curve given by $\AA_{11}$ in $\PP(V_1^\vee)
        \times \PP(V_2^\vee)$.  Thus, the line bundle corresponding to
        the sum of these four points is just the pullback of
        $\OO_{\PP(V_1^\vee) \times \PP(V_2^\vee)}(1,1)$ to $C_{12}$;
        that is, $\OO(D_3 + D_4) \cong L_1 \tns L_2$, which is
        the desired relation.
\end{proof}

Using this relation among the line bundles, a computation shows that the
four-cycles in the tetrahedron \eqref{eq:tet} are not commutative.  For example, the composition map
$$\tau_{124}^{123} \circ \tau_{134}^{124} \circ \tau_{234}^{134} \circ \tau_{123}^{234} : C_{123} \ra C_{123}$$
is a nontrivial automorphism, given as a translation by the point 
$$P := L_3 \tns L_1^{-1} = L_2 \tns L_4^{-1} \in \Pic^0(C_{12}) \cong \Jac(C_{12}) \cong \Jac(C_{123}),$$
where the two isomorphisms are entirely canonical, so $P$ can be thought of as a point on $\Jac(C_{123})$.
The reverse four-cycle is the inverse map and is given by translation by
$-P$.  Similarly, the other four-cycles are given by translation by the points
$L_4 \tns L_1^{-1}$ and $L_2 \tns L_1^{-1}$ (up to sign).  Because of the relation in \eqref{eq:HCrelation},
these points (with the correct choice of sign) add up to $0$ on the Jacobian!  This may also be seen directly
from the tetrahedron picture, using the facts that each four-cycle decomposes as two consecutive three-cycles and
that each three-cycle composed with itself is the identity (by Lemma \ref{lem:threecycle}).  

We summarize these results in the following proposition:

\begin{prop} \label{prop:fourcycles}
	Given a nondegenerate hypercube $\AA$, we have the following statements,
	for any permutation $(i,j,k,l)$ of $(1,2,3,4)$:
	\begin{enumerate}
	\item[{\rm (i)}]
		The composition map
			\begin{equation*}
				\alpha_{ijkl} := \tau_{ijk}^{jkl} \circ \tau_{ijl}^{ijk} \circ \tau_{ikl}^{ijl} \circ \tau_{jkl}^{ikl} : C_{jkl} \ra C_{jkl}
			\end{equation*}
		is the automorphism of $C_{jkl}$ given by translation by the point
			$$P_{ijkl} := M_l \tns M_j^{-1} \in \Pic^0(C_{jl}) \cong \Jac(C_{jl}) \cong \Jac(C_{jkl})$$
		where $M_j = (\pi_{jl}^j)^*\OO_{\PP(V_j^\vee)}(1)$ and $M_l = (\pi_{jl}^l)^* \OO_{\PP(V_l^\vee)}(1)$ are degree $2$
		line bundles on $C_{jl}$.
	\item[{\rm (ii)}] We have $P_{ijkl} = - P_{ilkj}$, as $\alpha_{ijkl} \circ \alpha_{ilkj}$ is the identity map on $C_{jkl}$.
	\item[{\rm (iii)}]
		The points $P_{ijkl}$, $P_{iklj}$, and $P_{iljk}$ sum to $0$ on the Jacobian of $C_{jkl}$, so the composition
		of the automorphisms $\alpha_{ijkl}$, $\alpha_{iklj}$, and $\alpha_{iljk}$ in any order is the identity map
		on $C_{jkl}$.
	\end{enumerate}
\end{prop}


\subsubsection{Bijections}

Because the geometric constructions of the previous section are entirely $G$-invariant, we have already seen
that the $G$-orbit of a nondegenerate hypercube gives rise to a genus one curve and four line bundles (with a relation), up to isomorphism.  
In fact, we may construct a nondegenerate hypercube from such a curve, along with the
line bundles, which will prove Theorem \ref{thm:hypercube}.

\begin{proof}[Proof of Theorem $\ref{thm:hypercube}$]
Let $C$ be a genus one curve, and let $L_1$, $L_2$, $L_3$, $L_4$ be degree $2$ line bundles
on $C$ as in the statement of the theorem.  We first show how to construct a hypercube from this
data.

\begin{lemma} \label{lem:musurjective}
	Given a genus one curve $C$ and three non-isomorphic degree
	$2$ line bundles $L_1$, $L_2$, $L_3$ on $C$, the multiplication map $($\ie the cup product
	on cohomology$)$
  	\begin{equation*}
  		\mu_{123}: H^0(C,L_1) \tns H^0(C,L_2) \tns H^0(C,L_3) \ra H^0(C,L_1 \tns L_2 \tns L_3)
  	\end{equation*}
	is surjective, and its kernel may be naturally identified with the space of global
	sections $H^0(C,L_i^{-1} \tns L_j \tns L_k)$ for any permutation $\{i,j,k\}$ of $\{1,2,3\}$.
\end{lemma}

\begin{proof}
	Recall from the proof of Lemma \ref{lem:ijembedding} that the multiplication map
			$$\mu_{ij}: H^0(C,L_i) \tns H^0(C,L_j) \ra H^0(C,L_i \tns L_j)$$
	for two such line bundles is an isomorphism, due to the basepoint-free pencil trick.
	We apply the same trick again here: for any permutation $\{i,j,k\}$ of $\{1,2,3\}$,
	we tensor the sequence $0 \to L_i^{-1} \to H^0(C,L_i) \tns \OO_C \to L_i \to 0$
	with $L_j \tns L_k$ and take cohomology to obtain the exact sequence
		\begin{align} \label{eq:baseptfreeHC}
			0 &\to H^0(C,L_i^{-1} \tns L_j \tns L_k) \to H^0(C,L_i) \tns H^0(C,L_j \tns L_k)
					\to H^0(C,L_i \tns L_j \tns L_k) \nonumber \\
				&\to H^1(C,L_i^{-1} \tns L_j \tns L_k) = 0.
		\end{align}
	As the map $\mu_{123}$ factors through the isomorphism
		$$(\mathrm{id},\mu_{jk}) : H^0(C,L_i) \tns H^0(C,L_j \tns L_k) \to H^0(C,L_i \tns L_j \tns L_k),$$
	the sequence \eqref{eq:baseptfreeHC} shows that $\mu_{123}$ is surjective and its kernel may be naturally identified with
	$H^0(C,L_i^{-1} \tns L_j \tns L_k)$.
\end{proof}

Given $C$, $L_1$, $L_2$, $L_3$ as in the lemma, by Riemann-Roch, the kernel of $\mu_{123}$ has dimension $2$, and
we may use the inclusion of this kernel into the domain to specify a hypercube
	\begin{align*}
	\AA &\cong \Hom(\ker \mu_{123},H^0(C,L_1) \tns H^0(C,L_2) \tns H^0(C,L_3)) \\
	&\in H^0(C,L_1) \tns H^0(C,L_2) \tns H^0(C,L_3) \tns (\ker \mu_{123})^\vee
 	\end{align*}
where $V_i = H^0(C,L_i)$ for $1 \leq i \leq 3$ and $V_4 = (\ker \mu_{123})^\vee$.
We will show below that the hypercube $\AA$ thus constructed is nondegenerate and that
the geometric construction from $\AA$ gives a tuple isomorphic to the original $(C, L_1, L_2, L_3, L_4)$.

Let $C_{ij}'$ be the image of $C$
via the natural immersion into $\PP(H^0(C,L_i)^\vee) \times \PP(H^0(C,L_j)^\vee)$.
Let $C_{ij}$ be constructed from $\AA$ by
	$$C_{ij} := \{ (w,x) \in \PP(V_i^\vee) \times \PP(V_j^\vee) :
	 \det (\AA \subs (w \tns x)) = 0 \} \subset \PP(V_i^\vee) \times \PP(V_j^\vee).$$
We will show that these two varieties are the same for all $i \neq j$, but first for {\em some} $i \neq j$.

\begin{claim} \label{claim:CijCijprime} For some $1 \leq i \neq j \leq 3$, we have $C_{ij} = C_{ij}'$ as sets.
\end{claim}

\begin{proof}
For all $i \neq j$, the inclusion $C_{ij}' \subseteq C_{ij}$ is easy:  for each $1 \leq k \leq 3$,
let $\{r_{k1},r_{k2}\}$ be a basis of $H^0(C,L_k)$.  Then the definition of $\AA$ implies that
	\begin{equation*}
		\AA \subs \left( [r_{i1}(p):r_{i2}(p)] \tns [r_{j1}(p):r_{j2}(p)] \tns [r_{k1}(p):r_{k2}(p)] \right) = 0
	\end{equation*}
for all points $p \in C$, so $\left( [r_{i1}(p):r_{i2}(p)], [r_{j1}(p):r_{j2}(p)] \right)$ lies in $C_{ij}$.
Since $C_{ij}$ is defined by a bidegree $(2,2)$ equation $f_{ij}$ in $\PP(V_i^\vee) \times \PP(V_j^\vee)$,
if we show that $f_{ij}$ is nonzero and irreducible, then we find that $C_{ij}$ is a smooth irreducible genus one curve and thus $C_{ij} = C_{ij}'$.

An irreducible bidegree $(d_1,d_2)$ form on $\Poneone$ defines a genus $(d_1-1)(d_2-1)$ curve.
So if $f_{ij}$ is nonzero and factors nontrivially, then no irreducible component can be a smooth irreducible
genus one curve.  However, since $C_{ij}$ contains the smooth irreducible genus
one curve $C_{ij}'$, the polynomial $f_{ij}$ must be either zero
or irreducible for each pair $i \neq j$.  If $f_{ij} = 0$ identically, then
$C_{ij}$ is all of $\PP(V_i^*) \times \PP(V_j^*)$.  The projection of
	\begin{equation*}
		C_{123} := \{(w,x,y) \in \PP(V_1^\vee) \times \PP(V_2^\vee) \times \PP(V_3^\vee) : \AA(w,x,y,\cdot) = 0 \}
	\end{equation*}
to any $\PP(V_i^\vee) \times \PP(V_j^\vee)$ is exactly $C_{ij}$
by definition, and we will show that at least one of these projections is not two-dimensional.

Let $f$ and $g$ be the two tridegree $(1,1,1)$ equations defining $C_{123}$.  Because $\AA$
is defined by two linearly independent elements of $\ker \mu_{123}$, we have that $f$
and $g$ are nonzero and not multiples of one another.  If $\gcd(f,g) = 1$, then
$C_{123}$ is a complete intersection and thus a one-dimensional variety.  Otherwise,
suppose without loss of generality that $\gcd(f,g)$ has tridegree $(1,1,0)$ or $(1,0,0)$.
In either case, the projection to $\PP(V_1^\vee) \times \PP(V_2^\vee)$ is still one-dimensional.  
Therefore, there exists {\em some} $i \neq j$ such that $C_{ij}$ is not two-dimensional,
and thus must be exactly $C_{ij}'$.
\end{proof}

Since $f_{ij}$ cuts out a smooth irreducible genus one curve, we have
$\disc f_{ij} \neq 0$.  Thus, the hypercube $\AA$ has nonzero discriminant and is nondegenerate.

As $\disc(\AA) \neq 0$, the polynomials $f_{kl}$ do not vanish for any $k \neq l$,
and the $C_{kl}$ are smooth irreducible genus one curves.
It follows from the proof of Claim \ref{claim:CijCijprime} that
all of the $C_{kl}$ are in fact set-theoretically equal to $C_{kl}'$.  Moreover,
$C_{123}$ is set-theoretically equal to the image $C_{123}'$ of the
embedding of $C$ into the triple product space
$\PP(H^0(C,L_1)^\vee) \times \PP(H^0(C,L_2)^\vee) \times \PP(H^0(C,L_3)^\vee)$.
Because there is a canonical isomorphism $C_{123}' \to C_{123}$,
for $1 \leq i \leq 3$, the pullbacks of $\OO_{\PP(H^0(C,L_i)^\vee)}(1)$ to $C_{123}$ and then to $C$
are exactly the line bundles $L_i$.

From a genus one curve and three nonisomorphic degree $2$ line bundles on this curve,
we have constructed a nondegenerate hypercube.  This
hypercube, in turn, produces---via the constructions of \S \ref{sec:HCgeom}---an
isomorphic curve and the same line bundles.  Similarly, by the definitions of these
maps, going from $G$-orbits of  nondegenerate hypercubes to quintuples $(C,L_1,L_2,L_3,L_4)$
as in the theorem, and then back to $G$-orbits of hypercubes, is the identity map.
\end{proof}

We now rewrite the basic bijection of Theorem \ref{thm:hypercube}, by describing the
geometric data in a slightly different way.  This is analogous to the corollary following Theorem
\ref{thm:333bij}.  We simply replace the data of the line bundles $L_2$, $L_3$, $L_4$ by the differences
between each of them and $L_1$, which is a point on $\Pic^0(C) \cong \Jac(C)$.  
Since the sum of these differences (up to sign) is zero, it suffices
to keep track of two such points on the Jacobian of the curve, say $P : = L_2 \tns L_1^{-1}$ and
$P' := L_3 \tns L_2^{-1}$.  Recall (or see Appendix~\ref{appendix:torsors})
that the differences of two line bundles on a genus one curve are rational points in the appropriate period-index
subgroup of that curve.

\begin{cor} \label{cor:HCbijP}
Let $V_1$, $V_2$, $V_3$, $V_4$ be $2$-dimensional $K$-vector spaces.  Then the nondegenerate
$\GL(V_1) \times \GL(V_2) \times \GL(V_3) \times \GL(V_4)$-orbits of $V_1 \tns V_2 \tns V_3 \tns V_4$
are in bijection with isomorphism classes of quadruples $(C,L,P,P')$, where $C$ is a genus one curve
over $K$, $L$ is a degree $2$ line bundle on $C$, and $P$ and $P'$ are distinct nonzero points in $\Jac^2_C(K) \subset \Jac(C)(K)$.
\end{cor}

This also concludes the proof of Theorem \ref{hyperpar}, where we take $P''$ to represent the difference between $L_1$ and $L_3$.

\subsubsection{Invariant theory}

The $\SL(V_1) \times \SL(V_2) \times \SL(V_3) \times \SL(V_4)$-invariants of an element
of $V_1 \tns V_2 \tns V_3 \tns V_4$, for two-dimensional vector spaces $V_i$, form a polynomial
ring generated freely by $a_2$, $a_4$, $a_4'$, $a_6$ of degrees $2$, $4$, $4$, and $6$,
respectively \cite{littelmann}.  Just as in the previous cases considered, these invariants have
several interpretations in terms of each orbit's geometric data consisting of a genus one curve $C$,
a degree $2$ line bundle $L$, and nonzero points $P$, $P'$, and $P''$ in $\Jac^2_C(K)$ that sum to $0$.

One geometric interpretation of the generators of the invariant ring was discussed in \S \ref{sec:HCorbitpreview}: the Jacobian
of the genus one curve is given by
\begin{equation} \label{eq:EC812coeffs}
E: y^2 = x^3 + a_8 x + a_{12},
\end{equation}
where we have formulas for $a_8$ and $a_{12}$ in terms of $a_2$, $a_4$, $a_4'$, and $a_6$; then
$a_2$ is the slope of the line on which $P$, $P'$, $P''$ lie (on $E$); $(a_4, a_6)$ are the coordinates for the point $P$ on $E$;
and $a_4'$ is the $x$-coordinate for $P'$.

Another interpretation gives a model for the Jacobian elliptic curve with fixed points corresponding to $P$ and $P'$:

\begin{prop} \label{prop:HCinvsEC}
There exists a choice of normalization for the 
$\SL(V_1) \times \SL(V_2) \times \SL(V_3) \times \SL(V_4)$-invariants $\delta_2$, $\delta_4$, $\delta_4'$, $\delta_6$
such that given a nondegenerate tensor in $V_1 \tns V_2 \tns V_3 \tns V_4$ corresponding to $(C,L,P,P')$ as in
Corollary~$\ref{cor:HCbijP}$, the Jacobian of $C$ may be given in normal form as
	\begin{equation} \label{eq:HCEC}
		E: y^2 + \delta_4' y = x^4 + \delta_2 x^3 + \delta_4 x^2 + \delta_6 x
	\end{equation}
with identity point $(0,0)$, and the points $P$ and $P'$ correspond to the two points at infinity when homogenized.
\end{prop}

Just as for Proposition \ref{prop:333invthy}, straightforward proofs of both of these results are computational.  In this case, the invariants
are very reasonable to work with explicitly, and it is easy to show that the elliptic curve in \eqref{eq:HCEC} is isomorphic to the
Jacobian of the genus one curves constructed from the hypercube and that the points at infinity give the translations $\alpha_{ijk}$.  An
abstract proof of Proposition \ref{prop:HCinvsEC}, again like in Proposition \ref{prop:333invthy}, relies on the fact 
that elliptic curves with two distinct non-identity points may be written in the form \eqref{eq:HCEC}, 
so the coefficients $\delta_2$, $\delta_4$, $\delta_4'$, $\delta_6$ are relative 
$\GL(V_1) \times \GL(V_2) \times \GL(V_3) \times \GL(V_4)$-invariants of $V_1 \tns V_2 \tns V_3 \tns V_4$.

For any elliptic curve $E$ over $K$ of the form \eqref{eq:HCEC}, there always exists a $G$-orbit of $V$ where E is the Jacobian of the associated genus one curve, giving the analogous statement to Corollary \ref{cor:333surjorbits}:

\begin{cor} \label{cor:HCsurjorbits}
The map from nondegenerate orbits $V(K)/G(K)$ to elliptic curves of the form
	$$E: y^2 + \delta_4' y = x^4 + \delta_2 x^3 + \delta_4 x^2 + \delta_6 x$$
with $\delta_2$, $\delta_4$, $\delta_4'$, $\delta_6 \in K$, by
taking the Jacobian of the genus one curve associated to the orbit, is surjective.
\end{cor}

\subsection{Symmetric hypercubes} \label{sec:symhypercubes}

In this section, we study ``symmetrized'' hypercubes, as discussed in \S \ref{sec:symHCpreview}--\ref{sec:doubledoublesym}.  We show that nondegenerate orbits of the different types of symmetric hypercubes correspond to genus one curves with different numbers of degree $2$ line bundles.  Nondegeneracy for symmetric hypercubes is determined by the nonvanishing of the same degree~$24$ discriminant for hypercubes, although this discriminant factors differently in each case.

\subsubsection{Doubly symmetric hypercubes} \label{sec:2symHC}

The simplest case is that of doubly symmetric hypercubes, i.e., elements of the representation $V_1 \tns V_2 \tns \Sym_2V_3 \subset V_1 \tns V_2 \tns V_3 \tns V_3$ of $\GL(V_1) \times \GL(V_2) \times \GL(V_3)$, where the $V_i$ are $2$-dimensional $K$-vector spaces. With choices of bases for each vector space, the elements may be viewed as doubly symmetric hypercubes or as $2 \times 2$ matrices of binary quadratic forms.  Away from characteristic $2$, this is the same as the quotient representation $V_1 \tns V_2 \tns \Sym^2V_3$.

\begin{thm} \label{thm:2symHCL}
  Let $V_1$, $V_2$, $V_3$ be $2$-dimensional vector spaces over $K$.  Then
  nondegenerate $\GL(V_1) \times \GL(V_2) \times \GL(V_3)$-orbits of $V_1
  \tns V_2 \tns \Sym_2V_3$ are in bijection with isomorphism classes
  of quadruples $(C,L_1,L_2,L_3)$, where $C$ is a genus one curve over
  $K$, and $L_1$, $L_2$, $L_3$ are degree $2$ line bundles on $C$ satisfying
  $L_1 \tns L_2 \cong L_3^{\tns 2}$ and $L_3$ not isomorphic to
  $L_1$ or $L_2$.
\end{thm}

\begin{proof}
Given a nondegenerate element $\AA \in V_1 \tns V_2 \tns \Sym_2V_3 \subset V_1 \tns V_2 \tns V_3 \tns V_3$, we construct the genus one curve $C \subset \PP(V_1^\vee) \times \PP(V_2^\vee)$ and four degree $2$ line bundles $L_1$, $L_2$, $L_3$, $L_4$ on $C$ in the same way as in \S \ref{sec:hypercube}.  The symmetry implies that the line bundles $L_3$ and $L_4$ may be naturally identified, so we have from before that $L_1 \tns L_2 \cong L_3^{\tns 2}$ and $L_3$ is not isomorphic to $L_1$ or $L_2$.

Conversely, given such $(C,L_1,L_2,L_3)$ as in the theorem, we may set $L_4 = L_3$ and use the construction in the proof of Theorem \ref{thm:hypercube} to obtain a nondegenerate orbit of a hypercube $\AA \in V_1 \tns V_2 \tns V_3 \tns V_4$, where $V_i = H^0(C,L_i)$ for $i = 1, 2, 3$ and $V_4$ is the dual of the kernel of the multiplication map $\mu_{123}$.  From the proof of Theorem \ref{thm:hypercube}, we see that $V_4$ is also identified with $H^0(C,L_4)$, so the spaces $V_3$ and $V_4$ are naturally identified, say by $\psi_{43}: V_4 \to V_3$.  The maps from $C \cong C_{12} \in \PP(V_1^\vee) \times \PP(V_2^\vee)$ to $\PP(V_3^\vee)$ and $\PP(V_4^\vee)$ are both given by sections of the same line bundle $L_3$ and thus identical (after applying the identification $\psi_{43}$).

Now given a rank one element $B \in V \tns V$ for a $2$-dimensional $K$-vector space $V$, if $B(x, \cdot) = 0$ and $B(\cdot, x) = 0$ for some nonzero $x \in V^\vee$, then $B$ is in the subspace $\Sym_2 V$ of $V \tns V$.  Therefore, for any $(x,y) \in C_{12}$, the map $1 \tns \psi_{43}$ on $V_3 \tns V_4$ takes $\AA(x,y,\cdot,\cdot)$ to an element of $\Sym_2 V_3 \subset V_3 \tns V_3$.  And since $C_{12}$ spans $\PP(H^0(C,L_1)^\vee) \times \PP(H^0(C,L_2)^\vee)$, the map $1 \tns 1 \tns 1 \tns \psi_{43}$ on $V_1 \tns V_2 \tns V_3 \tns V_4$ sends $\mathcal{A}$ to an element of $V_1 \tns V_2 \tns \Sym_2 V_3 \subset V_1 \tns V_2 \tns V_3 \tns V_3$. 
\end{proof}

The two points $P$ and $P'$ referred to in Corollary \ref{cor:HCbijP} are now related; in particular, we have $P = 2 P'$.  Therefore, it suffices to keep track of only a single point $P'$, giving the following basis-free formulation of Theorem \ref{doublesympar}:

\begin{cor} \label{cor:2symHCP}
  Let $V_1$, $V_2$, $V_3$ be $2$-dimensional vector spaces over $K$.  Then
  nondegenerate $\GL(V_1) \times \GL(V_2) \times \GL(V_3)$-orbits of $V_1
  \tns V_2 \tns \Sym_2V_3$ are in bijection with isomorphism classes
  of quadruples $(C,L,P')$, where $C$ is a genus one curve over
  $K$, and $L$ is a degree $2$ line bundle on $C$, and $P'$ is a nonzero, non-$2$-torsion point
  in $\Jac^2_C(K)$. 
\end{cor}

The ring of $\SL(V_1) \times \SL(V_2) \times \SL(V_3)$-invariants of $V_1 \tns V_2 \tns \Sym_2 V_3$ is a polynomial ring generated in degrees $2$, $4$, and $6$ \cite{littelmann}.  Again, we may find several related geometric interpretations of these invariants.

Recall that there is a choice of rational generators $a_2$, $a_4$, $a_4'$, $a_6$ for the invariant ring of hypercubes such that the Jacobian of the genus one curve is given by \eqref{eq:EC812coeffs}, where $a_8$ and $a_{12}$ are given as in \eqref{eq:a8a12} and the two points on the Jacobian are $P = (a_4, a_6)$ and $P' = (a_4', a_6')$.  For the doubly symmetric hypercube, as we know from Corollary \ref{cor:2symHCP}, because we have $P = 2P'$, we compute that $2 a_4' = 9 a_2^2 - a_4$.  The expressions for $a_8$, $a_{12}$, and the discriminant~$\Delta$ also simplify significantly, when written in terms of $a_2$, $a_4'$, $a_6'$:
	\begin{align}
	a_8 &= -3 a_4'^2 + 2 a_2 a_6', \nonumber \\
	a_{12} &= 2 a_4'^3 - 2 a_2 a_4' a_6' + a_6'^2,\nonumber \\ \label{discstar}
	\Delta &= -16 a_6'^2 (-36 a_2^2 a_4'^2 + 108 a_4'^3 + 32 a_2^3 a_6' - 108 a_2 a_4' a_6' + 27 a_6'^2).
	\end{align}
Note that because of the factorization of $\Delta$, the nondegeneracy condition that we require is now actually the nonvanishing of a degree $6 + 12 = 18$ invariant.

A little bit of algebra shows that the Jacobian is isomorphic to the elliptic curve
\begin{equation} \label{eq:Jacfor2symHC}
E: y^2 + 2 a_2 x y + 2 a_6' y = x^3 + (3 a_4' - a_2^2) x^2
\end{equation}
where the point $P'$ is now at $(x,y) = (0,0)$.

\begin{prop}
There is a choice of normalization for the relative invariants $b_2$, $b_4$, $b_6$ for the space of doubly symmetric hypercubes such that
given a nondegenerate element of $V_1 \tns V_2 \tns Sym_2 V_3$ corresponding to $(C,L,P')$ as in Corollary~$\ref{cor:2symHCP}$, the Jacobian of $C$ may be given in generalized Weierstrass form as
\begin{equation} \label{eq:2symHCJacnormal}
E: y^2 + b_2 x y + b_6 y = x^3 + b_4 x^2
\end{equation}
with the point $P'$ being $(x,y) = (0,0)$.  Furthermore, the map from nondegenerate $\GL(V_1) \times \GL(V_2) \times \GL(V_3)$-orbits of this
representation to elliptic curves of the form \eqref{eq:2symHCJacnormal}, given by taking the Jacobian of the associated elliptic curve, is
surjective.
\end{prop}

\subsubsection{Triply symmetric hypercubes} \label{sec:3symHC}

Next, we study the space of triply symmetric hypercubes.  We may use the same methods as before to obtain a parametrization of the $\GL(V_1) \times \GL(V_2)$-orbits of the space $V_1 \tns \Sym_3 V_2 \subset V_1 \tns V_2 \tns V_2 \tns V_2$, for $2$-dimensional $K$-vector spaces $V_1$ and $V_2$.  As $K$ does not have characteristic $2$ or $3$ by assumption, this space is isomorphic to the quotient space $V_1 \tns \Sym^3 V_2$, and it may also be thought of as pairs of binary cubic forms.  The following is a basis-free version of Theorem \ref{triplesympar}:

\begin{thm} \label{thm:3symHCL}
  Let $V_1$ and $V_2$ be $2$-dimensional vector spaces over $K$.  Then
  nondegenerate $\GL(V_1) \times \GL(V_2)$-orbits of $V_1 \tns \Sym_3V_2$
  are in bijection with isomorphism classes of triples $(C,L,P)$,
  where $C$ is a genus one curve over $K$, $L$ is a degree $2$ line bundle on $C$,
  and $P$ is a nonzero $3$-torsion point of $\Jac(C)(K)$.
\end{thm}

\begin{proof}
From a nondegenerate element of $V_1 \tns \Sym_3V_2 \subset V_1 \tns V_2 \tns V_2 \tns V_2$, we obtain a genus one curve $C$ and four line bundles $L_1$, $L_2$, $L_3$, $L_4$ such that $L_1 \tns L_2 \cong L_3 \tns L_4$, just as for the usual hypercube.  For ease of exposition, we will sometimes refer to the second and third copies of $V_2$ as $V_3$ and $V_4$, respectively.  The symmetry clearly indicates that $L_3$ and $L_4$ are isomorphic.  While it is tempting to conclude that the symmetry also implies that $L_2$ is isomorphic to $L_3$ and $L_4$, recall that this cannot be so even in the usual hypercube case!  We instead have that the map from the curve $C$ to $\PP(V_1^\vee) \times \PP(V_2^\vee) \times \PP(V_3^\vee)$ is invariant under switching the latter two factors (where $V_2 = V_3$), so there must be an isomorphism $L_2 \tns L_3 \cong L_1^{\tns 2}$.  Thus, combining these relations and setting $P := L_2 \tns L_1^{-1}$ as a point on $\Jac(C)$, we have that $3P = 0$.  Note that because $2 \Jac(C)(K) \subset \Jac_C^2(K)$, all $3$-torsion points are in the degree $2$ period-index subgroup.  Therefore, from a nondegenerate triply symmetric hypercube, we obtain a genus one curve, a degree $2$ line bundle $L_1$, and a nonzero $3$-torsion point in $\Jac(C)(K)$.

Conversely, given such a triple $(C,L,P)$, we claim that we may construct a triply symmetric hypercube.  We may define the line bundles $L_1 = L$, $L_2 = L \tns P$ and $L_3 = L_4 = L \tns P \tns P$, and the usual construction produces a nondegenerate hypercube $\AA \in V_1 \tns V_2 \tns V_3 \tns V_4$ from $(C,L_1,L_2,L_3,L_4)$, where $V_1 = H^0(C,L_1)$, $V_2 = H^0(C,L_2)$, $V_3 = H^0(C,L_3)$, and $V_4$ is dual to the kernel of the multiplication map $\mu_{123}$.  By the same argument as in Theorem \ref{thm:2symHCL}, we may choose an appropriate identification of $V_3$ and $V_4$ such that $\AA$ lies in $V_1 \tns V_2 \tns \Sym_2 V_3$; in other words, $\AA$ is invariant under the transposition $(34)$ acting on the indices of the vector spaces $V_i$ for $i = 1$, $2$, $3$, $4$.

In fact, we may identify $V_2$ and $V_3$ as well.  Our $\AA$ gives rise again to a genus one curve isomorphic to $C$, but we may choose different line bundles to reconstruct the hypercube.  That is, by focusing on $C \hookrightarrow C_{14} \in \PP(V_1^\vee) \times \PP(V_4^\vee)$, we have line bundles $L_1$ and $L_4$, as before, which are the pullbacks of $\OO_{\PP(V_1^\vee)}(1)$ and $\OO_{\PP(V_4^\vee)}(1)$, respectively, to $C_{14}$.  The pullbacks of $\OO_{\PP(V_2^\vee)}(1)$ and $\OO_{\PP(V_3^\vee)}(1)$ to $C_{14}$ via $\rho_{14}^{124}$ and $\rho_{14}^{134}$, respectively, are now both isomorphic to $L_1 \tns P$.  Using the multiplication map $\mu_{124}$ to reconstruct the same hypercube $\AA$ gives a natural identification of $V_2$ and $V_3$ where the maps from $C_{14}$ to $\PP(V_2^\vee)$ and to $\PP(V_3^\vee)$ are identical.  Thus, we obtain an identification of $V_2$ and~$V_3$ such that the hypercube $\AA$ remains invariant under the transposition $(23)$.

Therefore, because $\AA$ is fixed under the transpositions $(23)$ and $(34)$, it is a triply symmetric hypercube in $V_1 \tns \Sym_3 V_2$, as desired.
\end{proof}

The $\SL(V_1) \times \SL(V_2)$-invariants for the space $V_1 \tns \Sym_3 V_2$ form a polynomial ring, generated by two polynomials $a_2$ and $a_6$ of degrees $2$ and $6$, respectively.  We may use our understanding of the invariant theory of normal hypercubes and of doubly symmetric hypercubes to explain these invariants geometrically.  In particular, the two degree $4$ invariants for hypercubes (and the one for doubly symmetric hypercubes) are now just $a_2^2/3$.  Substituting this relation into (\ref{discstar}) then gives that the Jacobian of the associated genus one curve $C$ has discriminant 
$$\Delta = 16 (4 a_2^3 - 27 a_6) a_6^3.$$
Thus, the condition that a triply symmetric hypercube is nondegenerate is given by the nonvanishing of a polynomial of degree $6 + 6 = 12$.
The Jacobian of $C$ may also be written in the form
$$E: y^2 + 2 a_2 x y + 2 a_6 y= x^3,$$
where $P$ is the $3$-torsion point at $(x,y) = (0,0)$.

\begin{prop}
There is a choice of normalization for the relative invariants $b_2$ and $b_6$ for the space of triply symmetric hypercubes such that
given a nondegenerate element of $V_1 \tns \Sym_3 V_2$ corresponding to $(C,L,P)$ as in Theorem $\ref{thm:3symHCL}$, the Jacobian
of $C$ may be given in generalized Weierstrass form as
\begin{equation} \label{eq:3symHCJacnormal}
E: y^2 + b_2 x y + b_6 y = x^3
\end{equation}
with the $3$-torsion point $P$ being $(x,y) = (0,0)$.  Furthermore, the map from nondegenerate $\GL(V_1) \times \GL(V_2)$-orbits of this
representation to elliptic curves of the form \eqref{eq:3symHCJacnormal}, given by taking the Jacobian of the associated elliptic curve, is
surjective.
\end{prop}

\subsubsection{Doubly doubly symmetric hypercubes, or bidegree \texorpdfstring{$(2,2)$}{(2,2)} forms again} \label{sec:22symHC}

We now study the subrepresentation $\Sym_2V_1 \tns \Sym_2V_2 \subset V_1 \tns V_1 \tns V_2 \tns V_2$ of $\GL(V_1) \times \GL(V_2)$, where $V_1$ and $V_2$ are $2$-dimensional $K$-vector spaces.  We call these doubly doubly symmetric hypercubes.  Away from characteristic $2$, this space is isomorphic to the representation $\Sym^2V_1 \times \Sym^2V_2$ of bidegree $(2,2)$ forms, which we examined in \S \ref{sec:bideg22forms} with a different interpretation.  

\begin{thm} \label{thm:22symHC}
Let $V_1$ and $V_2$ be $2$-dimensional vector spaces over $K$.  Then nondegenerate $\GL(V_1) \times \GL(V_2)$-orbits of $\Sym_2V_1 \times \Sym_2V_2$ are in bijection with isomorphism classes of triples $(C,L,P)$, where $C$ is a genus one curve over $K$, $L$ is a degree $2$ line bundle on $C$, and $P$ is a nonzero non-$2$-torsion point on $\Jac(C)(K)$.
\end{thm}

This statement is a basis-free version of Theorem \ref{doubledoublesympar}.  Note that the moduli problem for doubly doubly symmetric hypercubes is identical to that for doubly symmetric hypercubes!

\begin{proof}
Starting from an element of $\Sym_2V_1 \tns \Sym_2 V_2 \subset V_1 \tns V_1 \tns \Sym_2V_2$,
Corollary~\ref{cor:2symHCP} gives the triple $(C,L,P)$ as desired.

Conversely, given such $(C,L,P)$, recall from the proof of Theorem \ref{thm:2symHCL} that we may construct a hypercube $\AA \in U_1 \tns U_2 \tns \Sym_2U_3$, where $U_i := H^0(C,L_i)$ for
\begin{align*}
L_1 := L && L_2 := L \tns P \tns P && L_3 := L \tns P.
\end{align*}
From this hypercube $\AA$, we may create the usual curves $C_{ijk}$ in $\Pone \times \Pone \times \Pone$; for example, we have $C_{123} \subset \PP(U_1^\vee) \times \PP(U_2^\vee) \times \PP(U_3^\vee)$ with $L_i$ the pullback of $\OO_{\PP(U_i^\vee)}(1)$ to $C_{123}$.  Under the composition
$$C_{123} \to C_{13} \to C_{134} \to C_{34} \to C_{234} \to \PP(U_2^\vee),$$ 
the line bundle $\OO_{\PP(U_2^\vee)}(1)$ pulled back to $C_{123}$ is isomorphic to $L_1$ (using the relations of the form in Lemma \ref{lem:HCrelation}).  In particular, this gives a natural identification of the vector spaces $U_1 = H^0(C,L_1)$ with $U_2$!  Moreover, the two maps
\begin{align*}
\pi_{134}^1 \circ \rho_{34}^{134}: C_{34} \to C_{134} \to \PP(U_1^\vee) \\
\pi_{234}^2 \circ \rho_{34}^{234}: C_{34} \to C_{234} \to \PP(U_2^\vee)
\end{align*}
are the same after the identification of $U_1$ and $U_2$.  Therefore, there exists a choice of basis for $U_1$ and for $U_2$ such that $\AA$ is actually
invariant when switching the first and second factor; in other words, it is in the orbit of an element in $\Sym_2U_1 \tns \Sym_2U_3$.
\end{proof}

There also exists a straightforward computational proof, by exhibiting a linear transformation in $\GL(V_1)$ taking an element of $V_1 \tns V_1 \tns \Sym_2V_2$ to an element of $\Sym_2V_1 \tns \Sym_2V_2$.  This linear transformation has entries that are degree $3$ in the coefficients of the
original element; its determinant is nonzero for nondegenerate doubly symmetric hypercubes, as it is exactly the degree~$6$ invariant~$a_6'$ from \S \ref{sec:2symHC}, which appears as a factor of the discriminant for doubly symmetric hypercubes.  If we choose bases $\{u_1,u_2\}$ and $\{v_1,v_2\}$ for $V_1$ and $V_2$, respectively, we may represent a doubly symmetric hypercube as 
\begin{equation} \label{eq:2symHCex}
\sum_{i,j=1}^2 (r_{ij} v_1^2 + s_{ij} v_1 v_2 + t_{ij} v_2^2) u_i u_j.
\end{equation}
Then the linear transformation $(\begin{smallmatrix} a & b \\ c & d \end{smallmatrix}) \in \GL(V_1)$ with
\begin{align*}
a &= -r_{22} s_{21} t_{11} + r_{21} s_{22} t_{11} + r_{22} s_{11} t_{21} - r_{11} s_{22} t_{21} - r_{21} s_{11} t_{22} + r_{11} s_{21} t_{22} \\
b &= -r_{21} s_{12} t_{11} + r_{12} s_{21} t_{11} + r_{21} s_{11} t_{12} - r_{11} s_{21} t_{12} - r_{12} s_{11} t_{21} + r_{11} s_{12} t_{21} \\
c &= -r_{22} s_{21} t_{12} + r_{21} s_{22} t_{12} + r_{22} s_{12} t_{21} - r_{12} s_{22} t_{21} - r_{21} s_{12} t_{22} + r_{12} s_{21} t_{22} \\
d &= -r_{22} s_{12} t_{11} + r_{12} s_{22} t_{11} + r_{22} s_{11} t_{12} - r_{11} s_{22} t_{12} - r_{12} s_{11} t_{22} + r_{11} s_{12} t_{22}
\end{align*}
will send the doubly symmetric hypercube \eqref{eq:2symHCex} to a doubly doubly symmetric hypercube.

Note that this moduli interpretation for the orbits is very similar to
the one in Corollary \ref{cor:22formsCLP}.  However, the curve $X$
obtained in that corollary (and the previous Theorem
\ref{thm:22curvesbij}) has discriminant of degree $12$.  The curve $C$
here, from Theorem \ref{thm:22symHC}, has discriminant of degree $24$,
so they are clearly not the same --- but they are closely related.  In
particular, the curve $C$ is the (generalized) Hessian of the curve
$X$!  Here, we define the {\em Hessian} of a bidegree $(2,2)$ form
$f(w_1,w_2,x_1,x_2)$ in $\Sym^2 V_1 \tns \Sym^2 V_2$ (and by abuse of
terminology, the Hessian of the corresponding curve) as the curve cut
out by the determinant of the matrix
$$\left(\frac{\partial^2 f}{\partial w_i \partial x_j}\right)_{1 \leq i,j \leq 2}$$
which is also a bidegree $(2,2)$ form on $\Sym^2 V_1 \tns \Sym^2 V_2$.
It is a small computation to check that the genus one curve $C$ is the
Hessian of $X$.

The $\SL(V_1) \times \SL(V_2)$-invariant ring of $\Sym_2V_1 \times
\Sym_2V_2$ is generated by three invariants $\delta_2$, $\delta_3$,
$\delta_4$ of degrees $2$, $3$, and $4$, respectively, as we know from
\S \ref{sec:bideg22invtheory}.  We may apply our understanding of
the invariants in the doubly symmetric hypercube case; in particular,
we find that
\begin{align*}
\delta_2 &= - \frac{8}{3} a_2, &
\delta_3^2 &= \frac{4}{27}a_6', &
\delta_4 &= \frac{64}{9} a_2^2 - 16 a_4',
\end{align*}
where $a_2$, $a_4'$, and $a_6'$ are the polynomials from \S \ref{sec:2symHC}.  We can substitute these formulas into
\eqref{eq:Jacfor2symHC} to obtain a formula for the Jacobian of the
curve $C$ arising from an element $f$ of $\Sym_2V_1 \tns
\Sym_2V_2$ via Theorem \ref{thm:22symHC}.  Thus, there are
rational generators $b_2$, $b_3$, $b_4$ for the invariants such that
$\Jac(C)$ may be written in generalized Weierstrass form as
\begin{equation} \label{eq:22symHCJac}
y^2 + b_2 x y + 6 b_3^2 y = x^3 + b_4 x^2,
\end{equation}
and the point $P$ is at $(x,y) = (0,0)$.  The discriminant of the
Jacobian of $C$ factors as a rational multiple of
$$\delta_3^4 (-64 a_2^2 a_4'^2 + 192 a_4'^3 + 384 a_2^3 \delta_3^2 - 1296 a_2 a_4' \delta_3^2 + 2187 \delta_3^4),$$
where the second factor (up to a scalar) is the discriminant of the Jacobian of the
genus one curve $X$ cut out by $f$ directly.

A priori, it may not be obvious that any elliptic curve over $K$ with a nonzero
non-$2$-torsion point can be expressed in the form (\ref{eq:22symHCJac}).  To see this, recall that any such curve can be expressed in the form
$$y^2 + a x y + b y = x^3 + c x^2$$
where the discriminant is nonzero (and thus $b \neq 0$).  We then note that the latter elliptic
curve is actually isomorphic to one of the form \eqref{eq:22symHCJac} by taking
$b_2 = 6a/b$, $b_4 = 6/b$, and $b_6 = 36c/b^2$.

\subsubsection{Quadruply symmetric hypercubes, or binary quartic forms again} \label{sec:4symHC}

The last symmetrization that we study is the case of fully symmetric hypercubes, i.e., elements of $\Sym_4 (V) \subset V^{\tns 4}$ for a $2$-dimensional
$K$-vector space $V$.  Since our field $K$ is not of characteristic $2$ or $3$, this space is isomorphic to the quotient space $\Sym^4 V$ of binary quartic forms.  Here, we take the standard $\GL(V)$-action on $\Sym_4(V)$, which is slightly different than the action considered on binary quartic forms in \S \ref{sec:binaryquartics}.  The following is a basis-free version of Theorem \ref{sympar}:

\begin{thm} \label{thm:4symHC}
For a $2$-dimensional $K$-vector space $V$, nondegenerate $\GL(V)$-orbits of $\Sym_4 (V)$ are in bijection with isomorphism classes of triples $(C,L,P)$, where $C$ is a genus one curve over~$K$, $L$ is a degree $2$ line bundle on $C$, and $P$ is a nonzero $3$-torsion point of $\Jac(C)(K)$.
\end{thm}

This moduli problem is identical to that of triply symmetric hypercubes!

\begin{proof}
Given an element of $\Sym_4(V) \subset V \tns \Sym_3V$, we may apply Theorem \ref{thm:3symHCL} to obtain a genus one curve $C$ with a degree $2$ line bundle $L$ and a nonzero $3$-torsion point $P$ of $\Jac(C)(K)$.

Conversely, given such a triple $(C,L,P)$, we may use Theorem \ref{thm:3symHCL} to construct a hypercube $\AA$ in $U_1 \tns U_2 \tns U_3 \tns U_4$, where $U_1 = H^0(C,L)$ and $U_2 = H^0(C, L \tns P)$ has a natural identification with $U_3$ and $U_4$ such that $\AA$ is invariant under permutations of $U_2$, $U_3$, and $U_4$.  An almost identical argument to the one in Theorem \ref{thm:22symHC} shows that we may in fact identify $U_1$ with $U_2$ and that $\AA$ is invariant under switching $U_1$ and $U_2$.  In other words, the hypercube $\AA$ coming from $(C,L,P)$ is in the orbit of a quadruply symmetric hypercube.
\end{proof}

A simple computational proof, just as for Theorem \ref{thm:22symHC}, is also possible; we only need to specify an element of $\GL(V)$ that acts on
a given nondegenerate element of $V \tns \Sym_3 (V)$ (via the first factor only) to produce an element of $\Sym_4(V)$.  With a choice of basis $\{v_1, v_2\}$ for $V$, we may represent an element of $V \tns \Sym_3V$ as
\begin{equation} \label{eq:3symHCex}
\sum_{i=1}^2 v_i \tns (r_i v_1^3 + s_i v_1^2 v_2 + t_i v_1 v_2^2 + u_i v_2^3).
\end{equation}
Then applying the linear transformation $(\begin{smallmatrix} a & b \\ c & d \end{smallmatrix}) \in \GL(V)$,
with
\begin{align*}
a &= -s_2^2 t_1 + s_1 s_2 t_2 + r_2 t_1 t_2 - r_1 t_2^2 - r_2 s_1 u_2 + r_1 s_2 u_2 \\
b &= s_1 s_2 t_1 - r_2 t_1^2 - s_1^2 t_2 + r_1 t_1 t_2 + r_2 s_1 u_1 - r_1 s_2 u_1 \\
c &= s_2 t_1 t_2 - s_1 t_2^2 - s_2^2 u_1 + r_2 t_2 u_1 + s_1 s_2 u_2 - r_2 t_1 u_2 \\
d &= - s_2 t_1^2 + s_1 t_1 t_2 + s_1 s_2 u_1 - r_1 t_2 u_1 - s_1^2 u_2 + r_1 t_1 u_2,
\end{align*}
to the first factor of $V$, gives a quadruply symmetric hypercube.  The determinant of this transformation is just the degree $6$ invariant $a_6$
for triply symmetric hypercubes (from \S \ref{sec:3symHC}), and since it is a factor of the discriminant, it is nonzero for nondegenerate
hypercubes.

Just as for the triply symmetric Rubik's cubes and the doubly doubly symmetric hypercubes, this parametrization of binary quartic forms is related to the other ``dual'' parametrization for the same space, namely the one described in \S \ref{sec:binaryquartics}.  In Theorem~\ref{thm:bqorbit}, the genus one curve $X$ arising from a binary quartic form has discriminant of degree $6$.  Here, in Theorem~\ref{thm:4symHC}, we produce a genus one curve $C$ whose discriminant has degree $24$.  But these two curves $X$ and $C$ are again related by the Hessian construction, i.e., $C$ is the Hessian of $X$!  More precisely, let $q(w_1,w_2)$ be a binary quartic form in $\Sym^4 V$. Then the {\em Hessian} of $q$ (or of the genus one curve $X$ given as $y^2 = q(w_1,w_2)$) is the binary quartic
	$$H(q)(w_1,w_2) := \disc \left( \frac{\partial^3 q}{\partial w_i \partial w_j \partial w_k} \right)_{1 \leq i, j, k \leq 2},$$
i.e., the discriminant of the three-dimensional matrix of triple derivatives.  (Recall from \S \ref{sec:cubes} that the discriminant is the unique polynomial invariant of a $2 \times 2 \times 2$ cube.)   By abuse of terminology, we also say that the genus one curve $C$ associated to $H(q)$ is the Hessian of $q$ and of $X$.  This curve $C$ is the one obtained from Theorem \ref{thm:4symHC}.  In particular, we obtain a proof of the following:

\begin{cor}
Given a binary quartic form $q$ of $\Sym^4 V$, where $V$ is a $2$-dimensional $K$-vector space, let $H(q)$ denote the Hessian binary quartic form.  Then the Jacobian of the genus one curve given by the equation $y^2 = H(q)$ has a nonzero $3$-torsion point defined over $K$.
\end{cor}

Recall from \S \ref{sec:binaryquartics} that the $\SL(V)$-invariants of $\Sym^4 V$ are generated by two invariants $I$ and $J$ of degrees $2$ and $3$, respectively.  These invariants appear in the coefficients for the Jacobian of the curve $C$ arising from an element $\AA \in \Sym_4 (V)$ via Theorem \ref{thm:4symHC}.  In particular, because the space $\Sym_4 (V)$ is contained in both $V \tns \Sym_3 V$ and $\Sym_2 V \tns \Sym_2 V$, we may use the geometric interpretations of the invariants of triply symmetric hypercubes and of doubly doubly symmetric hypercubes to easily understand the invariants in this case!  The Jacobian of $C$ is isomorphic to
\begin{equation} \label{eq:4symHCJac}
E: y^2 + 2 a_2 x y + 216 a_3^2 y = x^3,
\end{equation}
where $a_2$ is defined as for hypercubes and $a_3$ is the primitive integral degree $3$ invariant generator.  In terms of the invariants $I$ and $J$, because $\AA$ viewed as a binary quartic form has coefficients with factors of $4$ and $6$, the polynomials $I(\AA)$ and $J(\AA)$ are not primitive: $I(\AA) = -4 a_2$ and $J(\AA) = -432 a_3$.  The discriminant of the Jacobian \eqref{eq:4symHCJac} factors as a rational multiple of
$$a_3^6 (a_2^3 + 729 a_3^2),$$
and the latter factor is just the discriminant of the Jacobian of $Z$ (or of $\AA$ as a binary quartic form).

Finally, just as before, we observe that any elliptic curve over $K$ with a nonzero $3$-torsion point defined over $K$ can be expressed in the form \eqref{eq:4symHCJac}.  Namely, if we have an elliptic curve of the form
$y^2 + a x y + b y = x^3$
with nonzero discriminant (implying $b \neq 0$),
then setting $a_2 = 108 a / b$ and $a_3 = 216 / b$ gives an isomorphic elliptic curve of the form (\ref{eq:4symHCJac}).


\subsection{Hermitian cubes}\label{subsec:Hcubes}

In this section, we discuss spaces $\mathscr{C}$, sometimes called ``Freudenthal algebras,''
or ``Freudenthal triple systems,'' that have a quartic norm form, originally introduced in \cite{freudenthal-magic}.  These vector spaces are related to the space of $2 \times 2 \times 2$ cubes---that is, the tensor space $V_1 \tns V_2 \tns V_3$
where $V_i$ is a two-dimensional vector space over $K$---in the same way that the spaces of Hermitian matrices
over composition algebras are related to the usual matrix algebras.  Our goal here, then, is to
``triply Hermitianize'' the space of $2 \times 2 \times 2$ cubes with respect to a cubic algebra.

The natural cubic algebras to use are cubic Jordan algebras $J$.  In fact, this process will work for any
cubic Jordan algebra with a nondegenerate trace form (see \S \ref{sec:springer} for general constructions).
Some of the properties and formulas we use below for
Hermitian cubes are explained in more detail in \cite{krutelevich} and also rely on ideas from \cite{faulkner, clerc}.

Analogously to \S \ref{sec:cubicjordan}, we will describe spaces $\mathscr{C}$ of Hermitian $2\times 2\times 2$ cubical matrices and their properties, including
a norm form, a trace form, an adjoint map, and a notion of rank.  We will then be interested in the rank one
loci of such spaces and their moduli descriptions.  
We will use such moduli
interpretations to obtain vector bundles on varieties mapping to these rank one loci.
This will allows us 
in \S \ref{sec:deg2moduli}
to uniformly study representations of the form $V\otimes \mathscr{C}$, where $V$ is a $K$-vector space of dimension~2, in terms of genus one curves.

\subsubsection{Definitions and invariants}

\begin{defn}
	A {\em Hermitian cube space} $\CJ$ over a cubic Jordan algebra $J$ is a vector space of elements of the form
		
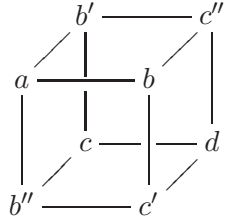
\begin{equation*}
		\xymatrix@!0{
			& b' \ar@{-}[rr]\ar@{-}'[d][dd] & & c'' \ar@{-}[dd] \\
			a \ar@{-}[ur]\ar@{-}[rr] \ar@{-}[dd] & & b \ar@{-}[ur]\ar@{-}[dd] \\
			& c \ar@{-}'[r][rr] & & d \\
			b'' \ar@{-}[rr]\ar@{-}[ur] & & c' \ar@{-}[ur]
		} \end{equation*}
	where $a, d \in K$ and $(b,b',b'')$ and $(c,c',c'')$ are conjugates in $J$.  Addition and multiplication
	by elements of $K$ occurs componentwise.  We will
	abbreviate elements, or {\em Hermitian cubes}, in $\CJ$ as $(a,b,c,d)$.
\end{defn}

\begin{example} \ \begin{enumerate}[(i)]
	\item
		If $J$ is the field $K$ itself with cubic norm form $\Norm(x) = x^3$ for $x \in K$, the conjugates of
		any $x \in K$ are both just $x$, so the space $\CJ$ is isomorphic to $\Sym^3K^2$.
	\item
		If $J$ is the split algebra $K \times K \times K$ with norm form $\Norm((x_1,x_2,x_3)) = x_1 x_2 x_3$, then
		$\CJ$ is isomorphic to $K^2 \tns K^2 \tns K^2$, or the space of $2 \times 2 \times 2$ cubes.
\end{enumerate} \end{example}

In general, if $J$ has dimension $d$ over $K$, then $\CJ$ is a $K$-vector space of dimension $2d + 2$.
Although there is a very weak algebra structure on $\CJ$, we will not use it in the sequel.  We will only
use the structure of $\CJ$ as a representation of a certain reductive group $G_\CJ$, which will be
a prehomogeneous vector space with a relative invariant of degree $4$.  This invariant will be
the norm form for Hermitian cubes.

Recall that a cubic Jordan algebra $J$ comes equipped with a trace form $\Tr$ and a cubic norm form $\Norm$,
as well as their (bi)linearizations (see equation \eqref{eq:jordanalgforms} in \S \ref{sec:springer}).
We restrict our attention to cubic Jordan algebras for which the trace form $\Tr$ is nondegenerate,
in which case we obtain a natural adjoint map $\sharp$, as in equation \eqref{eq:sharpmap}.
We will use these properties of $J$ to construct a trace form $\Tr_\CJ$ and quartic norm
form $\Norm_\CJ$ for $\CJ$, given the basepoint $\varepsilon := (1,0,0,1)$. 

For a $3 \times 3$ Hermitian matrix over a composition algebra $A$, the cubic norm form is a generalization
of the determinant of the matrix; here, we generalize the quartic discriminant of a $2 \times 2 \times 2$ cube,
which is the generator of the ring of $\SL_2^3$-invariants for $2 \times 2 \times 2$ cubes.

\begin{defn}
	The {\em discriminant} of a Hermitian cube $A = (a,b,c,d)$ is given by
		\begin{equation} \label{eq:discHermcube}
			\disc(A) = (ad - \Tr(b,c))^2 - 4\Tr(b^\sharp,c^\sharp) + 4a\Norm(c) + 4d\Norm(b).
		\end{equation}
	The {\em norm form} $\Norm_\CJ$ is the complete linearization
	of the discriminant form, so it is a symmetric quadrilinear map
		$$\Norm_\CJ: \CJ \times \CJ \times \CJ \times \CJ \ra K$$
	with $\Norm_\CJ(A,A,A,A) = \disc(A)$.  The linear trace form $\Tr_\CJ$ is defined as
		$$\Tr_\CJ(A) := \Norm_\CJ(A,\varepsilon,\varepsilon,\varepsilon)$$
	and there is an alternating bilinear trace form
		$$ \langle A,A' \rangle = ad' - \Tr(b,c') + \Tr(b',c) - a'd.$$
\end{defn}

Since the trace form $\Tr$ for the Jordan algebra $J$ is nondegenerate, and hence
the bilinear trace form $\langle \cdot,\cdot \rangle$ is nondegenerate,
the space $\CJ$ is self-dual.  There is a natural cubic ``adjoint'' map $\flat$ associated to $\CJ$:
	\begin{equation} \label{eq:flatmap} \xymatrix@R=0pt@C=18pt{
	  \flat:& \CJ \ar[r]& \CJ^\vee \ar[r]^{\cong} &  \CJ \\
		&A \ar@{|->}[r] & \Norm_\CJ(A,A,A,\cdot) \ar@{|->}[r] & A^\flat.
	}\end{equation}
This map is the analogue of the adjoint map $\sharp$ for cubic Jordan algebras
(defined in equation \eqref{eq:sharpmap}) and satisfies the properties
	\begin{eqnarray} \label{eq:adjadj}
	\Norm_\CJ(A,A,A,B) = \langle A^\flat, B \rangle & \mathrm{and} & (A^\flat)^\flat = - \disc(A)^2 A.
	\end{eqnarray}
Explicitly, for a Hermitian cube $A = (a,b,c,d)$, the adjoint $A^\flat = (a^\flat, b^\flat, c^\flat, d^\flat)$ is given by the formulas
	\begin{align} a^\flat &= a^2 d - a \Tr(b,c) + 2 \Norm(b) \nonumber \\
		b^\flat &= 2 c \times b^\sharp - 2 a c^\sharp + (a d - \Tr(b,c)) b  \\
		c^\flat &= -2 b \times c^\sharp + 2 d b^\sharp - (a d - \Tr(b,c))c \nonumber \\
		d^\flat &= -a d^2 + d \Tr(b,c) - 2 \Norm(c) \nonumber
	\end{align}
where $\times$ denotes the bilinearization of $\sharp$ defined in \eqref{eq:bilinearsharp}.
As in the cubic case, by an abuse of notation, the map $\flat$ will sometimes
just refer to the first map in \eqref{eq:flatmap} from $\CJ$ to $\CJ^\vee$.

\begin{example} \ \begin{enumerate}[(i)]
	\item
		If $J = K$, the discriminant of $A$ is the usual quartic discriminant
		of a binary cubic form in $\Sym^3K^2$.  That is, if $A = (a,b,c,d)$
		represents the binary cubic $a X^3 + 3 b X^2 Y + 3 c X Y^2 + d Y^3$, then
			$$\disc(A) = a^2 d^2 - 3 b^2 c^2 - 6 a b c d + 4 a c^3 + 4 b^3 d$$
		and the adjoint is the covariant given by the determinant of the Jacobian
		matrix of the cubic and its Hessian, namely
			$$A^\flat =   (2 b^3 - 3 a b c + a^2 d,b^2 c - 2 a c^2 + a b d, 
			- b c^2 + 2 b^2 d - a c d, - 2 c^3 + 3 b c d - a d^2).$$
	\item
		If $J = K \times K \times K$, then the discriminant of
		the Hermitian cube $A = (a,(b_1,b_2,b_3),(c_1,c_2,c_3),d)$ is the usual quartic discriminant
		of a $2 \times 2 \times 2$ cube (see \cite[\S 2.1]{hcl1}):
			\begin{align*}
				\disc(A) &= a^2 d^2 + b_1^2 c_1^2 + b_2^2 c_2^2 + b_3^2 c_3^2 + 4(ac_1c_2c_3 + b_1 b_2 b_3 d) \\
			 		&\quad - 2(ab_1c_1d + ab_2c_2d + ab_3c_3d + b_1 b_2 c_1 c_2 + b_1 b_3 c_1 c_3 + b_2 b_3 c_2 c_3).
			\end{align*}
		The adjoint $A^\flat := (a^\flat,(b_1^\flat,b_2^\flat,b_3^\flat),(c_1^\flat,c_2^\flat,c_3^\flat),d^\flat)$
		is easy to compute, with coordinates
			\begin{align*}
			 	a^\flat &= 2 b_1 b_2 b_3 - a (b_1 c_1 + b_2 c_2 + b_3 c_3) + a^2 d \\ 
			 	d^\flat &= -2 c_1 c_2 c_3 + d (b_1 c_1 + b_2 c_2 + b_3 c_3) - a d^2 \\
			 	b_i^\flat &= (-b_i c_i + b_j c_j + b_k c_k) b_i - 2 a c_j c_k + a d b_i \\
			 	c_i^\flat &= - (- b_i c_i + b_j c_j + b_k c_k) c_i + 2 d b_j b_k - a d c_i
			 \end{align*}
		where $\{i,j,k\} = \{1,2,3\}$.
\end{enumerate} \end{example}

\subsubsection{Rank and linear transformations}

We may also define an analogue of rank for Hermitian cube spaces that agrees with
the natural notion of rank for a tensor space in the simplest cases.

\begin{defn} \label{defn:rankHermitiancubes}
	A nonzero Hermitian cube $A \in \CJ$ has {\em rank one}
	if it is a scalar multiple of a Hermitian cube of the form 
		\begin{equation} \label{eq:rankonecubes}
			(a,b,c,d) = (\Norm(\alpha),\alpha^\sharp \bullet \beta, \beta^\sharp \bullet \alpha, \Norm(\beta))
		\end{equation}
	for any $\alpha, \beta \in J$.
	A Hermitian cube $A \in \CJ$ has {\em rank $\leq 2$} if its adjoint $A^\flat$ is $0$.
	It has {\em rank $\leq 3$} if its discriminant $\disc(A)$ is $0$.  Finally, it has {\em rank four} if its discriminant is nonzero.
\end{defn}

Note that the condition on scalar multiples for rank one cubes is necessary: if $k \in K$
is not in the image of the map $\Norm$, then $(k,0,0,0)$ is not of the form in
\eqref{eq:rankonecubes}, but intuitively we would still like this cube
to have rank one.  With this definition, the rank one (and zero) cubes are,
up to $K$-scaling, given by elements of $J^2$.  More precisely, there
is a map
	\begin{align*}
		J \oplus J &\ra \CJ \\
		(\alpha,\beta) &\longmapsto (\Norm(\alpha),\alpha^\sharp \bullet \beta, \beta^\sharp \bullet \alpha, \Norm(\beta))
	\end{align*}
that descends to a rational map 
\begin{equation} \label{eq:PJtoPCJ}
		\tau: \PP(J \oplus J) \dashrightarrow \PP(\CJ)
\end{equation}
where $\PP(J \oplus J)$ and $\PP(\CJ)$ denotes $K$-lines in the
$K$-vector spaces $J \oplus J$ and $\CJ$, respectively.  Let $X_\CJ$
denote the image of the map $\tau$.  Intuitively, the variety $X_\CJ$
is like the projective line over $J$, with $\tau$ being analogous to
the embedding of the twisted cubic.  (This comparison is not entirely
accurate, of course, since scaling $(\alpha,\beta)$ by elements of $J$
does not usually fix its image under $\tau$.)

Let $\SL_2(J)$ denote the group of discriminant-preserving $K$-linear
transformations of the space $\CJ$.\footnote{This group is denoted by $\mathrm{Inv}(\CJ)$ in \cite{krutelevich}.}
Some examples include
	\begin{center}
		\begin{tabular}{c|c|c} \label{table:SL2J}
			$J$ & $\CJ$ & $\SL_2(J)$ \\
			\hline
			$K$ & $\Sym^3(2)$ & $\SL_2(K)$ \\ 
			$K^{3}$ & $2 \tns 2 \tns 2$ & $\SL_2(K)^3$ \\
			$\HH_3(K)$ & $\wedge^3_0(6)$ & $\Sp_6(K)$ \\
			$\HH_3(K \times K)$ & $\wedge^3(6)$ & $\SL_6(K)$ \\
			$\HH_3(\mathscr{Q})$ & $S^+(32)$ & $\Spin_{12}$ \\
			$\HH_3(\mathscr{O})$ & $56$ & $E_7$
		\end{tabular}
	\end{center}
where the last two may also range over different quaternion and octonion algebras
over $K$ (and thus are associated with different forms of the corresponding groups).

Thus, the space $\CJ$ may be thought of as a representation of
$\SL_2(J)$.  In fact, the variety $X_\CJ$ is the projectivization
of the orbit of the highest weight vector of the representation of $\SL_2(J)$ on $\CJ$,
and is isomorphic to the homogeneous space given by this representation!  That is, the variety
$X_\CJ$ is isomorphic to $\SL_2(J) / P_\CJ$, where $P_\CJ$ is the
parabolic subgroup associated to the representation $\CJ$, and it has
a moduli interpretation as a flag variety.

In fact, the rank of all the elements of $\CJ$ is preserved under the action of $\SL_2(J)$, and over an
algebraically closed field, the group $\SL_2(J)$ acts transitively on the set of rank $r$ elements \cite[Lemma 21 \& Thm 2]{krutelevich}.
In other words, the space $\CJ$ is stratified by rank into orbits of $\SL_2(J)$ when $K$ is algebraically closed.
Thus, for many computations related to elements in $\CJ$, they may just be checked on representatives of $\CJ$ of
the appropriate rank; see, e.g., \cite[eqs.~(58)-(61)]{krutelevich} for some simple choices of representatives.%
\footnote{For example, the definition of ``rank one'' cubes in \cite{krutelevich} differs from Definition \ref{defn:rankHermitiancubes},
but it is easy to check that they are equivalent by this method.}

In the case where $J$ is the $\HH_3(\CC \times \CC)$ and
$X_\CJ$ is the Grassmannian $\Gr(3,6)$, these varieties and some
of the geometric constructions in the sequel are studied in Donagi's
work \cite{donagi-grassmannians}.

Furthermore, in each of the cases in the above table over an algebraically closed field,
the secant variety of $X_\CJ$ is the entire space $\PP(\CJ)$, and the tangent variety of $X_\CJ$ is the quartic hypersurface
$Y_\CJ$ given by the vanishing of the discriminant \cite[Chap.~III]{zak-book}; that is, 
$Y_\CJ$ is made up of the the rank $\leq 3$ elements of $\CJ$ (up to $K$-scaling).
In fact, Zak shows (and it is easy to check on representatives of the appropriate ranks):

\begin{lemma} \label{lem:ptonsecant} 
  Each general point of $\PP(\CJ) \setminus Y_\CJ$ lies on exactly one secant
  line of $X_\CJ$.  Each point of $Y_\CJ \setminus X_\CJ$ lies on exactly one tangent
  line of $X_\CJ$.
\end{lemma}

Therefore, we find that the adjoint map $\flat$ induces a birational map
		$$\beta_J: \PP(\CJ) \dashrightarrow \PP(\CJ)$$
whose reduced base locus is the variety $X_\CJ$.  Under $\beta_J$, the quartic hypersurface $Y_\CJ$
is blown down to $X_\CJ$, as the adjoint of rank $\leq 3$ elements in $\CJ$ have rank $\leq 1$.  By \eqref{eq:adjadj}, applying $\beta_J$ twice is the identity away from $Y_\CJ$.

Furthermore, the adjoint map $\beta_J$ preserves each secant line.  This is easy to check by computation, e.g., by showing that
the adjoint of the sum of two rank one cubes is in their span.  For example, the adjoint of the sum of $A = (1,0,0,0)$ and $B = (\Norm(\alpha),\alpha^\sharp \bullet \beta, \beta^\sharp \bullet \alpha, \Norm(\beta))$ is $\Norm(\beta) A - \Norm(\beta) B$.

\subsection{Triply Hermitian hypercubes} \label{sec:deg2moduli}

As in \S \ref{sec:deg3moduli}, we would like to study the orbits of a class of ``Hermitianized''
representations uniformly.  Namely, for $V$ a two-dimensional $K$-vector space and $J$ a cubic
Jordan algebra, we study the representation of $\GL(V) \times \SL_2(J)$-orbits on
$V \tns \CJ$.  We find that the orbits correspond to isomorphism classes of genus one curves with degree $2$ line bundles, 
along with bundles related to $X_\CJ$.  We consider only nondegenerate elements of the tensor space, which will correspond
to smooth curves.

\begin{defn}
  An element $\phi \in V \tns \CJ$ is called {\em nondegenerate} if
  the induced composition map
		$$\flat \circ \phi: V^\vee \to \CJ \to \CJ^\vee$$
	is injective.
Note that nondegeneracy implies that the elements in the image of $\phi$ do not have rank one, and the
discriminant of all but four points (over $\overline{K}$, up to multiplicity) in the image of $\phi$ is nonzero.

\end{defn}

The following theorem---a more precise version of Theorem \ref{triplehermpar}---states that the orbits of such nondegenerate elements of $V \tns \CJ$ are in correspondence with genus one curves with certain vector bundles:

\begin{thm} \label{thm:hermHC} 
  The nondegenerate $\GL_2(K) \times \SL_2(J)$-orbits of $V \tns \CJ$
  are in bijection with isomorphism classes of nondegenerate quadruples
  $(C,L,\mathcal{F}, \kappa)$, where $C$ is a genus one curve over~$K$; \,$L$ is a degree $2$ line bundle on $C$; \;$\mathcal{F} = (E_i)$
  is the flag of vector bundles $E_i$ given by pulling back the
  universal flag via the map $\kappa: C \to X_\CJ$; \,and $\kappa^*
  \OO_{\PP(\CJ)}(1) \cong L^{\tns 3}$.
\end{thm}

As in \S \ref{sec:deg3moduli}, we will discuss the nondegeneracy condition
on quadruples $(C,L,\mathcal{F},\kappa)$ in the proof.  It is again an open
condition, so the statement of the theorem may be rephrased as giving
a bijection between orbits of $V \tns \CJ$ and the $K$-points of an open
substack of the moduli space of $(C,L,\mathcal{F},\kappa)$ (with the
isomorphism from $\kappa^*\OO_{\PP(\CJ)}(1)$ to  $L^{\tns 3}$).  Again,
in many cases, this nondegeneracy condition will be a simple condition
on the bundles.

The rest of this subsection contains the proof of Theorem \ref{thm:hermHC}.
First, from a nondegenerate element $\phi$ in $V \tns \CJ$, we explain how to
naturally construct the genus one curve and associated data described in Theorem
\ref{thm:hermHC}.  The construction will be invariant under the group
action.  Although many features of this construction are similar to
the one described in \S \ref{sec:deg3moduli}, here the curve will
not necessarily be immersed in $\PP(\CJ)$ but instead have a degree
$2$ map to a line in $\PP(\CJ)$.  In addition, the adjoint map plays a
different role here and is not logically necessary for describing
the genus one curve and other geometric data from
the multilinear object.  For the remainder of this subsection, we fix the cubic
Jordan algebra $J$ and set $X := X_\CJ$ and $Y := Y_\CJ$.

We will work with the Hilbert scheme $\Hilb_2(X)$ of two points on $X$.  Let
$\zeta: \mathcal{Z}^{\univ} \to \Hilb_2(X)$ denote the universal degree $2$ subscheme over $\Hilb_2(X)$, so there is also a natural map $\epsilon: \mathcal{Z}^{\univ} \to X$.
Let $\mathcal{L}^{\univ}$ denote the universal line over $\Hilb_2(X)$ pulled back from
the universal line over the Hilbert scheme $\Hilb_2(\PP(\CJ))$ of two points in $\PP(\CJ)$.
That is, to a zero-dimensional degree $2$ subscheme in $X \subset \PP(\CJ)$,
we associate the unique line passing through it; a nonreduced subscheme gives a point and a tangent direction,
and thus also a unique line.  We then have the diagram
	\begin{equation} \label{eq:ZLuniv} \xymatrix{
		\mathcal{Z}^{\univ} \ar[rd]_-{\zeta} \ar@{^{(}->}[r]^{\iota} & \mathcal{L}^{\univ} \ar[d]^-{\PP^1 \textrm{-bundle}} \ar[r]^-{b} & \PP(\CJ) \\
		& \Hilb_2(X)
	} \end{equation}
where the map $b$ on the right comes from the construction of $\mathcal{L}^{\univ}$.  
A straightforward computation (e.g., as found in \cite[\S 2.1]{vermeire}) gives
	\begin{equation} \label{eq:Luniv}
	\mathcal{L}^{\univ} \cong \PP(\zeta_* \epsilon^* \OO_X(1)),
	\end{equation}
where $\OO_X(1)$ denotes the pullback of $\OO_{\PP(\CJ)}(1)$ to $X \subset \PP(\CJ)$.

We first give the construction of a genus one curve and the appropriate bundles from an orbit of $V \tns \CJ$.

\begin{construction} \label{cnstr:hermHCtoCurve}
Given $\phi \in V \tns \CJ$, we view  $\phi$ as a linear map
in $\Hom(V^\vee,\CJ)$.  Nondegeneracy implies that there is a map
	$$\PP(\phi): \PP(V^\vee) \ra \PP(\CJ)$$
whose image does not intersect $X$.
Lemma \ref{lem:ptonsecant} implies that the general points of $\PP(V^\vee)$ each lie
in exactly one secant line.  The idea is that the curve will be made up
of the ``pivot'' points of these secant lines and thus be a double cover
of $\PP(V^\vee)$.

More precisely, because the image of the map $\PP(\phi): \PP(V^\vee) \to \PP(\CJ)$ is not completely contained in $X$, the map $\PP(\phi)$ lifts uniquely to $\mathcal{L}^\univ$ by the valuative criterion, as $b$ is birational.  We thus have
	\begin{equation} \label{eq:ZLunivP1} \xymatrix{
		& & \PP(V^\vee) \ar[d]^{\PP(\phi)} \ar@{..>}[ld]_{\widetilde{\PP(\phi)}} \\
		\mathcal{Z}^{\univ} \ar[rd]_-{\zeta} \ar@{^{(}->}[r]^{\iota} & \mathcal{L}^{\univ} \ar[d]^-{p} \ar[r]^b & \PP(\CJ) \\
		& \Hilb_2(X)
	} \end{equation}
If $\widetilde{\PP(\phi)}$ factors through a fiber of $p$, then the image of $\PP(\phi)$
would itself be a secant line to $X$, which contradicts the nondegeneracy assumption.
Thus, the composite map $p \circ \widetilde{\PP(\phi)}$ is a finite map; because two
lines in $\PP(\CJ)$ intersect in at most one point, this composite map is actually an
isomorphism onto its image in $\Hilb_2(X)$.

Pulling back the bottom left triangle of \eqref{eq:ZLunivP1}
via $p \circ \widetilde{\PP(\phi)}: \PP(V^\vee) \to \Hilb_2(X)$, we obtain the diagram
	\begin{equation*} \xymatrix{
		C \ar[rd]_{\eta} \ar@{^{(}->}[r]^{\iota} & \Sigma \ar[d]^{p} \\
		& \PP(V^\vee)
	} \end{equation*}
where $\Sigma$ is a ruled surface and $C$ is a degree $2$ cover of $\PP(V^\vee)$.

Recall that $\Hilb_2(X)$ is the blowup of $\Sym^2(X)$ along the diagonal.
The $2$-to-$1$ map $C \to \PP(V^\vee)$ ramifies exactly where $\PP(V^\vee)$ intersects the locus of fat schemes of $\Hilb_2(X)$,
namely the pullback of the diagonal of $\Sym^2(X)$ via the natural birational map $\Hilb_2(X) \dasharrow \Sym^2(X)$.  
This ramification locus is the intersection of $\Im(\PP(\phi)) \subset \PP(\CJ)$ and the
tangent variety $Y$.  Since $Y$ is a quartic hypersurface in $\PP(\CJ)$, the ramification locus is a zero-dimensional subscheme
of degree $4$, and by the nondegeneracy assumption, the four points of ramification over $\overline{K}$ are distinct.  By Riemann-Hurwitz, the curve $C$, possibly after normalization, has genus one.

This construction also produces an explicit equation for the curve $C$, e.g., for $v \in V^\vee$, it is the double cover of $\PP(V^\vee)$ given by $y^2 = \disc(\phi(v)).$
The bundles on the curve $C$ may also be found immediately from the construction.
First, the pullback of $\OO_{\PP(V^\vee)}(1)$ via the degree $2$ map
$\eta$ is the desired degree $2$ line bundle $L$.  The map
\begin{equation*}
  \kappa: C \ra \mathcal{Z}^{\univ} \stackrel{\epsilon}{\ra} X \hookrightarrow \PP(\CJ)
\end{equation*}
and the moduli interpretation of $X$ as a flag variety in
$\PP(\CJ)$ together give a flag of bundles $\mathcal{F} = (E_i)$ on
the curve $C$.

Finally, we determine a relation between the line bundles $L$ and $\kappa^* \OO_{\PP(\CJ)}(1)$ on $C$.  It is easiest to
describe geometrically.  Take a hyperplane $H$ in $\PP(\CJ)$
containing the image of $\PP(\phi)$ but not its (cubic) image under the adjoint map $\beta_J$.  Then for
any point of $C$ intersecting $H$, its conjugate (under the map $\eta$) also lies on $H$, since the
secant line containing these two conjugate points intersects $\Im{\PP(\phi)}$ by the construction of $C$.
Thus, the line bundle $\kappa^* \OO_{\PP(\CJ)}(1)$ is a power of $L$.  

To show that this power is $3$, recall that the adjoint map $\beta_J$ preserves secant lines.  In particular, for a general
point in $\PP(\CJ) \setminus X$, the line spanned by itself and its image under $\beta_J$ is a secant line of $X$.
Now applying the adjoint map $\beta_J$ to $\Im{\PP(\phi)}$ gives a cubic rational curve in $\PP(\CJ)$, i.e., a curve whose intersection with $H$ is degree $3$.  Therefore, for a general choice of $H$, there are exactly three secant lines that contain points of $\Im{\PP(\phi)}$ and are contained in $H$; these give rise to exactly three pairs of conjugate points on the curve $C$ contained in $H$.

As noted earlier, the data of the curve $C$, the line bundle $L$, the flag $\mathcal{F}$, and the map $\kappa$
are all clearly preserved (up to isomorphism) under the action of the group
$\GL(V) \times \SL_2(J)$, since each factor acts by linear
transformations on its respective projective space.
\end{construction}

  We have described the map from $\GL(V) \times \SL_2(J)$-orbits of $V
  \tns \CJ$ to isomorphism classes of quadruples
  $(C,L,\mathcal{F},\kappa)$.  We now describe the reverse map.
  It will be clear that these two constructions are inverse to one another.

\begin{construction} \label{cnstr:curvetohermHC}
 The general idea of the reverse map is as follows: starting with the
  geometric data of $(C,L,\mathcal{F},\kappa)$, we would like to pick
  out a linear $\PP^1$ in $\PP(\CJ)$.  For each degree $2$ divisor on
  the curve in the linear series $\left| L \right|$, we find the
  ``average'' point of the images of its support in $X$, and all such points
  together form the desired $\PP^1$.
     
     More precisely, given a quadruple $(C,L,\mathcal{F},\kappa)$ as in
  the theorem, we have a natural degree $2$ map $\eta: C \to
  \PP(H^0(C,L)^\vee) \cong \PP^1$.  Using the hyperelliptic involution $\iota$ on $C$ given by $\eta$
  and the map $\kappa: C \to X$, we obtain the commutative diagram
  	\begin{equation*} \xymatrix{
  		C \ar[r]^{(\kappa, \kappa \circ \iota)} \ar[d]_{\eta} & X \times X \ar[d]^{S_2-\textrm{quotient}} \\
  		\PP(H^0(C,L)^\vee) \ar[r]_-{h'} & \Sym^2(X).
  	} \end{equation*}
  The map $h'$ may be lifted to a map $h: \PP(H^0(C,L)^\vee) \to \Hilb_2(X)$,
  since the image of $h'$ does not lie completely in the diagonal of $\Sym^2(X)$.  We thus have
  the commutative diagram
  	\begin{equation*} \xymatrix{
  		C \ar[r] \ar[d]_{\eta} & \mathcal{Z}^{\univ} \ar[d]_-{\zeta} \ar[r] & X \times X \ar[d]^{S_2-\textrm{quotient}} \\
  		\PP(H^0(C,L)^\vee) \ar[r]_-{h} & \Hilb_2(X) \ar[r] & \Sym^2(X).
  	} \end{equation*}
  By diagram \eqref{eq:ZLuniv}, recall that $\mathcal{L}^{\univ}$ is a $\PP^1$-bundle
  over $\Hilb_2(X)$.  Define $\Sigma := h^* \mathcal{L}^{\univ}$, 
  so $p: \Sigma \to \PP(H^0(C,L)^\vee)$ gives a ruled surface, specifically a $\PP^1$-bundle
  over $\PP(H^0(C,L)^\vee)$.
  In fact, using the relation between $L$ and $\kappa^*\OO_X(1)$ and \eqref{eq:Luniv},
  along with the projection formula, we see
  that $$\Sigma = \PP(\eta_* \OO_C \tns \OO_{\PP(H^0(C,L^\vee)}(3)).$$  
  We would like to pick out a section $s: \PP(H^0(C,L)^\vee) \to \Sigma$
  such that the image of the composite $b \circ h \circ s$ is linear in $\PP(\CJ)$.
  
  First, we claim there is a unique section $s: \PP(H^0(C,L)^\vee) \to \Sigma$ classifying
  $\eta_* \OO_C \tns \OO(3) \to \OO(1)$.
  That is, a cohomology computation shows that there is an exact sequence
  $$0 \to \OO(3) \to \eta_* \OO_C \tns \OO(3) \to \OO(1) \to 0$$
  on $\PP(H^0(C,L)^\vee)$, and because $\Hom(\OO(3),\OO(1)) = 0$,
  any such map $\eta_* \OO_C \tns \OO(3) \to \OO(1)$ is this one, up to scalars.
 
  Therefore, we obtain a map from $\PP(H^0(C,L)^\vee)$ to $\mathcal{L}^{\univ}$,
  and thus to $\PP(\CJ)$.  The nondegeneracy condition on the geometric data that we require
  is exactly that the image of this map does not intersect $X$; it is clear that it is satisfied
  by the data in Construction \ref{cnstr:hermHCtoCurve} by assumption.

  We now check that this map $\PP(H^0(C,L)^\vee) \to \PP(\CJ)$ is linear,
  i.e., the pullback of $\OO_{\PP(\CJ)}(1)$ to $\PP(H^0(C,L)^\vee)$
  from $\PP(\CJ)$ (via $b \circ h \circ s$) is isomorphic to $\OO_{\PP(H^0(C,L)^\vee)}(1)$.
  Since $\Sigma$ is a ruled surface,
  there exists some $a_1, a_2 \in \ZZ$ such that
  \begin{equation} \label{eq:linearityHerm}
  (b \circ h)^* \OO_{\PP(\CJ)}(1) = \OO_p(a_1) \tns p^* \OO_{\PP(H^0(C,L)^\vee)}(a_2).
  \end{equation}
  It is easy to see that $a_1 = 1$ because the fibers of $\Sigma$ map to lines in $\PP(\CJ)$.
  To compute $a_2$, we pull back \eqref{eq:linearityHerm} to $C$ via $\iota: C \to \Sigma$:
  \begin{align} \label{eq:linearityHerm2}
  \iota^* (b \circ h)^* \OO_{\PP(\CJ)}(1)
	  &= \iota^* \OO_p(1) \tns \eta^* \OO_{\PP(H^0(C,L)^\vee)}(a_2) \\
	  &= \eta^* \OO_{\PP(H^0(C,L)^\vee)}(3) \tns \eta^* \OO_{\PP(H^0(C,L)^\vee)}(a_2), \nonumber
  \end{align}  
  Since the left-hand side of \eqref{eq:linearityHerm2}
  is just $\eta^* \OO_{\PP(H^0(C,L)^\vee)}(3)$ by the
  assumed relation, we must have $a_2 = 0$.
  Thus, we obtain
  $$s^* (b \circ h)^* \OO_{\PP(\CJ)}(1) = s^* (\OO_p(1)) = \OO_{\PP(H^0(C,L)^\vee)}(1),$$
  as desired.
  
\end{construction}

\subsection{Specializations} \label{sec:hermHCcases}

Allowing the cubic Jordan algebra $J$ in Theorem \ref{thm:hermHC} to vary gives many special cases, as highlighted in Table \ref{table:examples}.  For certain choices of $J$, we recover some of the previously considered spaces related to hypercubes and the corresponding parametrization theorems, while for others, we obtain moduli spaces of genus one curves with higher rank vector bundles.

For example, for $J = K \times K \times K$, we recover the case of standard hypercubes from \S \ref{sec:hypercube}.  In this case,
the homogeneous variety $X_\CJ$ is just the Segre embedding of $\PP^1 \times \PP^1 \times \PP^1$.
If we instead let $J = K$, with norm form $\Norm(x) = x^3$ for an element $x \in K$, then the space $\CJ$ coincides with the space of triply symmetric hypercubes studied in \S \ref{sec:3symHC}, and $X_\CJ$ is the twisted cubic in $\PP^3$.  Also, the space of doubly symmetric hypercubes (see \S \ref{sec:2symHC}) may be obtained by taking $J = K \times K$, with the norm form $\Norm(x_1, x_2) = x_1 x_2^2$ for $(x_1,x_2) \in K \times K$ and $X_\CJ$ the image of $\PP^1 \times \PP^1$ embedded in $\PP^5$ via $\mathcal{O}(1,2)$.

We describe some of the new moduli problems below.  In each case, we also describe more carefully the bijections when the algebra $J$ contains split algebras, e.g., matrix algebras instead of more general central simple algebras.  Taking other forms of these algebras then give twisted versions of the geometric data on the genus one curve.

\subsubsection{Doubly skew-symmetrized hypercubes} \label{sec:2skewHC}

A new case arises by choosing the cubic Jordan algebra $J$ to be $K \times \Mat_{2 \times 2}(K)$.  We first describe the structure of $J$ and the space of Hermitian cubes $\CJ$ for this choice of $J$.

The norm of an element $(x, M) \in K \times \Mat_{2 \times 2}(K)$ is $\Norm(x,M) = x \det(M)$, and for Springer's construction, we take the basepoint $e$ to be $(1, \mathrm{Id})$.  Composition in this algebra is component-wise, with the usual Jordan structure on $\Mat_{2 \times 2}(K)$: for elements $(x_1, M_1), (x_2, M_2) \in K \times \Mat_{2 \times 2}(K)$, we have $(x_1, M_1) \bullet (x_2, M_2) = (x_1 x_2, (M_1 \cdot M_2 + M_2 \cdot M_1)/2)$.

We claim that there is a natural isomorphism between $\CJ$ and the space $W := K^2 \tns \wedge^2 K^4$, where the action of $\SL_2(J)$ on $\CJ$ corresponds to the natural action of $\SL_2(K) \times \SL_4(K)$ on $W$.  We represent elements of $W$ as a pair of $4 \times 4$ skew-symmetric matrices.  Let $a, b, c, d, b_{ij}, c_{ij} \in K$ for $1 \leq i, j \leq 2$.  Then this isomorphism sends
$(a, (b, (b_{ij})), (c, (c_{ij})), d) \in \CJ$ to the pair
\begin{equation*}
{\begin{pmatrix}
0 & a & -b_{12} & b_{11} \\
-a & 0 & -b_{22} & b_{21} \\
b_{12} & b_{22} & 0 & c \\
-b_{11} & -b_{21} & -c & 0
\end{pmatrix}}
\qquad \qquad
{\begin{pmatrix}
0 & b & -c_{12} & c_{11} \\
-b & 0 & -c_{22} & c_{21} \\
c_{12} & c_{22} & 0 & d \\
-c_{11} & -c_{21} & -d & 0
\end{pmatrix}}
\end{equation*}
in $W$.  It is a slightly tedious but trivial computation to check that the two group actions align.

The homogeneous variety $X_\CJ$ is $\PP^1 \times \Gr(2,4) \hookrightarrow \PP^1 \times \PP(\wedge^2 K^4) \hookrightarrow \PP(W^\vee)$, where the first inclusion is by the Pl\"{u}cker map and the second is the Segre embedding.  A map from a scheme $T$ to $X$ thus gives a degree $2$ line bundle and a rank $2$ vector bundle on $T$.  Theorem \ref{thm:hermHC} with this choice of $J$ gives the following moduli description for pairs of Hermitian cubes up to equivalence:

\begin{thm}
  Let $V_1$, $V_2$, and $V_3$ be $K$-vector spaces of dimensions $2$, $2$, and $4$,
  respectively.  Then nondegenerate $\GL(V_1) \times \GL(V_2) \times \GL(V_3)$-orbits of
  $V_1 \tns V_2 \tns \wedge^2(V_3)$ are in bijection with isomorphism classes of nondegenerate 
  triples $(C,L_1,L_2,E)$, where $C$ is a genus one curve over $K$, $L_1$ and $L_2$ are non isomorphic
  degree $2$ line bundles on $C$, and $E$ is a rank $2$ vector bundle
  on $C$ with $\det E \cong L_1 \tns L_2$.
\end{thm}

If we record the point $P := L_2 \tns L_1^{-1}$ instead of the line bundle $L_2$, we recover Theorem~\ref{doubleskewpar}.
Also,  we may also replace $J$ with $K \times Q$, for any quaternion algebra $Q$, in which case the homogeneous variety $X$ is a product of $\PP^1$ and a twisted form of the Grassmannian $\Gr(2,4)$.

\subsubsection{Triply skew-symmetrized hypercubes} \label{sec:3skewHC}

Now let $J = \HH_3(K \times K)$, i.e., the space of $3 \times 3$ Hermitian matrices over the quadratic algebra
$K \times K$.  It is easy to check that $J$ is isomorphic to the space $\Mat_{3 \times 3}(K)$ of $3 \times 3$ matrices over $K$,
with composition given by $(M_1 \cdot M_2 + M_2 \cdot M_1)/2$ for $M_1, M_2 \in \Mat_{3 \times 3}(K)$.

Then the space $\CJ$ of Hermitian cubes is isomorphic to $W := \wedge^3 K^6$, with the action of $\SL_2(J)$ matching the natural action of $\SL_6(K)$ on $W$.  We will write down the isomorphism, following \cite[Example 19]{krutelevich}.  Let $a, b_{ij}, c_{ij}, d \in K$ for $1 \leq i, j \leq 3$.  Let $\{e_1, e_2, e_3, f_1, f_2, f_3 \}$ be a basis for $K^6$, and let $e_j^* = e_{j+1} \wedge e_{j+2}$ and $f_j^* = f_{j+1} \wedge f_{j+2}$ in $\wedge^2 K^6$, where the indices are taken cyclically. Then the element $(a,(b_{ij}),(c_{ij}),d)$ is sent to
$$a e_1 \wedge e_2 \wedge e_3 + \sum_{i,j = 1}^3 b_{ij} e_i \wedge f_j^* + \sum_{i,j = 1}^3 c_{ij} f_i \wedge e_j^* + d f_1 \wedge f_2 \wedge f_3$$
in $W = \wedge^3 K^6$.

Here, the homogeneous variety $X_\CJ$ is the Grassmannian $\Gr(3,6)$, which lies in $\PP(W^\vee)$ via the Pl\"{u}cker map.
Specializing Theorem \ref{thm:hermHC} gives the following basis-free version of Theorem \ref{thm:3skewHCpreview}:

\begin{thm} \label{thm:3skewHC}
  Let $V_1$ and $V_2$ be $K$-vector spaces of dimensions $2$ and $6$,
  respectively.  Then nondegenerate $\GL(V_1) \times \SL(V_2)$-orbits of
  $V_1 \tns \wedge^3(V_2)$ are in bijection with isomorphism classes of
  nondegenerate triples $(C,L,E)$, where $C$ is a genus one curve over $K$, $L$ is a
  degree $2$ line bundle on $C$, and $E$ is a rank $3$ vector bundle
  on $C$ with $\det E \cong L^{\tns 3}$.
  \end{thm}

\subsubsection{Some more exceptional representations} \label{sec:excHC}

For $J = \HH_3(\mathscr{Q})$, where $\mathscr{Q}$ denotes the split quaternion algebra over $K$ (isomorphic to $\Mat_{2 \times 2}(K)$), we obtain a more exceptional representation and theorem.

\begin{thm}
	Let $V_1$ and $V_2$ be $K$-vector spaces of dimensions $2$ and $32$, respectively, where $V_2$
	is the half-spin representation of $\Spin_{12}$.  Let $X$ be the homogeneous space for this action in
	$\PP(V_2^\vee)$.
	Then nondegenerate $\GL(V_1) \times \Spin_{12}$-orbits of $V_1 \tns V_2$
	are in bijection with the $K$-points of an open substack
	of the moduli space of nondegenerate tuples $(C,L,\kappa, \psi)$, where $C$ is a genus one curve, $L$ is a degree $2$
	line bundle on $C$, and $\kappa$ is a map from $C$ to $X \subset \PP(V_2^\vee)$, along 
	with an isomorphism $\psi$ from $L^{\tns 3}$ to the pullback of $\OO_{\PP(V_2^\vee)}(1)$ to $C$ via $\kappa$.
\end{thm}

In Table \ref{table:examples}, we referred to this choice of $X_\CJ$ as the projective line
over the cubic algebra $J$; this is due to our interpretation of $X_\CJ$ as the rank one cubes
that are Hermitian over $J$.
In addition, analogous theorems hold when $\mathscr{Q}$ is replaced by a non-split quaternion algebra.  Then, the group $\Spin_{12}$ is replaced by the appropriate twists.

For $J = \HH_3(\mathscr{O})$, where $\mathscr{O}$ denotes the split octonion algebra over $K$, we have
a similar statement.

\begin{thm}
	Let $V_1$ and $V_2$ be $K$-vector spaces of dimensions $2$ and $56$, respectively, where $V_2$ is the miniscule
	representation of the group $E_7$.  Let $X$ be the homogeneous space for this action in $\PP(V_2^\vee)$.  Then
	nondegenerate $\GL(V_1) \times E_7$-orbits of $V_1 \tns V_2$ are in bijection with the $K$-points of an open substack
	of the moduli space of nondegenerate tuples $(C,L,\kappa,\psi)$, where $C$ is a genus one curve, $L$ is a degree $2$
	line bundle on $C$, and $\kappa$ is a map from $C$ to $X \subset \PP(V_2^\vee)$, along 
	with an isomorphism $\psi$ from $L^{\tns 3}$ to the pullback of $\OO_{\PP(V_2^\vee)}(1)$ to $C$ via $\kappa$.
\end{thm}

Again, taking different octonion algebras over $K$ gives similar theorems, where the split $E_7$ is replaced by twisted forms
of $E_7$.


\section{Connections with exceptional Lie groups and Lie algebras} \label{sec:ExcLieAlgs}

In this section, we describe two ways in which the coregular representations we have considered
in this paper are related to exceptional Lie groups and Lie algebras.  
These still mysterious connections give hints 
as to further moduli problems and directions for investigation.

\subsection{Vinberg's \texorpdfstring{$\theta$}{theta}-groups}\label{vinberg}

All of the representations we have considered in this paper are  {\em $\theta$-groups} in the sense
of Vinberg \cite{vinberg}, when viewed as representations of complex Lie groups.
Vinberg's idea was to extend the concept of a Weyl group and a Cartan
subspace to graded Lie algebras.  Then the invariants of the representation correspond exactly
to the invariants of the Cartan subspace under the action of the Weyl group.  Since the Weyl group
here is generated by complex reflections, its ring of invariants is free, so the representations
obtained in this way are coregular.  Moreover, this construction gives a description of the invariants
(including their degrees) in terms of the Lie theory.

Let $\mathfrak{g}$ be a $\ZZ/m\ZZ$-graded (or $\ZZ$-graded) Lie algebra for some integer $m \geq 1$ (or
respectively, $m = \infty$).  Then for $m$ finite, we may write
	$$\mathfrak{g} = \mathfrak{g}_0 + \mathfrak{g}_1 + \cdots + \mathfrak{g}_{m-1}$$
with $[\mathfrak{g}_i,\mathfrak{g}_j] \subset \mathfrak{g}_{i+j}$ for $i,j \in \ZZ/m\ZZ$.
To each such graded Lie algebra, one associates an automorphism $\theta$ (or, for $m$ infinite, a one-parameter
family of such $\theta$) of $\mathfrak{g}$, \eg for $m$ finite, we have $\theta(x) = \omega^k x$ for
$x \in \mathfrak{g}_k$ and $\omega = e^{2\pi i}$.

Given a graded Lie algebra $\mathfrak{g}$, let $G$ be any connected group having $\mathfrak{g}$ as its Lie algebra,
and let $G_0$ be the connected subgroup of $G$ with $\mathfrak{g}_0$ as its Lie algebra.  Then a 
{\em $\theta$-group} corresponding to $\mathfrak{g}$ is the representation of $G_0$ on $\mathfrak{g}_1$.
(The name ``group'' makes sense by thinking of $G_0$ as a subgroup of $\GL(\mathfrak{g}_1)$.)

Vinberg showed that the $G_0$-invarants of $\mathfrak{g}_1$ form a polynomial ring, and in fact,
the elements of $\mathfrak{g}_1$ with the same $G_0$-invariants comprise a finite number of orbits 
over $\CC$ \cite{vinberg}. Moreover, Kac showed that most such representations arise in this way
\cite{kac-nilpotent}.

These observations give a heuristic reason for looking at these particular representations if
we want to find parametrizations of genus one curves with data such as line bundles and 
points on the Jacobian.  The coarse moduli space of such objects is often a weighted projective space,
or a generically finite cover thereof, e.g., in many of the cases in this paper, the orbit spaces over $\CC$ parametrize elliptic
curves in some family determined by the invariants.  In contrast, arithmetic objects such as rings and ideal
classes, whose coarse moduli spaces are just a finite number of points, are often parametrized by
orbits of prehomogeneous vector spaces.

The $\theta$-groups may be read off directly from subdiagrams of Dynkin diagrams or
affine Dynkin diagrams \cite{vinberg,kac-nilpotent}.  For subdiagrams of Dynkin 
diagrams, the $\theta$-groups are all prehomogeneous vector spaces.  All other
irreducible $\theta$-groups are listed in \cite[Table III]{kac-nilpotent}.  
These are given by removing a single node from the affine Dynkin diagram.
We have indicated the affine Dynkin diagram that gives rise to each representation
in Table \ref{table:examples}.  (Note that in Table~\ref{table:examples} we only 
list the semisimple part of the $\theta$-group.) We have $m=2$ for lines 1--12 of Table~\ref{table:examples}, $m=3$ for lines 13--18, $m=4$ for line 19, and $m=5$ for line 20.  Thus the value of $m$ corresponds to the degree of the associated line bundles in the geometric data!

The connection between these moduli problems and Vinberg theory, especially in the special case $m=2$, is investigated in the beautiful work of Thorne~\cite{jthorne-thesis}.  It is an interesting question as to how Thorne's representation-theoretic constructions of families of curves when $m=2$, obtained via Vinberg theory and the deformation theory of simple singularities, are related to our more direct geometric constructions of these families.  

\subsection{The Magic Triangle of Deligne and Gross}

In~\cite{delignegross}, Deligne and Gross observed that many of the remarkable properties of the
adjoint representations of the groups in the exceptional series
\begin{equation}\label{es}
1\subset A_1 \subset A_2 \subset G_2 \subset D_4 \subset F_4 \subset E_6 \subset E_7\subset E_8
\end{equation}
(as observed in~\cite{deligne-exceptionalseries}) persist for certain other natural sequences of representations.
Namely, for each pair $H \subset K$ of distinct subgroups in (\ref{es}), we may consider the centralizer $Z(H,K)$ of $H$ in $K$. In this way, we obtain a triangle of subgroups of $E_8$, as shown in Table~\ref{dgtable}, where the rows are 
indexed by $H=E_7,\ldots,A_1$ from top to bottom, and the columns are indexed by $K=A_1,\ldots,E_7$ from left to right, and where we have 
ignored all semidirect products with, and quotients by, finite groups.
  Deligne and Gross show that each of the groups in this triangle is naturally equipped with a certain
{\it preferred} representation of that group, obtained from the action of $\Z(H,K) \times H$ on ${\rm Lie}(K)$ 
(see~\cite{delignegross} for details).
In Table~\ref{dgtable}, we have included also the dimensions of these preferred representations.
\begin{table}\label{dgtable}
$$\begin{array}{cccccccc}
 & & & & & & & (A_1,2) \\[.1in]
   & & & & & &(\G_m,1) & (A_2,3) \\[.1in]
   & & & & &(\mu_3,1) &(A_1,3) & (G_2,7) \\[.1in]
    & & & &(\mu_2^2,1) & (\G_m^2,2)& (A_1^3,4)& (D_4,8) \\[.1in]
     & & &(\mu_2^2,2) & (A_1,5)& (A_2,8) & (C_3,14) & (F_4,26) \\[.1in]
      & & (\mu_3,1)&(\G_m^2,3) & (A_2,6)& (A_2^2,9)& (A_5,15)& (E_6,27)\\[.1in]
       & (\G_m,2)& (A_1,4)& (A_1^3,8)& (C_3,14)& (A_5,20)& (D_6,32)& (E_7,56) \\[.1in]
\end{array} 
$$
\caption{Deligne and Gross's Magic Triangle of group representations}
\end{table}

Note that the groups in the bottom right $4\times 4$ square in Table~\ref{dgtable} correspond to the
entries of ``Freudenthal's Magic Square'' of Lie algebras \cite{freudenthal-magic}.

If one takes the last row of representations $(Z(H,K),V)$, where $H=A_1$, and considers instead $(Z(H,K)\times H, V\otimes 2)$, one obtains many of the representations that we used in Sections~\ref{sec:HCpreview} and 
\ref{sec:hermHC} to understand degree 2 line bundles on genus one curves.    Similarly, if one takes the second-to-last row of representations $(Z(H,K),V)$, where $H=A_2$, and considers instead $(Z(H,K)\times H,V\otimes 3)$, one obtains many of the representations that we used in Sections~\ref{sec:RCpreview} and \ref{sec:hermRC} to understand degree~3 line bundles on genus one curves.

As with \S\ref{vinberg}, we suspect that much more lies behind this connection.

\appendix

\section{Torsors of elliptic curves and the period-index subgroup} \label{appendix:torsors}

Let $E$ be an elliptic curve over the field $K$, and let $E[n]$ denote the
$n$-torsion of $E$.  Then if $n$ is invertible in $K$, the Kummer
sequence
	\begin{equation*}
		0 \ra E[n] \ra E \stackrel{n}{\ra} E \ra 0
	\end{equation*}
given by multiplication by $n$ induces the sequence of Galois cohomology
	\begin{equation} \label{eq:galcohom}
		0 \ra E(K)/nE(K) \stackrel{a_n}\ra H^1(K,E[n]) \stackrel{b_n}{\ra} H^1(K,E).
	\end{equation}
Elements of the group $H^1(K,E)$, also known as the Weil-Ch\^{a}telet
group $\WC(E/K)$, may be thought of as isomorphism classes of 
torsors for the elliptic curve
$E$, namely genus one curves $C$ over $K$ with a specified isomorphism
between $E$ and the connected component $\Aut^0(C)$ of the
automorphism group scheme of $C$.  Two such torsors are isomorphic if
there exists an isomorphism between the curves that respects the
action of $E$.  This identification of $H^1(K,E)$ with $E$-torsors $C$ is an example
of the phenomenon that $H^1(K,G)$ for any group $G$ parametrizes $G$-torsors.  

On the other hand, the group $H^1(K,G)$ may be identified with
$\mathrm{Gal}(\Kbar/K)$-sets whose automorphism group is isomorphic
to $G$.  By this principle, elements of the group $H^1(K,E[n])$ are in
correspondence with twists of objects with automorphism group $E[n]$.
As explained in \cite{explicitndescentI}, there are many
interpretations for these twists.  For example, the pair $(C,[D])$,
where $C$ is a torsor for $E$ and $[D]$ is a $K$-rational%
\footnote{The divisor class $[D]$ being $K$-rational means that $D$ is
  linearly equivalent to all of its Galois conjugates; the divisor $D$
  itself may not be $K$-rational.}
divisor class on $C$ of degree $n$, is a twist of $(E,[n \cdot O])$
where $O$ is the identity point of $E$.  The group $H^1(K,E[n])$
parametrizes isomorphism classes of such pairs $(C,[D])$, where two
such pairs $(C,[D])$ and $(C',[D'])$ are isomorphic if there is an
isomorphism $\sigma: C \to C'$ respecting the $E$-action such that $\sigma^* D'$ is linearly
equivalent to $D$.  Under this interpretation, the map $b_n: H^1(K,E[n])
\to H^1(K,E)$ from \eqref{eq:galcohom} simply sends the pair $(C,[D])$
to the class $[C]$ of the curve $C$.

Pairs $(C,[D])$ are equivalent to so-called Brauer-Severi diagrams $[C
  \to \mathbb{P}]$, where $\mathbb{P}$ is a $(n-1)$-dimensional
Brauer-Severi variety.  Given a pair $(C,[D])$, the $K$-rationality of the divisor class $[D]$ gives rise to 
a $K$-rational structure on the embedding of $C_{\Kbar} := C \tns_K \Kbar$ into $\PP(H^0(C_{\Kbar},\OO(D_{\Kbar}))^\vee)$. 
The resulting closed immersion $[C \to \PP]$ is the Brauer-Severi diagram 
representing $(C,[D])$ in $H^1(K,E[n])$.  These
Brauer-Severi diagrams are twists of the diagram $[E \to \PP^{n-1}]$,
and they are (up to isomorphism) another way to represent elements of
$H^1(K,E[n])$.

\begin{defn}
  The {\em obstruction map}
	\begin{equation*}
		\Ob: H^1(K,E[n]) \ra \Br(K)
	\end{equation*}
	sends a Brauer-Severi diagram $[C \to \mathbb{P}]$ to the Brauer class of $\mathbb{P}$.
\end{defn}

The obstruction map, defined and studied extensively in \cite{cathy-periodindex}, is not a group
homomorphism in general; although the kernel is not a group, it
contains the identity of $H^1(K,E[n])$ and is closed under inverses.
The kernel of the obstruction map consists of pairs $(C,[D])$ for which there actually exists a
$K$-rational divisor $D$ representing the class $[D]$ (equivalently,
such that $\left| D \right|$ is isomorphic to $\PP^{n-1}_K$).

The obstruction map may also be given in terms of natural
cohomological maps coming from the elliptic curve $E$, using the Heisenberg
group $\Theta_{E,n}$.  The action of $E[n]$ on $E$ by translation extends to a
linear automorphism of the projective space
$\PP(H^0(E,n \cdot O)^\vee) = \PP^{n-1}$, so there exists a map
	$$E[n] \ra \PGL_n,$$
in other words, a projective representation of $E[n]$.  The inverse
image $\Theta_{E,n}$ of $E[n]$ in $\GL_n \to \PGL_n$ is a central
extension of $E[n]$ by $\Gm$:
	\begin{equation} \label{eq:defTheta}
		0 \ra \Gm \ra \Theta_{E,n} \ra E[n] \ra 0
	\end{equation}
with commutator given by the Weil pairing \cite{mumford-AVs}.  As
proved in \cite{explicitndescentI}, the obstruction map is just the
coboundary map
	$$ \Ob: H^1(K,E[n]) \ra H^2(K,\Gm)$$
from taking non-abelian cohomology of the exact sequence \eqref{eq:defTheta}.
Thus, elements of $\ker(\Ob)$ may be identified with
$H^1(K,\Theta_{E,n})$ by the exact sequence of pointed sets
  $$ 0 = H^1(K,\Gm) \ra H^1(K,\Theta_{E,n}) \stackrel{c_n}{\ra} H^1(K,E[n]) \stackrel{\Ob}{\ra} H^2(K,\Gm).$$

The elements of $H^1(C,\Theta_{E,n})$ may be viewed as isomorphism classes
of torsors for $\Theta_{E,n}$, namely pairs $(C,L)$
where $C$ is an $E$-torsor and $L$ is a degree $n$ line bundle on $C$.
Two pairs $(C,L)$ and $(C',L')$ are isomorphic if there is an
isomorphism $\sigma: C \stackrel{\cong}{\ra} C'$ respecting the $E$-action such that $\sigma^*
L' \cong L$.  The action of $E[n]$ on $C$ fixes the degree $n$ line
bundle $L$, and the line bundle $L$ itself has automorphism group
$\Gm$.  This viewpoint agrees with the interpretation of $\ker(\Ob)$
as pairs $(C,[D])$ where $D$ is a $K$-rational degree $n$ divisor on
the $E$-torsor $C$, since the divisor $D$ (up to equivalence) gives
rise to a line bundle $\OO(D)$.

Following \cite{cathy-periodindex}, we define the period-index subgroup of
a genus one curve $C$ with Jacobian $E$.  
The {\em period} of the curve $C$ is the order of the class of $C$ 
as an $E$-torsor, that is, as an element in $H^1(K,E)$.  The {\em index}
of $C$ is the smallest positive integer $n$ such that there is a line bundle
of degree $n$ on $C$ over $K$, in other words, a lift of $[C]$ to
$H^1(K,E[n])$ that lies in $\ker(\Ob)$.

Let the bilinearization of the quadratic map $\Ob$ be denoted $B_{\Ob}$.
Given a curve $C$ of period $n$ and $P \in E(K)$, the element 
$\langle P, C \rangle := B_{\Ob}(a_n(P),\widetilde{[C]}) \in \Br(K)[n]$ does not depend on the choice of the
lift $\widetilde{[C]} \in H^1(K,E[n])$ of $[C]$.

\begin{defn}
	The {\em period-index subgroup} of a genus one curve $C$ with Jacobian $E$ and period $n$ is the group
		\begin{equation*}
			\Jac_C^n(K) = \left\{ P \in E(K) : \langle P,C \rangle = 0 \right\}.
		\end{equation*}
\end{defn}

A point $P \in E(K)$ lies in $\Jac_C^n(K)$ if it corresponds to an actual $K$-rational degree $0$ divisor
(as opposed to only a $K$-rational divisor {\em class}) on the curve $C$.

\begin{prop} \label{prop:periodindexsubgp}
The period-index subgroup $\Jac_C^n(K) \subset \Jac(C)(K)$ for a genus one curve $C$ of index dividing $n$ consists of
points $P \in \Jac(C)(K) = \Pic^0(C)(K)$ such that $P$ is the difference of two line bundles
in $\Pic^n(C)(K)$.
\end{prop}

\begin{proof}
	Given a genus one curve $C$ over $K$ and two degree $n$ line bundles $L$ and $L'$ on $C$, let $P := L' \tns L^{-1}$.
	We check that $P \in \Jac_C^n(K)$:
		\begin{align*}
			B_{\Ob}(a_n(P),\widetilde{C}) &= B_{\Ob}(a_n(P),[(C,L')]) \\
			&= \Ob(a_n(P) + [(C,L')]) - \Ob(a_n(P)) - \Ob([(C,L')]) \\
			& = \Ob([C,L]) - \Ob(a_n(P)) - \Ob([(C,L')]) = 0
		\end{align*}
	as the last three terms are all trivial in $\Br(K)$.	
	
	Conversely, given a degree $n$ line bundle $L$ and a point $P \in \Jac_C^3(K)$, the same computation shows
	that there exists a lift $[(C,L')]$ of $C$ in $\ker(\Ob)$ such that the
	difference of $L'$ and $L$ is $P$.
\end{proof}

Proposition~\ref{prop:periodindexsubgp} explains why, in Theorems \ref{hyperpar}, \ref{doublesympar}, \ref{doubleskewpar}, \ref{thm:RCparam1}, etc., the points $P$, $P'$, $P''$ arising are not general $K$-points of the Jacobian of the curve $C$, but lie in the relevant period-index subgroup.

\small{
\bibliography{allbib}
\bibliographystyle{amsalpha}
}

\end{document}